\documentclass[12pt]{amsart}
\usepackage[usenames,dvipsnames]{xcolor}
\usepackage{amsmath,amssymb,amsthm,tikz,xspace,float,tabularx,graphicx,hyperref}
\usepackage[letterpaper,margin=1in]{geometry}
\usetikzlibrary{decorations.pathreplacing}
\usetikzlibrary{arrows, arrows.meta, shapes, svg.path}

\usepackage{fullpage, amscd, setspace, bm, graphicx, indentfirst, multirow, tikz-cd, enumerate, verbatim, appendix}
\usepackage{adjustbox,amsfonts,array,graphicx,booktabs,tabularx,multirow,multicol,stmaryrd,tabu,enumitem}
\usepackage[utf8]{inputenc}
\usepackage{cite}
\usepackage{hyperref}
\usetikzlibrary{arrows, arrows.meta, shapes, svg.path}

\newtheorem{theorem}{\textbf{Theorem}}[section]
\newtheorem{proposition}[theorem]{\textbf{Proposition}}
\newtheorem{corollary}[theorem]{\textbf{Corollary}}
\newtheorem{lemma}[theorem]{\textbf{Lemma}}

\theoremstyle{definition}
\newtheorem{remark}[theorem]{Remark}
\newtheorem{definition}[theorem]{Definition}
\theoremstyle{remark}
\newtheorem{example}[theorem]{Example}

\tikzstyle{invisiblepath}=[draw=none, line width=1.5pt, spins]
\tikzstyle{path}=[line width=1.5pt, spins]
\tikzstyle{spins}=[every node/.style={spin}]
\tikzstyle{spin}=[circle, draw, fill=white, minimum size=14pt, inner sep=0pt, text=black]

\DeclareMathOperator{\MC}{MC}
\DeclareMathOperator{\stab}{stab}
\newcommand{\bs}{\boldsymbol}

\newcounter{Comment}

\newcommand{\icegrid}[2]{

    \foreach \i in {1,...,#1}{
        \pgfmathtruncatemacro\col{int(#1-\i+1)}
        \draw (2*\i-1,0) -- (2*\i-1,2*#2) node[spin, label=above:$\col$] {$+$};
    }

    \foreach \i in {1,...,#2}{
        \pgfmathtruncatemacro\row{int(#2-\i+1)}
        \draw (0,2*\i-1) -- (2*#1, 2*\i-1) node[spin, label=right:$\row$] {$+$} ;
    }

    \foreach \i in {1,...,#1}{
        \foreach \j in {1,...,#2}{
            \node[spin] at (2*\i-1, 2*\j-2) {$+$};
            \node[spin] at (2*\i-2, 2*\j-1) {$+$};
        }
    }
}

\definecolor{darkgreen}{RGB}{0,180,0}
\colorlet[named]{green}{darkgreen}

\newcommand{\Col}{\text{Col}}
\DeclareMathOperator{\GL}{GL}
\newcommand{\wt}{\text{wt}}

\newcommand{\vertex}[4]{%
\begin{tikzpicture}
\coordinate (a) at (-1, 0);
\coordinate (b) at (0, 1);
\coordinate (c) at (1, 0);
\coordinate (d) at (0, -1);
\coordinate (aa) at (-.75,.5);
\coordinate (cc) at (.75,.5);
\draw (a)--(c);
\draw (b)--(d);
\draw[white, fill=white] (a) circle (.25);
\draw[white, fill=white] (b) circle (.25);
\draw[white, fill=white] (c) circle (.25);
\draw[white, fill=white] (d) circle (.25);
\path[fill=white] (0,0) circle (.25);
\node at (0,0) {$\scriptstyle x,y$};
\node at (a) {$#1$};
\node at (b) {$#2$};
\node at (c) {$#3$};
\node at (d) {$#4$};
\end{tikzpicture}
}

\newcommand{\smallvertex}[4]{%
\begin{tikzpicture}[scale=.8,every node/.style={scale=1.0}]
\coordinate (a) at (-1, 0);
\coordinate (b) at (0, 1);
\coordinate (c) at (1, 0);
\coordinate (d) at (0, -1);
\coordinate (aa) at (-.75,.5);
\coordinate (cc) at (.75,.5);
\draw (a)--(c);
\draw (b)--(d);
\draw[white, fill=white] (a) circle (.25);
\draw[white, fill=white] (b) circle (.25);
\draw[white, fill=white] (c) circle (.25);
\draw[white, fill=white] (d) circle (.25);
\path[fill=white] (0,0) circle (.25);
\node at (0,0) {$\scriptstyle x,y$};
\node at (a) {$#1$};
\node at (b) {$#2$};
\node at (c) {$#3$};
\node at (d) {$#4$};
\end{tikzpicture}
}

\newcommand{\vertexnoxy}[4]{%
\begin{tikzpicture}
\coordinate (a) at (-1, 0);
\coordinate (b) at (0, 1);
\coordinate (c) at (1, 0);
\coordinate (d) at (0, -1);
\coordinate (aa) at (-.75,.5);
\coordinate (cc) at (.75,.5);
\draw (a)--(c);
\draw (b)--(d);
\draw[white, fill=white] (a) circle (.25);
\draw[white, fill=white] (b) circle (.25);
\draw[white, fill=white] (c) circle (.25);
\draw[white, fill=white] (d) circle (.25);
\node at (a) {$#1$};
\node at (b) {$#2$};
\node at (c) {$#3$};
\node at (d) {$#4$};
\end{tikzpicture}
}

\newcommand{\rmatrix}[4]{%
\begin{tikzpicture}[scale=0.8]
\draw (0,0) to [out = 0, in = 180] (2,2);
\draw (0,2) to [out = 0, in = 180] (2,0);
\path[fill=white] (0,0) circle (.35);
\path[fill=white] (0,2) circle (.35);
\path[fill=white] (2,2) circle (.35);
\path[fill=white] (2,0) circle (.35);
\path[fill=white] (1,1) circle (.35);
\node at (0,0) {$#1$};
\node at (0,2) {$#2$};
\node at (2,2) {$#3$};
\node at (2,0) {$#4$};
\node at (1,1) {$z_i,z_j$};
\end{tikzpicture}
}

\newcommand{\ybelhs}[9]{
\begin{tikzpicture}[baseline=(current bounding box.center)]
  \draw (0,1) to [out = 0, in = 180] (2,3) to (4,3);
  \draw (0,3) to [out = 0, in = 180] (2,1) to (4,1);
  \draw (3,0) to (3,4);
  \path[fill=white] (0,1) circle (.3);
  \path[fill=white] (0,3) circle (.3);
  \path[fill=white] (3,4) circle (.4);
  \path[fill=white] (4,3) circle (.3);
  \path[fill=white] (4,1) circle (.3);
  \path[fill=white] (3,0) circle (.4);
  \node at (0,1) {$#1$};
  \node at (0,3) {$#2$};
  \node at (3,4) {$#3$};
  \node at (4,3) {$#4$};
  \node at (4,1) {$#5$};
  \node at (3,0) {$#6$};
\path[fill=white] (1,2) circle (.3);
\node at (1,2) {$\scriptstyle #7$};
\path[fill=white] (3,3) circle (.3);
\node at (3,3) {$\scriptstyle #8$};
\path[fill=white] (3,1) circle (.3);
\node at (3,1) {$\scriptstyle #9$};
\end{tikzpicture}}

\newcommand{\yberhs}[9]{
\begin{tikzpicture}[baseline=(current bounding box.center)]
  \draw (0,1) to (2,1) to [out = 0, in = 180] (4,3);
  \draw (0,3) to (2,3) to [out = 0, in = 180] (4,1);
  \draw (1,0) to (1,4);
  \path[fill=white] (0,1) circle (.3);
  \path[fill=white] (0,3) circle (.3);
  \path[fill=white] (1,4) circle (.3);
  \path[fill=white] (4,3) circle (.3);
  \path[fill=white] (4,1) circle (.3);
  \path[fill=white] (1,0) circle (.3);
  \node at (0,1) {$#1$};
  \node at (0,3) {$#2$};
  \node at (1,4) {$#3$};
  \node at (4,3) {$#4$};
  \node at (4,1) {$#5$};
  \node at (1,0) {$#6$};
\path[fill=white] (3,2) circle (.3);
\node at (3,2) {$\scriptstyle #7$};
\path[fill=white] (1,1) circle (.3);
\node at (1,1) {$\scriptstyle #8$};
\path[fill=white] (1,3) circle (.3);
\node at (1,3) {$\scriptstyle #9$};
\end{tikzpicture}}

\begin{document}
\title{Lattice Models for Double Whittaker Polynomials and Motivic Chern Classes}
\author{Ben Brubaker}
\address{School of Mathematics, University of Minnesota, Minneapolis, MN 55455}
\email{brubaker@umn.edu}
\author{Daniel Bump}
\address{Department of Mathematics, Stanford University, Stanford, CA 94305-2125}
\email{bump@math.stanford.edu}
\author{Andrew Hardt}
\address{Department of Mathematics, University of Illinois Urbana-Champaign, Urbana, IL 61801}
\email{ahardt@illinois.edu}
\author{Hunter Spink}
\address{Department of Mathematics, University of Toronto, 40 St. George St., Toronto, ON, Canada}
\email{hunter.spink@utoronto.ca}

\begin{abstract}
We will describe solvable lattice models whose partition functions
depend on two sets of variables, $x_1,\cdots,x_n$ and $y_1, y_2, \cdots $
that have different connections with the representation theory of
$\GL(n,F)$ where $F$ is a nonarchimedean local field. If the
boundary conditions are chosen in one way, they are
essentially the Motivic Chern classes that were used
very effectively by Aluffi, Mihalcea, Sch\"urmann and Su (AMSS) to study such
problems. In particular, using this specialization we can obtain
deformations $r_{u,v}$ of the Kazhdan-Lusztig R-polynomials that were used
by Bump, Nakasuji and Naruse to study matrix coefficients of
intertwining operators (introduced by Casselman). Thus we are
able see that the recursion formula for the $r_{u,v}$ is a reflection
of the Yang-Baxter equation.  On the other hand, with more general 
boundary conditions, specializing the parameters $y_i\to 0$ we
recover colored lattice models that were previously used by Brubaker, 
Buciumas, Bump and Gustafsson to represent Iwahori Whittaker functions on $GL(n,F)$. Thus
we term the resulting two-variable-set family of functions as ``double Whittaker polynomials.''
\end{abstract}

\maketitle

\section{Introduction}

We introduce a new family of partition functions for solvable lattice models
on the square lattice whose various specializations and special cases result
in polynomials from:
\begin{enumerate}
	\item {\bf Schubert calculus:} new polynomial representatives for (double) motivic Chern classes for the
  Type~A flag variety and closely related Laurent polynomial representatives for localizations of K-theoretic
  stable envelopes \cite{MaulikOkounkov};
  
  \item {\bf $p$-adic group representation theory:} Iwahori Whittaker functions for $p$-adic groups in Type~A \cite{BBL, BBBGIwahori}; 
  
  \item {\bf Hecke algebra structure theory:} deformations of Kazhdan-Lusztig R-polynomials introduced by Bump
		and Nakasuji~{\cite{BumpNakasujiKL}}. (Although the lattice models that we consider are special to Type~A,
		this portion of our results is for general Cartan type.)
\end{enumerate}

While there's a large and growing body of literature dedicated to representing families of orthogonal polynomials
and special functions from Schubert calculus as partition functions of solvable lattice models,
only item (2) in our list has previously been represented by lattice models (see {\cite{BBBGIwahori}}). 
It may thus seem surprising that a lattice model simultaneously generalizing all three above items should exist.
However, there is a long precedent for structures in the
representation theory of $p$-adic groups mirroring objects coming from
equivariant
K-theory~{\cite{LusztigEquivariant,KazhdanLusztigDeligneLanglands}},
since the same representations of the affine Hecke algebra appear in both
contexts. More relevant for us here, Su, Zhao and Zhong~\cite{SuZhaoZhongKStable}
introduced Demazure recursions for stable envelopes, and then
Aluffi, Mihalcea, Sch\"urmann and Su~\cite{AMSSCasselman} gave similar
Demazure recursions for motivic Chern classes, allowing them to 
make a deep connection between motivic Chern classes and Iwahori fixed vectors
of $p$-adic groups. The existence of those Demazure recursions were a source of
inspiration for this paper, as we sought Boltzmann weights for our solvable models that would
give rise to these operators. In Section~\ref{sec:polys}, we define the double polynomials in question in terms
of these generalized Demazure operators, referring to the resulting functions as ``double Whittaker polynomials.''
Once these are defined, we conclude that section by summarizing the results in the remainder of the paper. 

The use of lattice models to study problems in Schubert calculus as in item (1) above is already too expansive
to survey here, but we highlight several close connections. In particular, solvable lattice models 
representing motivic Segre classes appeared in~{\cite{KnutsonZinnJustinMotivic}}.
Both the lattice models and their associated polynomial representatives
differ substantially from those of the present paper. 
See the introduction in Miller and Knutson~\cite{KnutsonMillerGrobner} for a lucid
discussion of the problem of choice of representatives for Schubert and related
polynomials. The representatives in~{\cite{KnutsonZinnJustinMotivic}} are natural
from the point of view of matrix Schubert varieties, but our representatives
combine highly desirable interpolation properties with an economy of terms when
expressed in the monomial basis (see Example~\ref{ex:kzjcomparison}). Further specializations result in other double polynomials
and connections to Schubert calculus which have previously appeared as partition functions
of lattice models in the following sources. For factorial and double Schur functions, see 
{\cite{ABPWSymmetric, BumpMcNamaraNakasujiFactorial, BumpHardtScrimshawF4, NaprienkoSchur}}; for 
variations on $\beta$-double Grothendieck polynomials and double Schubert polynomials, see \cite{FrozenPipes, Wheeler-Zinn-Justin, KnutsonMillerGrobner, BuciumasScrimshawWeber}. We address precise connections
to these specializations in Section~\ref{sec:special}.

As always with solvable lattice models, the underlying mechanism
depends on R-matrices and the Yang-Baxter equation. The R-matrix
produces recursion formulas for the partition functions involving 
Demazure operators. Since the same recursions can be realized for each of the
three items listed at the outset through
other means, one may prove in general that
the lattice model partition functions represent these special classes of polynomials.
The source of the R-matrices and Yang-Baxter equations is understood
to be the quantum superalgebra $U_q (\widehat{\mathfrak{g}\mathfrak{l}}(n|1))$,
though we will not rely on the quantum group for proofs of solvability. Instead, in Section~\ref{sec:ybe} we
prove Yang-Baxter equations through a series of reductions to a lattice model with
a bounded set of colors, which is then ultimately a finite computation. This approach may
prove valuable in other contexts where the quantum interpretation of the Boltzmann weights is not known.

In Section~\ref{sec:partition-function}, we then carry out the explicit determination of our partition functions 
as a recursion in generalized Demazure operators. In later sections, we make connections through these 
operators to the families of special functions alluded to at the outset. The expression of Double Whittaker 
polynomials in terms of lattice models opens the door to new polynomial identities including branching
rules and generalize Cauchy-type identities and new proofs of them with further instances of Yang-Baxter equations. 
We don't emphasize these techniques or pursue those new identities here, having seen many of them in 
special cases (e.g., Theorem~7 of~\cite{BumpMcNamaraNakasujiFactorial}, Sections~7 and~8 of~\cite{FrozenPipes}, Chapter~7 of~\cite{AggarwalBorodinWheelerColored}), but we do address several novel features of the lattice model 
representation in later sections. 

First, in Section~\ref{subsec:vanishing} we consider the vanishing locus for these polynomials. 
The prototype here are results of Knop-Sahi \cite{SahiInterpolation} on a family of interpolation polynomials indexed by compositions.
These polynomials are determined, up to normalization, by their vanishing at certain integral points in the symmetric group orbit
of the composition. There are symmetric and non-symmetric versions of these polynomials; as usual, the latter may be averaged 
over the symmetric group to obtain the symmetric ones. Vanishing results for partition functions of lattice models were addressed 
in Chapter 12 of~\cite{AggarwalBorodinWheelerColored}, in the context of factorial LLT polynomials, where vanishing resulted 
from a lack of non-zero admissible states (Section~1.4 of~\cite{AggarwalBorodinWheelerColored}). These factorial LLT polynomials 
generalize factorial Schur functions whose interpolation properties were considered in \cite{SahiInterpolation}. This is the symmetric 
version of the theory. We explore the non-symmetric version of the theory from the lattice model point of view, where we observe 
vanishing both from the absence of non-zero states and from more surprising non-trivial cancellation among states.

A second application is addressed in Section~\ref{sec:stable}, where we use the lattice model to provide two successive 
generalizations of certain localizations of the K-theoretic stable envelopes of~\cite{OkounkovEnumerative, MaulikOkounkov} 
that we refer to as ``stab polynomials.'' Initially we use the lattice to provide a lifting of these polynomials to two variable 
sets $\boldsymbol{x}$ and $\boldsymbol{y}$. Then in Section~\ref{subsec:interpstab}, we further extend this lifting 
beyond the familiar dominant and anti-dominant Weyl chambers by interpolating our lifted polynomials via an additional 
permutation indexing the Weyl chamber. This is another example where natural families of boundary data for 
lattice models may be used to properly generalize important classes of special functions.

We thank Amol Aggarwal for extended correspondence about the relations between our model
and degenerations of the model appearing in~\cite{AggarwalBorodinWheelerColored}. We also thank
Bogdan Ion for alerting us to the potential connection with Sahi's interpolation polynomials.
This work was partially supported by NSF grant DMS-2401470 (Brubaker)
and NSF RTG grant DMS-1937241 (Hardt).

\section{\label{sec:polys}Type A Double Whittaker Polynomials}

Iwahori Whittaker functions for principal series representations of $p$-adic
groups such as $\operatorname{GL}(n)$ over a nonarchimedean local field can be
computed using Demazure operators. These depend on a single set of variables
$\boldsymbol{x}$, the so-called ``Satake parameters'' classifying the principal series and
taking values in a certain complex torus. But the definition in terms of Demazure
operators admits an extension to two variables, just as in geometry where Schubert
polynomials can be extended almost trivially to double Schubert polynomials.
Such an extension of Whittaker functions does not have an obvious
interpretation for the representation theory of the $p$-adic group, but as we outline in
this section, it
enhances connections with geometry of the flag variety that have been noted
before, for example in~\cite{ReederCompositio,BBL,MihalceaSuWhittaker}.

Let $(\Phi, X(T), \Phi^\vee, X^\vee(T))$ denote the root datum for a split reductive Chevalley group $G$ with maximal torus $T$, let $(W,S)$ be the finite crystallographic Coxeter system of the datum with Weyl group $W$ and distinguished generators $S = \{ s_1, \ldots s_{n-1} \}$. Then to each $s_i$ we define an operator $\tau_i$ on elements $f$ in the polynomial ring $\mathbb{C}(t)[X(T)]$ by
\begin{equation} \tau_i \cdot f(\boldsymbol{x}) = \frac{t-1}{1-\boldsymbol{x}^{\alpha_i}} f(\boldsymbol{x}) + \frac{\boldsymbol{x}^{\alpha_i}-t}{1-\boldsymbol{x}^{\alpha_i}} f(s_i\boldsymbol{x}), 
	\label{eq:tauop}
\end{equation}
where $\alpha_i$ is the corresponding simple coroot in $\Phi^\vee$. 
The operator $\tau_i$ appeared previously as the operator $T_i$ defined in
{\cite{BBBGIwahori}}, equation (30), upon setting $t = v$. 

There is another collection of operators
\begin{equation} \mathfrak{T}_i = (\boldsymbol{x}^{\alpha_i} - 1)^{- 1} (1 - s_i - t (1
   -\boldsymbol{x}^{- \alpha_i} s_i)) \label{eq:demazurewhit} \end{equation}
that are defined in {\cite{BBBGIwahori}} equation (11) or in {\cite{BBL}} with
$t = q^{- 1}$. This operator is more natural in the theory of Whittaker
functions. It is related to $\tau_i$ by
\[ \tau_i =\boldsymbol{x}^{\rho} \mathfrak{T}_i \boldsymbol{x}^{- \rho}, \]
as is easy to check since $\boldsymbol{x}^{\rho} s_i \boldsymbol{x}^{- \rho}
=\boldsymbol{x}^{\alpha_i}$ where ${\rho}$ is the usual Weyl vector. The operators $\mathfrak{T}_i$ are
sometimes referred to as ``Demazure-Whittaker operators'' since they
define recursions for Iwahori Whittaker functions. Indeed a special case of Corollary~3.9 in~\cite{BBBGIwahori}
shows that to any Weyl group element $w$ and dominant weight
$\lambda$, and any path $(s_{i_1}, \ldots, s_{i_k})$ in the Bruhat graph of the Weyl group
from the identity $e$ to $w$, then the polynomial
\begin{equation}  \phi^\lambda_w(\boldsymbol{x}) := \mathfrak{T}_{i_k}^{\epsilon({i_k})} \cdots \mathfrak{T}_{i_1}^{\epsilon({i_1})} \boldsymbol{x}^\lambda 
\label{eq:whit-poly} \end{equation}
computes certain key values of the Iwahori Whittaker function for the standard Iwahori-fixed basis vector in the principal series corresponding to $w$. 
Here the exponents $\epsilon({i_j})$ are given by
\begin{equation}
\epsilon(i_j) = \begin{cases} +1 & \textrm{if $s_{i_j}$ is an ascent in the walk,} \\ -1 & \textrm{if $s_{i_j}$ is a descent in the walk.} \end{cases} \label{eq:epsilon}
\end{equation}
The expression on the right in (\ref{eq:whit-poly}) is well-defined as the $\mathfrak{T}_i$ satisfy braid relations and quadratic relations and are invertible.
In view of this connection to Iwahori Whittaker functions, we may refer to this family of functions given by (\ref{eq:whit-poly}) as ``Whittaker polynomials.''

We would like to define two variable analogues of Whittaker polynomials, often referred to in other examples as ``double polynomials.'' 
We do this by mimicking the construction above, but with slightly different notational conventions for the applications and connections 
that follow and in the full generality of Corollary 3.9 in~\cite{BBBGIwahori} which allows for a pair of Weyl group elements $v, w$ in $W$. 
The base case of our recursion will now become a family of base cases indexed by one of the two Weyl group elements and will depend on the
two sets of variables labeled $\boldsymbol{x}$ and $\boldsymbol{y}$. 
While the Demazure-Whittaker operators apply independent of the Cartan type of the datum, we don't have a type-free method 
of defining the base cases, and hence we'll ultimately restrict our definition of double Whittaker polynomials to type~$A$.  

Given two elements $v, w \in W$, we now let $(s_{i_1}, \ldots, s_{i_k})$ be a directed walk in the weak right Bruhat graph from $v w_0$ to $w$ with steps 
$$v w_0 \longrightarrow v w_0 s_{i_1} \longrightarrow v w_0 s_{i_1} s_{i_2} \longrightarrow \ldots \longrightarrow w. $$ 
To any dominant weight $\lambda$ and any directed walk from $v w_0$ to $v$ as above, we define
\begin{equation}
 \phi^\lambda_{v, w}(\boldsymbol{x}; \boldsymbol{y}) := \tau_{i_1}^{\epsilon(i_1)} \cdots \tau_{i_k}^{\epsilon(i_k)} \phi^\lambda_{v, vw_0}(\boldsymbol{x}; \boldsymbol{y}), \label{eq:twovarrecursion}
\end{equation}
where the operators $\tau_i$ are defined in~(\ref{eq:tauop}) and act on functions $f(\boldsymbol{x}; \boldsymbol{y})$ by treating the variables $\boldsymbol{y}$ as constants.
The exponents $\epsilon(i)$ are again defined as in~(\ref{eq:epsilon}).
Thus we've reduced the general family of polynomials in pairs $(v,w)$ to a single index $v$ and it remains 
to define $\phi^\lambda_{v, vw_0}(\boldsymbol{x}; \boldsymbol{y})$.

To define the initial seeds of the recursion $\phi^\lambda_{v, vw_0}(\boldsymbol{x}; \boldsymbol{y})$, we must restrict to Cartan type $A$. For any dominant weight  $v \in W$, set 
\begin{equation} \phi^\lambda_{v, vw_0}(\boldsymbol{x}; \boldsymbol{y}) := t^{\ell(v)} \prod_{i=1}^n \prod_{j < \lambda_i} \left(1 - \frac{x_i}{y_j}\right). \label{eq:twovarseed}
\end{equation}
Note this generalizes the seed of the recursion for double Schubert polynomials, which may be recovered by setting $v = e$, the identity of $W$, and setting $\lambda = \rho$, up to a monomial clearing denominators in the $y_j$.

\begin{definition}[Type A Double Whittaker Polynomials] To any dominant weight $\lambda$, and any pair of elements $v,w$ in $W$, the associated {\it Type A Double Whittaker Polynomial} is the unique function $\phi^\lambda_{v,w}(\boldsymbol{x}; \boldsymbol{y})$ in $\mathbb{C}(t)[x_1, \ldots, x_n; y_1, \ldots, y_{\lambda_1-1}]$ determined by identities~(\ref{eq:twovarrecursion}) and~(\ref{eq:twovarseed}).
\end{definition}

There are natural candidates for the base cases in other types, particularly classical groups, where two-variable generalizations of a polynomial representative for the ``top'' Schubert class of the 
cohomology of the flag variety have been proposed. Even in the ordinary (single set of variables) case, there are complications with any choice of polynomial representatives 
for Schubert classes and related objects as outlined
in~{FominKirillovFPSAC}. We are content with the definition in type A
for the present work, as we often rely on connections 
through solvable lattice models in type A to make connections between various special functions.

Our main results can now be rephrased in the context of these double Whittaker polynomials:
\begin{itemize}
\item In Sections~\ref{sec:ybe} and~\ref{sec:partition-function}, we show there
exists a family of solvable lattice models whose partition functions give all
		double Whittaker polynomials in Cartan type~$A$.
\item In Sections~\ref{sec:motivic-chern} and~\ref{sec:stable}, we show that similar recursions produce motivic Chern classes and stable envelope polynomials, and can be realized in type $A$ by modified solvable lattice models.
\item In Section~\ref{sec:denom}, we use the structure of the $\tau_i$ in the definition to provide a simplified proof of denominator conjectures for $r$-polynomials made by \cite{BumpNakasujiKL} and proved by other means in \cite{AMSSCasselman}. These results are valid for all Cartan types.
\item In Section~\ref{sec:special}, we describe other connections between double Whittaker polynomials and familiar families of polynomials, including double Schubert and Grothendieck polynomials and a more formal connection to Iwahori Whittaker functions, all in type $A$.
\end{itemize}

\section{\label{sec:ybe}The lattice models}

The class of models that we will consider are \textit{colored}
models similar to the models in~\cite{BorodinWheelerColored, BBBGIwahori}.
With our conventions to be described, states of the model will appear as colored paths
moving downward and rightward through the lattice. This behavior will arise as a consequence
of enumerating certain allowable local configurations in the model. 

The model is based on a rectangular grid,
a finite planar graph consisting of $n$ horizontal lines
called \textit{rows} and $N$ vertical lines called \textit{columns}.
The $nN$ points where the lines intersect are called \textit{vertices}.
The vertices partition the lines into segments called \textit{edges} and each vertex has four adjacent edges.
The \textit{boundary edges} of the grid are those at the
top, bottom, left or right side adjacent to a single
vertex. \textit{Interior edges} are those adjacent to two vertices, all of which reside within the
rectangle formed by the $nN$ vertices.

We next describe the set of {\it admissible states} on such grids by assigning decorations to each edge with
additional restrictions. The decorations, or labels, on each edge will be pictured as colorings of edges, so define
a finite set $\Col=\{c_1,\cdots,c_n\}$ comprised of \textit{colors} $c_i$ with total ordering
$c_1<\cdots<c_n$.  

\begin{remark}
  There is no \emph{a priori} reason why the number of colors should equal the number $n$ of
  rows of the grid. The assumption will be natural in light of eventual features of the lattice models
  we consider in this paper. Special cases of our models with fewer colors than lattice rows appeared in
  Section 8 of~\cite{BBBGIwahori} and in the general setup of~\cite{BorodinWheelerColored, AggarwalBorodinWheelerColored}.
  
\end{remark}

The allowable decorations on a given edge depend on its position in our grid. If the edge is \textit{horizontal},
meaning it lies in a row, then the allowable decorations on the edge are $\Col \cup \{ \oplus \} = \{c_1,\cdots,c_n,\oplus\}$ where $\oplus$ will be 
rendered in examples as the absence of color.\footnote{The strange notation $\oplus$ for the absence of color appears because our model 
may be seen as a generalization of the six-vertex model where each edge may be decorated with a $\oplus$ or $\ominus$. In this guise, our
colored paths are refining ``uncolored'' paths made by tracing out $\ominus$ signs along edges in a six-vertex model. Even here, it is a loose
analogy as our colored models allow for multiple distinct colors on a vertical edge, with no analogous configuration in the six-vertex model.}
 
If the edge is \textit{vertical}, meaning it lies in a column, then the allowable decorations for the edge are the power set of $\Col$ consisting of all subsets of colors. A \textit{decorated lattice} is
a lattice with decorations assigned to each edge in the grid.

\begin{definition} \label{def:admissiblestate} Let $B$ denote a choice of decoration for each boundary edge of a given lattice grid. Then an admissible state $\mathfrak{s}$ associated to $B$ is a decorated lattice grid with boundary edges decorated according to $B$ and internal edges decorated such that each vertex has adjacent edges matching one of the cases appearing in Figure~\ref{fig:whittaker_weights}. The set of all admissible states for choice of boundary decorations $B$ will be denoted $\mathfrak{S}_B$.
\end{definition} 

If the decoration assigned to a horizontal edge is the color $c_i$, or
if the color $c_i$ is an element of the assigned decoration of
the vertical edge, we say the edge \textit{carries}
the color. Note that while vertical edges may carry multiple colors,
they may not carry multiple copies of the same color. Thinking of each
color as representing a different type of particle and thinking of the columns between each row as a one dimensional particle system evolving discretely from row to row, our model may be said to be \textit{fermionic}.
That is we don't allow superposition of particles, or equivalenty that no edge may carry more than one instance of
the same color.

\begin{figure}[h]
\begin{center}
\begin{tabular}{|ccc|}
\hline
\hline
Case 1 & Case 2 & Case 3\\
\hline
\vertex{+}{\Sigma}{+}{\Sigma} &
\vertex{c}{\Sigma}{c}{\Sigma} &
\vertex{+}{\Sigma}{c}{\Sigma^-_c}
\\	
	$(-t)^{\Sigma_{[1,r]}}$ & 
	$t^{\Sigma_{(c,r]}} (x + t^{\Sigma_c} y)$ & 
	$(-t)^{\Sigma_{[1,c)}} t^{\Sigma_{(c,r]}} $ \\
\hline
\hline
Case 4 & Case 5 & Case 6\\
\hline
\vertex{c}{\Sigma}{+}{\Sigma^+_c} &
\vertex{c}{\Sigma}{d}{\Sigma^{+-}_{cd}} &
\vertex{d}{\Sigma}{c}{\Sigma^{+-}_{dc}}
\\	 
	 $(1-t) x (-t)^{\Sigma_{(c,r]}} $ & 
	 $(1-t) x (-t)^{\Sigma_{(c,d)}} t^{\Sigma_{(d,r]}} $ & 
	 $-y (1-t) (-1)^{\Sigma_{(c,d)}} t^{\Sigma_{(d,r]}} $ \\
\hline
\end{tabular}
\end{center}
\caption{Whittaker Boltzmann weights. The colors $c$ and $d$
are chosen so $c<d$. Each edge is labeled by its spintype,
which for the horizontal edges must be a color or $\oplus$.
Each vertical edge is labeled by a set $\Sigma$, the set 
of colors carried by the edge. It is understood that
if a configuration of spins does not appear in the figure, then its
Boltzmann weight is zero. If we specialize the second parameter
$y$ to zero, these are the same as the Boltzmann weights in
\cite{BBBGIwahori}, where it was shown that the partition
functions were values of Iwahori Whittaker functions.}
\label{fig:whittaker_weights}
\end{figure}

\begin{remark}[Colored paths]
	\label{rem:colored_paths}
	 Admissible states for our models are rendered as \textit{colored paths} through the lattice.
	Indeed if we paint every edge carrying a color $c$
	with that color, the admissible states consist of unbroken chains of colored edges
	forming \textit{paths} through the configuration.
	These always move down and to the right, starting at the
	top or left boundary, and exiting the grid on the right or
	bottom. (In this paper we will use boundary conditions that
	force the colored paths to start on the top and exit on
	the right.) At every vertex, if a colored path comes in
	from the top or left, it must exit to the right or bottom.
	If all four edged adjacent to the vertex carry the color,
	this description does not tell us whether the path coming
	in from the top exits to the right or bottom, but we can
	decide this arbitrarily, or for definiteness we can say
	it exits to the bottom, so two colored paths of the same
	color cross at this vertex. A characteristic example 
	is given in Figure~\ref{fig:characteristicstate}.
\end{remark}

Our next task will be to describe the \textit{Boltzmann weight}
$\beta(\mathfrak{s})$ of a state $\mathfrak{s}$. This is
the product of Boltzmann weights $\beta_v(\mathfrak{s})$ for each vertex $v$ in the lattice. The
Boltzmann weights at each vertex are functions of a global parameter $t$ and local variables $x$ and $y$
associated to the particular row and column, respectively, intersecting at the vertex $v$ in question.
Then the Boltzmann weights are determined by the labels on their adjacent edges and are given in 
Figure~\ref{fig:whittaker_weights}. 

To explain the weights at each vertex, we introduce a bit more notation.
If $\Sigma$ is a set of colors not containing $c$, let 
$\Sigma^+_c := \Sigma\sqcup \{c\}$. Similarly, if $\Sigma$ does contain $c$, 
then let $\Sigma^-_c := \Sigma\setminus \{c\}$. If $c\ne d$, let 
$\Sigma^{+-}_{cd} := (\Sigma^+_c)^-_d$.
If $c\leqslant d$ are colors, let $[c,d]$ be the set of
colors $e$ such that $c\leqslant e\leqslant d$. We will also
sometimes abuse notation and write $[i,j]$ instead of $[c_i,c_j]$
if $1\leqslant i\leqslant j\leqslant n$.
If $S\subseteq [1,r]$, let $\Sigma_S := \Sigma\cap S$, and if $S\subseteq [0,n]$, 
let $u_S := \prod_{i\in S} u_i$. For notational simplicity, let 
$\Sigma_c := \Sigma_{\{c\}}$, and let $a^S := a^{|S|}$ for any quantity~$a$.

As noted above, our lattice models will consist of $n$ rows (labeled $1,\cdots,n$ from top
to bottom) and $N$ columns, labeled from $1$
to $N$, from right to left. We assign local variables $x_i$ and $y_j$ to each vertex in row $i$ and column $j$. 
To any choice of boundary labels $B$, we may now define the partition function $Z_B$ of the lattice model as follows:
\begin{equation} 
\label{generic-partition-function}
Z_B(x_1, \ldots, x_n; y_1, \ldots, y_N) = 
\sum_{\mathfrak{s} \in \mathfrak{S}_B} \prod_{v \in \mathfrak{s}} \beta_v(\mathfrak{s}). 
\end{equation}

We will provide closed form expressions for $Z_B$ for a large family of boundary conditions $B$ in the next section. The key ingredient in
these arguments is the solvability of the lattice model. Here, solvability means that our Boltzmann weights will satisfy Yang-Baxter equations leading
to recursive relations for the partition function.

\subsection{Solvability of Boltzmann weights}

The Boltzmann weights defined in Figure~\ref{fig:whittaker_weights} are
{\em solvable}, meaning that they satisfy so-called Yang-Baxter equations.
The Yang-Baxter equations may be understood as identities of partition
functions for the two lattices on three vertices shown in Figure~\ref{fig:ybe}.
These diagrams consist of two types of vertices drawn with two different
orientations in Figure~\ref{fig:ybe}. The vertices labeled with $x_i, y$ and $x_j, y$
are understood as of the same type as those appearing in the square lattice, and
viewed as having their Boltzmann weights given. The vertices labeled with
$x_i, x_j$ and involving the crossing of a pair of horizontal rows are viewed as
having Boltzmann weights to be solved for so that the partition functions in Figure~\ref{fig:ybe}
are equal for every choice of boundary edge labels $a,b,c,d,e,f$. We sometimes refer to these
latter vertices and weights as \textit{R-vertices} and R-weights.
The vertices labeled $x_i,y$ and $x_j,y$ will be called \textit{T-vertices}.

\begin{definition} We say that a set of Boltzmann weights $T(x,y)$ is (row) solvable if, given row parameters $x_i$
and $x_j$, there exist a set of Boltzmann weights for the $R$-vertex at $x_i, x_j$ such that, for any choice
of boundary edge labels $a,b,c,d,e,f$, the partition functions in Figure~\ref{fig:ybe} made with weights $T(x_i, y)$, $T(x_j, y)$ and
$R(x_i, x_j)$ are equal.
\end{definition}

Our main result in this section is that the $R$-vertex weights in Figure~\ref{fig:rmatrix} provide such an equality, when
the Boltzmann weights of the square lattice are taken from Figure~\ref{fig:whittaker_weights}. Thus the
Whittaker weights are row solvable. We'll also have a second set of row-solvable weights, called ``K-theoretic weights''
which are variants of the Whittaker weights.

\begin{remark}
	The Yang-Baxter equations in Theorem~\ref{thm:ybe} can also be
	deduced from those of~\cite{AggarwalBorodinWheelerColored} by a somewhat intricate
	sequence of transformations. We thank Amol Aggarwal for working this out.
	Our proof is different, based on a reduction to the case of a
	small number of colors, which can then be checked by computer.
	The result in~\cite{AggarwalBorodinWheelerColored} is very general, but since both its proof
	is difficult and proving its relation to our models through a series of limits and specializations is
	challenging, it is useful to give an alternative
	proof in the case we need.
\end{remark}

\begin{figure}[h]
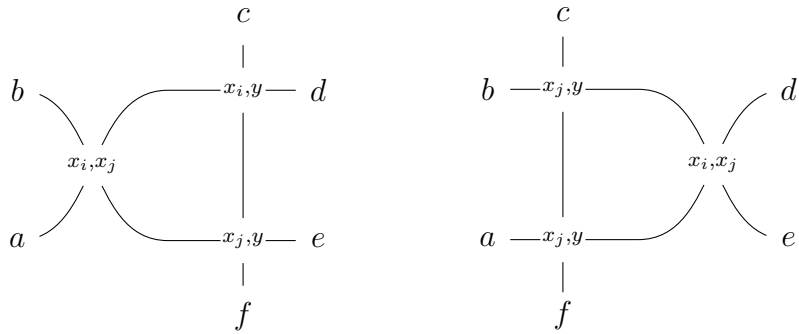

\[\ybelhs{a}{b}{c}{d}{e}{f}{x_i,x_j}{x_i,y}{x_j,y}
\qquad\qquad
\yberhs{a}{b}{c}{d}{e}{f}{x_i,x_j}{x_j,y}{x_j,y}
\]
\caption{The Yang-Baxter equation is the equality of partition functions on the lattices shown above.
The boundary conditions are fixed as shown, with labels $a,b,c,d,e,f$ matching on both sides, and the partition functions are then sums
of Boltzmann weights of the states over possible labelings of the three internal edges on each side.}
\label{fig:ybe}
\end{figure}

In Figure~\ref{fig:ybe} we will refer to the boundary edges
$a,b,c$ as \textit{input edges} and the boundary edges $d,e,f$
as \textit{output edges}. 

\begin{lemma}
	\label{lem:ybe_paths}
	In an admissible state of either the left or right 
	Yang-Baxter configurations (Figure~\ref{fig:ybe}) if an
	input edge carries a color $c$, then an output edge
	must carry the same color. 
\end{lemma}

\begin{proof}
	This follows by considering the colored paths 
	(Remark~\ref{rem:colored_paths}). The remark is easily
	extended to configurations involving an R-vertex.
	From the Boltzmann weights in Figures~\ref{fig:whittaker_weights}
	or~\ref{fig:modified_weights} and Figure~\ref{fig:rmatrix}
	that every colored path begins at an input edge and ends 
	at an output edge. Thus if an input edge carries the color, 
	some output edge must carry the same color.
\end{proof}

Either the left- or the right-hand figure can be regarded as a miniature system
with its own partition function. This means that the spins of the six boundary
edges, labeled $a$, $b$, $c$, $d$, $e$, $f$ are fixed, with 
$c$ and $f$ of vertical type, the other four of horizontal type. 
The spins of the three interior edges are summed over all possibilities. 
Then the Yang-Baxter equation asserts that these two configurations
have the same value. We make this assertion for both the
``Whittaker weights'' in Figure~\ref{fig:whittaker_weights} and
the ``K-theoretic weights'' in Figure~\ref{fig:modified_weights}.

\begin{figure}[h]
\[
\begin{array}{|ccccc|}
\hline
\rmatrix{+}{+}{+}{+} &
\rmatrix{c}{c}{c}{c} &
\rmatrix{c}{+}{+}{c} &
\rmatrix{+}{c}{c}{+} &
\rmatrix{+}{c}{+}{c}
\\
\hline
x_j - tx_i & x_i - tx_j & (1-t)x_i & (1-t)x_j & t(x_i-x_j) \\
\hline     
\rmatrix{c}{+}{c}{+} &
\rmatrix{c}{d}{d}{c} &
\rmatrix{c}{d}{c}{d} &
\rmatrix{d}{c}{c}{d} &
\rmatrix{d}{c}{d}{c}
\\
x_i-x_j & (1-t)x_i & t(x_i-x_j) & (1-t)x_j & x_i-x_j\\
\hline
\end{array}
\]
\caption{The R-matrix. Here $x_i$ and $x_j$ are the row parameters. 
It is assumed that $c$ and $d$ are colors with $c<d$.  
All edges are of horizontal type. This R-matrix
works for the Boltzmann weights in Figure~\ref{fig:whittaker_weights}
and Figure~\ref{fig:modified_weights}.} 
\label{fig:rmatrix}
\end{figure}

\begin{theorem}
	\label{thm:ybe}
	Using the vertex weights from either Figure~\ref{fig:whittaker_weights} or Figure~\ref{fig:modified_weights} and the 
	R-matrix from Figure~\ref{fig:rmatrix}, the Yang-Baxter equation is 
	satisfied. 
\end{theorem}	

\begin{figure}[h]
\begin{center}
\begin{tabular}{|ccc|}
\hline
\hline
Case 1 & Case 2 & Case 3\\
\hline
\vertex{+}{\Sigma}{+}{\Sigma} &
\vertex{c}{\Sigma}{c}{\Sigma} &
\vertex{+}{\Sigma}{c}{\Sigma^-_c}
\\[4pt]
$(-t)^{\Sigma_{[1,r]}}$ & 
$t^{\Sigma_{(c,r]}} (-xy^{-1} + t^{\Sigma_c})$ &
$(-t)^{\Sigma_{[1,c)}} t^{\Sigma_{(c,r]}}$ \\[4pt]
\hline
\hline
Case 4 & Case 5 & Case 6\\
\hline
\vertex{c}{\Sigma}{+}{\Sigma^+_c} &
\vertex{c}{\Sigma}{d}{\Sigma^{+-}_{cd}} &
\vertex{d}{\Sigma}{c}{\Sigma^{+-}_{dc}}
\\    
$-(1-t) x y^{-1} (-t)^{\Sigma_{(c,r]}}$ &
$-(1-t) x y^{-1} (-t)^{\Sigma_{(c,d)}} t^{\Sigma_{(d,r]}}$ &
$-(1-t) (-1)^{\Sigma_{(c,d)}} t^{\Sigma_{(d,r]}}$ \\
\hline
\end{tabular}
\end{center}
\caption{K-theoretic weights at a vertex labeled $x,y$. Here 
$c$ and $d$ are colors with $c<d$, $x$ is the ``row parameter,'' and $y$ is 
the ``column parameter''. These weights are obtained from the Whittaker weights
by dividing each weight by $y$ if there is a color to the left of the vertex,
and also replacing $x$ by~$-x$. This procedure does not affect the R-matrix
in Figure~\ref{fig:rmatrix} because the values are homogenous in the
$x_i$, hence changing $x_i\to -x_i$ multiplies both sides of the Yang-Baxter
equation by $-1$.}
\label{fig:modified_weights}
\end{figure}

\begin{proof}
  For definiteness, we will prove this for the Whittaker weights in
  Figure~\ref{fig:whittaker_weights}. But first let us explain why this
  implies the Yang-Baxter equation for the slightly different weights in
  Figure~\ref{fig:modified_weights}.
  
  The weights in Figure~\ref{fig:modified_weights} are obtained from those in
  Figure~\ref{fig:whittaker_weights} by dividing by $y$ if the edge to the
  left of the vertex carries a color (as opposed to the value $+$) and then replacing
  $x$ with $-x$. Let us first address the effect of division by $y$. Now the
  number of vertices that have a color to the left of them may be inferred
  from the boundary edges. Indeed, it is equal to the number $N_1$ of edges in
  $\{a, b\}$ that carry a color. This may be seen by considering the colored
  paths (as in Remark~\ref{rem:colored_paths} and Lemma~\ref{lem:ybe_paths}).
  Therefore the number is the same for both figures, so the value of either
  figure in Figure~\ref{fig:ybe} using the Figure~\ref{fig:modified_weights}
  weights is multipled by $y^{N_1}$.
  
  To address the substitution $x \mapsto -x$, note that the proposed $R$-matrix solution
  in Figure~\ref{fig:rmatrix} is homogeneous of degree one in the $x_i$ and hence the
  substitution introduces an overall minus sign to the $R$-matrix entries, which appears on
  both sides of the Yang-Baxter equation and hence may be dispensed with. 
  From these observations it is clear that the
  Yang-Baxter equation for the Figure~\ref{fig:modified_weights} will follow
  if we prove them for the Whittaker weights in
  Figure~\ref{fig:whittaker_weights}.
  
  We will need to distinguish between the boundary edges in
  Figure~\ref{fig:ybe} and their values. Therefore we will denote by
  $\boldsymbol{a}, \boldsymbol{b}, \boldsymbol{c}, \boldsymbol{d}, \boldsymbol{e}, \boldsymbol{f}$ the
  six boundary edges, and by $a, b, c, d, e, f$ their spins, which are either
  a color or $+$ for the edges $\boldsymbol{a}, \boldsymbol{b}, \boldsymbol{d},
  \boldsymbol{e}$, and a set of colors for edges $\boldsymbol{c}$ and $\boldsymbol{f}$.
  
  The first point is that for $\leqslant 4$ colors, the Yang-Baxter equation
	may be checked by hand. A Sage program to check this may be found at~\cite{whitfact4}.
  The rest of this proof will rely on the fact that for $\leqslant 4$ colors
  the Yang-Baxter equation has been checked.
  
  We introduce the notion of a \textit{reducible} set of boundary
  conditions. Suppose that either $c$ or $f$ carries some color $g$ that is
  not among the four $a$, $b$, $d$, $e$. If the system has any admissible
  state $\mathfrak{s}$, then the colored path with color $g$ must run directly
  from $c$ to $f$. Therefore both boundary edge spins $c$ and $f$ carry the
  color $g$. Such a system will be called \textit{reducible}.
  
  Any system that is not reducible can involve at most four colors and these
  are checked by the computer program. Therefore we look more closely at the
  reducible case. Assume for the rest of the proof that the system with
  boundary conditions $a, b, c, d, e, f$ is reducible. Let $c'$ and $f'$ be
  the vertical spin types obtained by deleting the color $g$ from $c$ and~$f$.
  Then we have another system with boundary conditions $a, b, c', d, e, f'$.
  
  We will write the Yang-Baxter equation in the form
  \[ \mathrm{LHS} (a, b, c, d, e, f) = \mathrm{RHS} (a, b, c, d, e, f), \]
  where the notation refers to the partition function of the left-hand side
  and the right-hand side in Figure~\ref{fig:ybe}, using the weights in
  Figure~\ref{fig:whittaker_weights}. By induction on the number of colors we
  have
  \[ \mathrm{LHS} (a, b, c', d, e, f') = \mathrm{RHS} (a, b, c', d, e, f') .
  \]
  The result will therefore follow if we prove that
  \begin{equation}
    \label{eq:hsratios} \frac{\mathrm{LHS} (a, b, c, d, e, f)}{\mathrm{LHS}
    (a, b, c', d, e, f')} = \frac{\mathrm{RHS} (a, b, c, d, e,
    f)}{\mathrm{RHS} (a, b, c', d, e, f')} .
  \end{equation}
  We will assume that the boundary data describe a system with legal states,
  so that both sides are nonzero. There is a bijection between the states of
  the system $\operatorname{LHS} (a, b, c, d, e, f)$ and the states of $\operatorname{LHS} (a,
  b, c', d, e, f')$, namely if we erase the path with color $g$ from a state
  $\mathfrak{s}$ of $\operatorname{LHS} (a, b, c, d, e, f)$ we obtain a system
  $\mathfrak{s}'$ of $\operatorname{LHS} (a, b, c', d, e, f')$. There is a similar
	bijection for the $\operatorname{RHS}$ states.
  
  The edges $\boldsymbol{a}, \boldsymbol{b}, \boldsymbol{d}, \boldsymbol{e}$ can only account
  for at most four colors. Let $d_1, d_2, d_3, d_4$ be four colors such that
  $a, b, d, e \in \{d_1, d_2, d_3, d_4 \}$. We will also denote $d_0 = \oplus$
  for the uncolored horizontal spin type. Let us fix a state $\mathfrak{s}$ of
  $\mathrm{LHS} (a, b, c, d, e, f)$. We will define for each of the vertices
  $v = x_i, y$ and $x_j, y$ a {\textit{crossing number}} $\chi_v (d_i, g)$ which
  is to be interpreted as the multiplicity with which the line of color $d_i$
  crosses the line of color $g$. We will relabel the colors in
  Figure~\ref{fig:whittaker_weights} and define the crossing numbers case-by
  case as follows.
  
  \
  
  {\textbf{Case 1.}} In this case $\chi_v (d_0, g) = 1$ where we recall
  $d_0 = \oplus$. We have $\chi_v (d_k, g) = 0$ for all other~$k$.
  
  \
  
  {\textbf{Case 2.}} The color $c = d_i$ for some $i > 0$. Then $\chi_v
  (d_i, g) = 1$ and all other $\chi_v (d_k, g) = 0$ for all other $k$.
  
  \
  
  {\textbf{Case 3.}} The color $c = d_i$ for some $i > 0$. If $d_i < g$
  then $\chi_v (d_i, g) = 1$ while $\chi_v (d_k, g) = 0$ for all $k \neq i$.
  On the other hand if $g < d_i$ then $\chi_v (d_0, g) = 1$ while $\chi_v
  (d_k, g) = 0$ for all $k \neq 0$.
  
  \
  
  {\textbf{Case 4.}} The color $c = d_i$ for some $i > 0$. If $g < d_i$
  then $\chi_v (d_i, g) = 0$ while $\chi_v (d_k, g) = 0$ for all $k \neq i$.
  If $d_i < g$, then $\chi_v (d_0, g) = 1$ while $\chi_v (d_k, g) = 0$ for all
  $k \neq 0$..
  
  \
  
  {\textbf{Case 5.}} The colors $c = d_i$ and $d = d_j$ for some $d_i <
  d_j$. There are now three cases. If $g < d_i$ then $\chi_v (d_i, g) = 1$
  while all other $\chi_v (d_k, g) = 0$. If $d_i < g < d_j$ then $\chi_v (d_0,
  g) = 1$ while all other $\chi_v (d_k, g) = 0$. Finally, if $d_j < g$ then
  $\chi_v (d_j, g) = 1$ while all other $\chi_v (d_k, g) = 0$.
  
  \
  
  {\textbf{Case 6.}} Now $c = d_j$ and $d = d_i$ for some $d_j < d_i$.
  There are three cases. If $g < d_j$ then $\chi_v (g, d_j) = 1$ while all
  other $\chi_v (g, d_k) = 0$. If $d_j < g < d_i$ then
  \[ \chi_v (d_i, g) = \chi_v (d_j, g) = 1, \qquad \chi_v (d_0, g) = - 1, \]
  and all other $\chi_v (g, d_k) = 0$. Finally, if $d_i < g$ then $\chi_k
  (d_i, g) = 1$ and all other $\chi_v (g, d_k) = 0$.
  
  \
  
  In every case, we have
  \begin{equation}
    \label{eq:chisumone} \sum_i \chi_v (d_i, g) = 1,
  \end{equation}
  meaning intuitively, that exactly one $d_i$-path crosses the $g$-path at the
  vertex. Indeed in almost all cases, a unique $\chi_v (d_i, g)$ is nonzero,
  and that crossing number is $1$. The exceptional case is Case~6 with $d_j <
  g < d_i$, in which case both paths of colors $d_i$ and $d_j$ cross, but also
  there is a crossing of $d_0 = \oplus$ type, with multiplicity $- 1$.
  
\begin{remark} Roughly, the crossing number $\chi_v(d_i,g)$ is the
	number of times a line of color $d_i$ beginning at $\boldsymbol{a}$, $\boldsymbol{b}$ or $\boldsymbol{c}$
	and ending at $\boldsymbol{d}$, $\boldsymbol{e}$ or $\boldsymbol{f}$ crosses the colored line of color
	$g$ running from $\boldsymbol{c}$ to $\boldsymbol{f}$. However if the line of color $d_i$
	begins at $\boldsymbol{c}$ or ends at $\boldsymbol{f}$, there is a nuance as to whether
	the line should be considered to cross the line of color $g$
	or not. For the moment, $\chi_v(d_i,g)$ can be taken to be
	an \textit{ad hoc} definition, which will be clarified in
	the proof of Lemma~\ref{lem:crossings}.
\end{remark}

  Now let $\mathfrak{s}$ be a state of either $\operatorname{LHS} (a, b, c, d, e, f)$
  or $\operatorname{RHS} (a, b, c, d, e, f)$. Let $\mathfrak{s}'$ be the corresponding
  state of $\operatorname{LHS} (a, b, c', d, e, f')$ or $\operatorname{RHS} (a, b, c', d, e,
  f')$, obtained by removing the path of color $g$. In
  Figure~\ref{fig:whittaker_weights}, the set $\Sigma$ always contains the
  color $g$. Let $\Sigma'$ be the set $\Sigma^-_g = \Sigma - \{ g \}$ obtained by removing
  it, and similarly for any set containing $g$.
  
  An examination of Figure~\ref{fig:whittaker_weights} shows that in every
  case (with $v = x_i, y$ one of the T-vertices in Figure~\ref{fig:ybe}) we
  will prove
  \begin{equation}
    \label{eq:bwratios} \frac{\beta_v (\mathfrak{s})}{\beta_v (\mathfrak{s}')}
    = (- t)^{\chi_v (d_0, g)} \prod_{i > 0} \left\{\begin{array}{ll}
      t^{\chi_v (d_i, g)} & \text{if $d_i < g$,}\\
      1 & \text{if $g < d_i$.}
    \end{array}\right.
  \end{equation}

  This may be proved on a case-by-case basis. Let $\Sigma$ be a vertical
  spintype carrying the color $g$, and let $\Sigma' = \Sigma - \{ g \}$. We
  will check Cases~1 and~6, and leave the remaining cases to the reader. To
  prove (\ref{eq:bwratios}) in Case~1, $\mathfrak{s}'$ eliminates a $d_0,
  g$-crossing and indeed the ratio of Boltzmann weights is $- t$, because
  $\Sigma'_{[1, r]} = \Sigma' = \Sigma - \{ g \} = \Sigma_{[1, r]} - \{ g \}$
  we have
  \[ \frac{\beta_v (\mathfrak{s})}{\beta_v (\mathfrak{s}')} = \frac{(-
     t)^{\Sigma_{[1, r]}}}{(- t)^{\Sigma_{[1, r]}'}} = - t, \]
  confirming (\ref{eq:bwratios}) in this case.
  
  The other cases are all similar, but we consider Case~6. In this case with
  $c = d_j$ and $d = d_i$ (so $d_j < d_i$) we have
  \[ \beta_v (\mathfrak{s}) = - y (1 - t) (- 1)^{\Sigma_{(d_j, d_i)}}
     t^{\Sigma_{(d_i, r]}} . \]
  If $g < d_j$ then $\Sigma'_{(d_j, d_i)} = \Sigma_{(d_j, d_i)}$ and
  $\Sigma'_{(d_i, r]} = \Sigma_{(d_i, r]}$ so
  \[ \frac{\beta_v (\mathfrak{s})}{\beta_v (\mathfrak{s}')} = 1. \]
  This is consistent with (\ref{eq:bwratios}) given $\chi_v (d_j, g) = 1$ but
  $g < d_j$. On the other hand if $d_j < g < d_i$ then \ $\Sigma'_{(d_j, d_i)}
  = \Sigma_{(d_j, d_i)} - \{ g \}$ and $\Sigma'_{(d_i, r]} = \Sigma_{(d_i,
  r]}$ so with
  \[ \frac{\beta_v (\mathfrak{s})}{\beta_v (\mathfrak{s}')} = - 1. \]
  This is consistent with (\ref{eq:bwratios}) with $\chi_v (d_0, g) = - 1$,
  $\chi_v (d_i, g) = \chi_v (d_j) = 1$, so the right hand side is $(- t)^{- 1}
  \cdot 1 \cdot t = - 1$. Finally if $g < d_i$ then $\Sigma'_{(d_i, r]} =
  \Sigma_{(d_i, r]} - \{ g \}$ so
  \[ \frac{\beta_v (\mathfrak{s})}{\beta_v (\mathfrak{s}')} = t, \]
  again consistent with~(\ref{eq:bwratios}).
  
  Now if $i \in \{ 0, 1, 2, 3, 4 \}$ let
  \[ N_i (\mathfrak{s}) = \sum_{\text{$v$ a T-vertex}} \chi_v (d_i, g) . \]
  where the sum is over the two T-vertices $x_i, y$ and $x_j, y$. By
  (\ref{eq:chisumone}) we have
  \begin{equation}
    \label{eq:nisumtwo} \sum_i N_i (\mathfrak{s}) = 2.
  \end{equation}
  \begin{lemma}
		\label{lem:crossings}
    The quantities $N_i (\mathfrak{s})$ are depend only on $a, b, c, d, e, f$
    and not on the state $\mathfrak{s}$. They are the same whether
    $\mathfrak{s}$ is a state of $\operatorname{LHS} (a, b, c, d, e, f)$ or
    $\operatorname{RHS} (a, b, c, d, e, f)$.
  \end{lemma}
  
  \begin{proof}
    It is only necessary to prove this for $N_1 (\mathfrak{s}), N_2
    (\mathfrak{s}), N_3 (\mathfrak{s}), N_4 (\mathfrak{s})$
    by~(\ref{eq:nisumtwo}).
    
    Let us consider the colored paths of type $d_i$, $i = 1, 2, 3, 4$. Such a
    path begins at a ``source'' boundary edge $\boldsymbol{a}$, $\boldsymbol{b}$ or
    $\boldsymbol{c}$ and terminates at ``sink'' $\boldsymbol{d}, \boldsymbol{e}$ or
    $\boldsymbol{f}$. Now there is a path of color $g$ running from $\boldsymbol{c}$
    to $\boldsymbol{f}$. If $\boldsymbol{c}$ or $\boldsymbol{f}$ also carries the color
    $d_i$, we will consider that on the vertical edges ($\boldsymbol{c},
    \boldsymbol{f}$ or the internal vertical edge) that path runs to the left of
    the $g$-colored path if $d_i < g$, and to the right of the $g$-colored
    path if $d_i > g$.
    
    Let $N_i$ be the number of sources to the left of the $g$-path that carry
    the color $d_i$, minus the number of sinks to the left of the path that
    carry the color $d_i$. Thus $\boldsymbol{a}$ contributes $1$ to this count if
    $a = d_i$, and $\boldsymbol{b}$ contributes $1$ if $b = d_i$, while
    $\boldsymbol{c}$ contributes $1$ if $\boldsymbol{c}$ carries the color $d_i$ and
    $d_i < g$, but not otherwise. If $\boldsymbol{f}$ carries the color $d_i$ and
    $d_i < g$, then $\boldsymbol{f}$ contributes $- 1$ because it is a sink to the
    left of the path. To emphasize that $N_i$ depends only on the boundary
    data, we write $N_i = N_i (a, b, c, d, e, f)$.
    
    Now we claim that $N_i (\mathfrak{s}) = N_i (a, b, c, d, e, f)$. Indeed
    $N_i$ is the number of times that the colored $d_i$-path crosses the
    $g$-path, and the problem is to show that $\chi_v (d_i, g)$ correctly
    counts whether a crossing takes place at the T-vertex $v$. A nuance is
    that this depends on whether $d_i < g$ or $d_i > g$. Let us look at
    Cases~3, 5 and~6.
    
    In Case~3, it is a consequence of the assumption that the $d_i$ path runs
    to the left of the $g$-path if $d_i < g$ and to the right of the $g$-path
    if $d_i > g$ that the paths cross at the vertex if $d_i < g$, but not if
    $d_i < g$. This justifies the definition of $\chi_v (d_i, g)$ as $1$ if
    $d_i < g$ but $0$ if $d_i > g$.
    
  %
 %

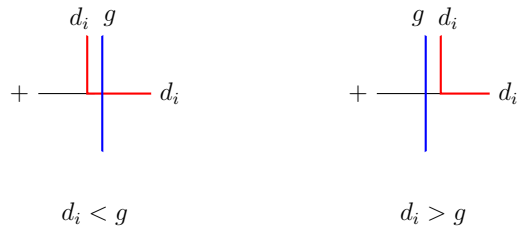
\[\begin{tikzpicture}[every node/.style={scale=.8}]
\draw (0,1)--(1,1);
\draw[thick, red] (1-.1,2)--(1-.1,1)--(2,1);
\draw[thick, blue] (1.1,2)--(1.1,0);
\foreach \x/\y in {0/1,2/1,1/0,1/2}
\path[fill=white] (\x,\y) circle (.25);
\node at (0,1) {$+$};
\node at (2,1) {$d_i$};
	\node at (1-.2,2) {$d_i$};
	\node at (1+.2,2) {$g$};
	\node at (1,-.6) {$d_i<g$};

\pgfmathsetmacro{\w}{4.5}

\draw (\w,1)--(\w+1+.1,1);
\draw[thick, red] (\w+1+.1,2)--(\w+1+.1,1)--(\w+2,1);
\draw[thick, blue] (\w+1-.1,2)--(\w+1-.1,0);
\foreach \x/\y in {0/1,2/1,1/0,1/2}
\path[fill=white] (\w+\x,\y) circle (.25);
\node at (\w,1) {$+$};
\node at (\w+2,1) {$d_i$};
	\node at (\w+1+.2,2) {$d_i$};
	\node at (\w+1-.2,2) {$g$};
	\node at (\w+1,-.6) {$d_i>g$};
		\end{tikzpicture}\]
    In Case~5, the $d_i$-path crosses the $d_i$-path if $g < d_i$ and it
    crosses the $d_j$ path at the vertex $v$ if $g > d_j$. If $d_i < g < d_j$
    it crosses neither path. This is reflected in the definition of $\chi_v
    (d_i, g)$.
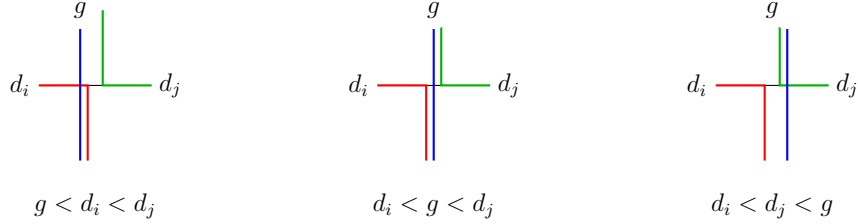
\[\begin{tikzpicture}[every node/.style={scale=.8}]
\draw (0,1)--(2,1);
\draw[thick, red] (0,1)--(1-.1,1)--(1-.1,0);
\draw[thick, green] (1+.1,2)--(1.1,1)--(2,1);
\draw[thick, blue] (1-.2,0)--(1-.2,2);
\foreach \x/\y in {0/1,2/1,1-.2/2}
\path[fill=white] (\x,\y) circle (.25);
\node at (0,1) {$d_i$};
\node at (2,1) {$d_j$};
\node at (1-.2,2) {$g$};
\node at (1,-.6) {$g<d_i<d_j$};

\pgfmathsetmacro{\w}{4.5};

\draw (\w,1)--(\w+2,1);
\draw[thick, red] (\w,1)--(\w+1-.1,1)--(\w+1-.1,0);
\draw[thick, green] (\w+1+.1,2)--(\w+1.1,1)--(\w+2,1);
\draw[thick, blue] (\w+1,0)--(\w+1,2);
\foreach \x/\y in {0/1,2/1,1/2}
\path[fill=white] (\w+\x,\y) circle (.25);
\node at (\w,1) {$d_i$};
\node at (\w+2,1) {$d_j$};
\node at (\w+1,2) {$g$};
\node at (\w+1,-.6) {$d_i<g<d_j$};

\draw (2*\w,1)--(2*\w+2,1);
\draw[thick, red] (2*\w,1)--(2*\w+1-.1,1)--(2*\w+1-.1,0);
\draw[thick, green] (2*\w+1+.1,2)--(2*\w+1.1,1)--(2*\w+2,1);
\draw[thick, blue] (2*\w+1+.2,0)--(2*\w+1+.2,2);
\foreach \x/\y in {0/1,2/1,1.2/2}
\path[fill=white] (2*\w+\x,\y) circle (.25);
\node at (2*\w,1) {$d_i$};
\node at (2*\w+2,1) {$d_j$};
\node at (2*\w+1,2) {$g$};
\node at (2*\w+1,-.6) {$d_i<d_j<g$};
\end{tikzpicture}\]
    In Case~6, with $d_j < d_i$, the $d_j$-path crosses the $g$-path if $d_j <
    g$, and the $d_i$-path crosses the $g$-path if $d_i < g$. In the case $d_j
    < g < d_i$ both the $d_i$ and $d_j$-paths cross the $g$-path, and this is
    reflected in the definition of $\chi_v$.
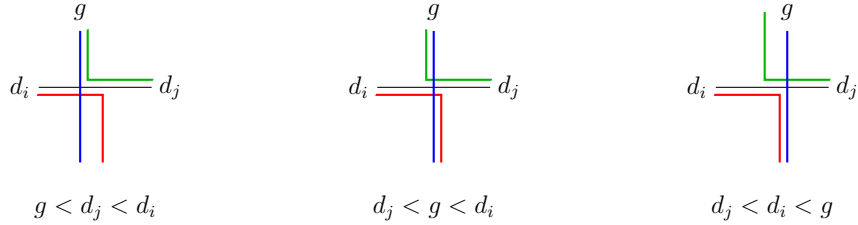
\[\begin{tikzpicture}[every node/.style={scale=.8}]
\draw (0,1)--(2,1);
\draw[thick, red] (0,1-.1)--(1+.1,1-.1)--(1+.1,0);
\draw[thick, green] (1-.1,2)--(1-.1,1+.1)--(2,1+.1);
\draw[thick, blue] (1-.2,0)--(1-.2,2);
\foreach \x/\y in {0/1,2/1,1-.2/2}
\path[fill=white] (\x,\y) circle (.25);
\node at (0,1) {$d_i$};
\node at (2,1) {$d_j$};
\node at (1-.2,2) {$g$};
\node at (1,-.6) {$g<d_j<d_i$};

\pgfmathsetmacro{\w}{4.5};

\draw (\w+0,1)--(\w+2,1);
\draw[thick, red] (\w,1-.1)--(\w+1+.1,1-.1)--(\w+1+.1,0);
\draw[thick, green] (\w+1-.1,2)--(\w+1-.1,1+.1)--(\w+2,1+.1);
\draw[thick, blue] (\w+1,0)--(\w+1,2);
\foreach \x/\y in {0/1,2/1,1/2}
\path[fill=white] (\w+\x,\y) circle (.25);
\node at (\w+0,1) {$d_i$};
\node at (\w+2,1) {$d_j$};
\node at (\w+1,2) {$g$};
\node at (\w+1,-.6) {$d_j<g<d_i$};

\draw (2*\w+0,1)--(2*\w+2,1);
\draw[thick, red] (2*\w,1-.1)--(2*\w+1+.1,1-.1)--(2*\w+1+.1,0);
\draw[thick, green] (2*\w+1-.1,2)--(2*\w+1-.1,1+.1)--(2*\w+2,1+.1);
\draw[thick, blue] (2*\w+1+.2,0)--(2*\w+1+.2,2);
\foreach \x/\y in {0/1,2/1,1+.2/2}
\path[fill=white] (2*\w+\x,\y) circle (.25);
\node at (2*\w+0,1) {$d_i$};
\node at (2*\w+2,1) {$d_j$};
\node at (2*\w+1+.2,2) {$g$};
\node at (2*\w+1,-.6) {$d_j<d_i<g$};
\end{tikzpicture} \qedhere \]	
  \end{proof}
  
  From the Lemma, we have
  \[ \frac{\operatorname{LHS} (a, b, c, d, e, f)}{\operatorname{LHS} (a, b, c', d, e, f')} =
     (- 1)^{N_0} \prod_{i = 1}^4 \left\{\begin{array}{ll}
       1 & \text{if $g < d_i,$}\\
       t^{N_i} & \text{if $g > d_i,$}
     \end{array}\right. \]
  since the expression on the right equals $\prod_v \beta_v (\mathfrak{s}) /
  \beta_v (\mathfrak{s}')$ for every state. Also the RHS ratio is equal to the
  same ratio, proving~(\ref{eq:hsratios}).
\end{proof}

\subsection{Relation with Lattice Models in \cite{AggarwalBorodinWheelerColored}}
Amol Aggarwal \cite{AggarwalPersonal} pointed out to us that
the Boltzmann weights in Figure \ref{fig:modified_weights} can be obtained
from the weights in \cite{AggarwalBorodinWheelerColored} via a series of
transformations (a mix of gauge transformations, Drinfeld twists,
normalizations, substitutions, and degenerations), each of which preserves
the Yang-Baxter equation. Combined with the Yang-Baxter equation in
\cite{AggarwalBorodinWheelerColored}, this gives an additional proof of
Theorem~\ref{thm:ybe}.

These transformations are as follows. First consider the weights in Figure \ref{fig:ABW-weights}, which are the Boltzmann weights from \cite[Definition~3.2.1]{AggarwalBorodinWheelerColored} in the case $m=1$ with a couple of minor convention changes.

For a given vertex, let $a, B, c$, and $D$ denote the (sets of) colors at the right, bottom, left, and top edges, respectively, and fix $B_0 = L-B_{[1,r]}, D_0 = L-D_{[1,r]}$. If $P$ is a logical statement, we consider $P$ to equal 1 if true and 0 if false.

Step 1: Start with the weights in Figure \ref{fig:ABW-weights}, multiply by
\[
(-1)^{\binom{B_0}{2} - \binom{D_0}{2}} (-t)^{B_{(a,r]} - D_{[0,c)} + L} \cdot \frac{1-t^{-1}z}{1-t^{L-1}z} \cdot \prod_{j=0}^r \frac{(t;t)_{c=j} (t;t)_{D_j}}{(t;t)_{a=j} (t;t)_{B_j}},
\]
where $(u;t)_k = \prod_{j=1}^k (1-t^{j-1}u)$ is the \emph{$t$-Pochhammer symbol}, and then reverse the colors.
Step 2: Multiply each weight by $1-t^{L-1}z$, set $x=t^{L-1}z$, and let $t^L\to 0$, keeping $x$ fixed. Further, multiply all weights by
\[
x^{-1} (-1)^{\binom{B_{[1,n]}}{2} - \binom{D_{[1,n]}}{2}}  (-1)^{c\neq +} (1-t)^{1_{a=+} - 1_{c=+}}.
\]
Step 3: Replace $x^{-1}$ with $xy^{-1}$, multiply by \[t^{1_{a\ne +} B_{[a+1,n]} - 1_{c\ne +} D_{[1,c-1]}},\] and let $L$ be odd so that $(-1)^L = -1$. The resulting weights match those in Figure \ref{fig:modified_weights}.

\begin{figure}[h]
\begin{center}
\begin{tabular}{|ccc|}
\hline
\hline
Case 1 & Case 2 & Case 3\\
\hline
\vertex{+}{B}{+}{B} &
\vertex{j}{B}{j}{B} &
\vertex{+}{B^+_j}{j}{B}
\\[4pt]
$t^{-B_{[1,r]}} \frac{1 - t^{B_{[1,r]}-1} z}{1 - t^{-1}z}$ & 
$t^{-B_{[i,r]}} (-1)^{B_j} \frac{1-t^{L-1+B_j}z}{1-t^{-1}z}$ &
$t^{L-B_{[1,r]}-1} z \frac{1 - t^{B_{[1,r]}-L}}{1-t^{-1}z}$ \\[4pt]
\hline
\hline
Case 4 & Case 5 & Case 6\\
\hline
\vertex{j}{B^-_j}{+}{B} &
\vertex{i}{B^{+-}_{ji}}{j}{B} &
\vertex{j}{B^{+-}_{ij}}{i}{B}
\\    
$t^{-B_{(j,r]}} \frac{t-1}{1-t^{-1}z}$ &
$t^{-B_{(j,r]} + L-1} z\frac{t-1}{1-t^{-1}z}$ &
$t^{-B_{(j,r]}} \frac{t-1}{1-t^{-1}z}$ \\
\hline
\end{tabular}
\end{center}
\caption{The set of Boltzmann weights in \cite[Definition~3.2.1]{AggarwalBorodinWheelerColored} in the case $m=1$, with some minor convention changes. $i$ and $j$ are colors with $i<j$. $B$ is a set of colors. $L$ is a global parameter, often taken to be a large positive integer.}
\label{fig:ABW-weights}
\end{figure}

\begin{remark} It follows from this connection with the results of~\cite{AggarwalBorodinWheelerColored} that
	there also exist \emph{column} Yang-Baxter equations. We do not need these, but
	they were used in~\cite{BumpMcNamaraNakasujiFactorial} 
	to prove stability properties of factorial Schur functions
	as the number of variables is increased. Note that~\cite{BumpMcNamaraNakasujiFactorial} 
	is a special case of the models in this paper. See Section~\ref{subsec:factorial}.
\end{remark}

\section{Evaluating Partition Functions via Demazure operators}\label{sec:partition-function}

Having specified Boltzmann weights for our solvable lattice models, we now
focus on the partition functions for finite rectangular lattices with $n$
rows, $N$ columns, and prescribed boundary conditions. Recall that we have
numbered the rows from top to bottom and columns from right to left. We
further assume that $N \geqslant n$. We need at least $n$ distinct colors, 
so we may take $r=n$, though any larger value of $r$ would work (but not all 
colors would be used).

\begin{definition}
  We associate a set of boundary conditions $B (\sigma, w, \lambda)$ on the
  rectangular lattice with $n$ rows and $N$ columns to a pair of permutations
  $\sigma$ and $w$ and an integer partition $\lambda = (\lambda_1, \cdots,
  \lambda_n)$ of length $n$ with $N \geqslant \lambda_1 \geqslant \lambda_2 \geqslant \cdots
  \geqslant \lambda_n>0$ as follows:
  \begin{itemize}
		\item Along the top boundary, the boundary edge in column $\lambda_{i}$ 
			carries the color $\sigma (n+1-i)$. All other top boundary edges are uncolored.
    In particular, when $\lambda$ has repeated parts, there will be columns
    with multiple colors.
    
    \item Along the right boundary, the boundary edge in row $i$ has color $w
    (i)$. All other right boundary edges are uncolored.
    
    \item Along the bottom and left boundaries, all boundary edges are
    uncolored.
  \end{itemize}
  \label{def:boundary}
\end{definition}

Following the language of Definition~\ref{def:admissiblestate}, given boundary condition $B$, 
the collection of admissible states is denoted $\mathfrak{S}_B$ or more explicitly for $B(\sigma, w, \lambda)$
by $\mathfrak{S}_{\sigma, w}^\lambda$.  Figure~\ref{fig:characteristicstate} below shows an example of an admissible state 
for the model $\mathfrak{S}^\lambda_{\sigma, w} = \mathfrak{S}^{(4, 3, 2, 1)}_{e, 1423}$ in the case $n = N = 4$. As that
state demonstrates, it is easy to verify that our collection of admissible vertices (i.e., those with non-zero Boltzmann weights)
produce states that appear as colored paths traveling downward and rightward from the top boundary.

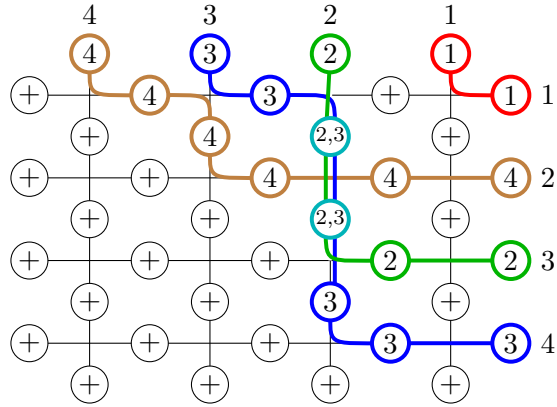
\begin{figure}[h] 
  \begin{center}
    \begin{tikzpicture}[xscale=0.8, yscale=0.55, font=\small]
    \icegrid{4}{4}

    \draw[white,line width=0.5mm] (5,3) to (5,8);
    \draw [brown, path]
    (1,8) node {$4$} to [out=-90,in=180,looseness=2.0]
    (2,7) node {$4$} to [out=-0,in=90,looseness=2.0]
    (3,6) node {$4$} to [out=-90,in=180,looseness=2.0]
    (4,5) node {$4$} --
    (6,5) node {$4$} --
    (8,5) node {$4$};

    \draw [blue, path]
    (3,8) node{$3$} to [out=-90,in=180,looseness=2.0]
    (4,7) node{$3$} to [out=0,in=90,looseness=2.0]
    (5+0.075,6);

    \draw[darkgreen, path]
    (5,8) node {$2$} --
    (5-0.075,6);

    \draw [blue, path] (5+0.075,6) -- (5+0.075,2);
    \draw [darkgreen, path] (5-0.075,6) -- (5-0.075,4) to [out=-90,in=180,looseness=2.0] (6,3);

    \draw [BlueGreen, invisiblepath]
    (5,6) node[draw=BlueGreen] {$\scriptstyle 2,3$} --
    (5,4) node[draw=BlueGreen] {$\scriptstyle 2,3$};

    \draw [blue, path]
    (5,2) node{$3$} to [out=-90,in=180,looseness=2.0]
    (6,1) node{$3$} --
    (8,1) node{$3$};

    \draw[darkgreen, path]
    (6,3) node {$2$} --
    (8,3) node {$2$};

    \draw [red, path]
    (7,8) node {$1$} to [out=-90,in=180,looseness=2.0]
    (8,7) node {$1$};

  \end{tikzpicture}
  \end{center}
   \caption{An admissible state for the model $\mathfrak{S}^{(4, 3, 2,
  1)}_{e, 1423}$.\label{fig:characteristicstate}}
\end{figure}

Given the required boundary data of permutations $\sigma$ and $w$ and partition $\lambda$, we make a
family of partition functions associated to this data:
\begin{equation}
  (\sigma, w, \lambda) \longmapsto Z^{\lambda}_{\sigma, w} (x_1, \ldots, x_n ;
  y_1, \ldots, y_N) \label{zpartition}
\end{equation}
where $Z_{\sigma, w}^{\lambda} := Z_{B (\sigma, w, \lambda)}$ defined
according to~(\ref{generic-partition-function}) and boundary conditions $B
(\sigma, w, \lambda)$ as in Definition~\ref{def:boundary} above. The Boltzmann
weights are as in Figure~\ref{fig:modified_weights}. For brevity, we sometimes
make use of the vector notation $\bs{x} = (x_1, \ldots, x_n)$ and $\bs{y} =
(y_1, \ldots, y_N)$.

\begin{example} \label{ex:partfun} In the simplest non-trivial case $N=n=2$, the Weyl group has just two elements $\{e, s \}$ and we must have $\lambda = (2,1)$. There are four allowable boundary conditions corresponding to choices of top boundary correspoding to $\sigma$ and right-hand boundary corresponding to $w$ in $\{ e, s \}$. We depict the admissible states for two of the four choices here. The boundary conditions $B(\sigma, w, \lambda)  = B(e,s,(2,1))$ admit a single admissible state (pictured below). Thus its partition function is just the weight of this lone admissible state. The table at right of the state shows the Boltzmann weights for each individual vertex in their respective positions, using the weights from Figure~\ref{fig:modified_weights}, resulting in the partition function $Z_{e,s}^{(2,1)}(\boldsymbol{x};\boldsymbol{y}) = 1 - x_1 y_1^{-1}$.
 \[  \begin{tikzpicture}[xscale=0.8, yscale=0.7, font=\small, baseline={1.5cm}]
    \icegrid{2}{2}
\draw [blue, path]
    (1,4) node{$2$} to [out=-90,in=180,looseness=2.0]
    (2,3) node{$2$};
   
    \draw [blue, path]
    (2,3) node{$2$} to [out=0,in=180,looseness=2.0]
    (4,3) node{$2$};

    \draw[red, path]
    (3,4) node {$1$} --
    (3,2) node {$1$};

    \draw [red, path]
    (3,2) node {$1$} to [out=-90,in=180,looseness=2.0]
    (4,1) node {$1$};
  \end{tikzpicture} \quad \begin{tabular}{c|c} 1 & $1 - x_1 y_1^{-1}$ \\ \hline 1 & 1 \end{tabular} \]
The boundary conditions $B(\sigma, w, \lambda) = B(e,e, (2,1))$ admit two admissible states depicted below, with the Boltzmann weights of vertices at their respective positions to the right. Summing over states, the resulting partition function $Z_{e,e}^{(2,1)}(\boldsymbol{x};\boldsymbol{y}) = (t-1) - t(1-x_2 y_1^{-1}).$
 \[  \begin{tikzpicture}[xscale=0.8, yscale=0.7, font=\small, baseline={1.5cm}]
    \icegrid{2}{2}
\draw [blue, path]
    (1,4) node{$2$} to [out=-90,in=90,looseness=2.0]
    (1,2) node{$2$};

\draw [blue, path]   
   (1,2) node{$2$} to [out=-90,in=180,looseness=2.0]
    (2,1) node{$2$};

\draw [blue, path]   
   (2,1) node{$2$} to [out=0,in=180,looseness=2.0]
    (4,1) node{$2$};
   
     \draw [red, path]
    (3,4) node {$1$} to [out=-90,in=180,looseness=2.0]
    (4,3) node {$1$};
  \end{tikzpicture} \quad \begin{tabular}{c|c} $-t$  & 1 \\ \hline 1 & $1 - x_2 y_1^{-1}$ \end{tabular} \; ;
  \qquad
  \begin{tikzpicture}[xscale=0.8, yscale=0.7, font=\small, baseline={1.5cm}]
    \icegrid{2}{2}
\draw [blue, path]
    (1,4) node{$2$} to [out=-90,in=180,looseness=2.0]
    (2,3) node{$2$};
   
    \draw [blue, path]
    (2,3) node{$2$} to [out=0,in=90,looseness=2.0]
    (3,2) node{$2$};

   \draw [blue, path]
    (3,2) node{$2$} to [out=-90,in=180,looseness=2.0]
    (4,1) node{$2$};

    \draw [red, path]
    (3,4) node {$1$} to [out=-90,in=180,looseness=2.0]
    (4,3) node {$1$};
  \end{tikzpicture} \quad \begin{tabular}{c|c} 1 & $t-1$ \\ \hline 1 & 1 \end{tabular} \; .
  \]
\end{example}

\medskip

Our next result uses the row solvability of our weights
to completely characterize the partition function $Z^{\lambda}_{\sigma, w}
(\bs{x} ; \bs{y})$ as the solution to a recurrence relation. 
Because we allow repeated parts in the partition, there is some ambiguity in
the choice of the top boundary permutation $\sigma$ up to elements of the
stabilizer $\operatorname{Stab}(w_0 \cdot \lambda)$ of $w_0 \cdot \lambda$ in
$S_n$. In what follows, we choose $\sigma$ in $S_n$ to be the
\textit{maximal} length element in its coset $[\sigma] \in S_n /
\mathrm{Stab} (w_0 \cdot \lambda)$.

\begin{example}
  Let $n = 3$ and $\lambda = (4, 2, 2)$ with $N \geqslant 4$. Then $w_0 \cdot
  \lambda = (2, 2, 4)$. Then $\operatorname{Stab}(w_0 \cdot \lambda) = \langle s_1
  \rangle$. If we choose the coset $[s_2]$ in $S_n / \mathrm{Stab} (w_0
  \cdot \lambda)$, then the maximal element in the left coset $[s_2]$ is $\sigma =
  s_2 s_1$. We have color set $(c_1, c_2, c_3)$ with $c_1 < c_2 < c_3$. The
  top boundary corresponding to $[\sigma]$, that is to any element of this
  coset, has color $c_2$ on column 4, and colors $c_1, c_3$ on column 2.
\end{example}

\begin{proposition}[Base Case]
  \label{prop:basecase}Given any partition $\lambda$ of length $n$, let $\sigma$ in $S_n$ be
  the maximal length coset representative for $[\sigma]$ in $S_n /
  \mathrm{Stab} (w_0 \cdot \lambda)$. Then
  \[ Z^{\lambda}_{\sigma, \sigma w_0} (\bs{x} ; \bs{y}) = t^{\ell (\sigma)} 
     \prod_{i = 1}^n \prod_{j < \lambda_i} \left( 1 - \frac{x_i}{y_j} \right)
     . \]
\end{proposition}

\begin{proof}
  By choosing the right boundary to be $w = \sigma w_0$, then there exists a
  unique admissible state with these boundary conditions. It is uniquely
  characterized by the property that the color exiting in row $i$ originated
  at column $\lambda_i$ and if $\lambda_i = \lambda_j$ then the smaller of the
  two colors $c_{\sigma (n + 1 - i)}$ and $c_{\sigma (n + 1 - j)}$ exits
  above. Our choice of the maximal length coset representative ensured that
  the latter convention holds when choosing the right boundary to be $\sigma
  w_0$.
  
  The vertices of the admissible state with at least one adjacent edge colored
  are those in row $i$ and column $j$ with $j \leqslant \lambda_i$. (The
  reader may visualize this as colored paths tracing out the Young diagram of
  $\lambda$ with boxes right-justified, where the boxes are centered on each
  of the vertices in the state having at least one adjacent colored edge.)
  
  The weight of this state may now be read from
  Figure~\ref{fig:modified_weights}. The vertices at positions $(i,
  \lambda_i)$ are in Case 3 and those at $(i, j)$ with $j < \lambda_i$ are in
  Case 2. The remaining vertices have no adjacent edges with color, so belong
  to Case 1 with Boltzmann weight 1.
\end{proof}

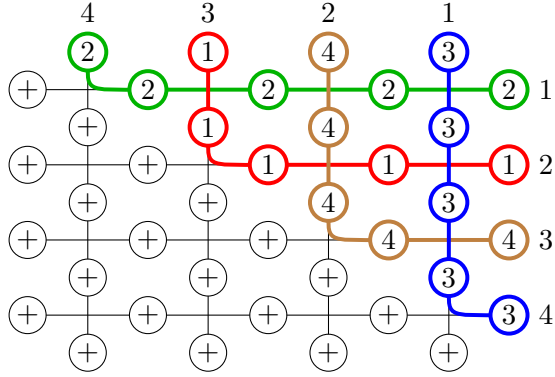
\begin{figure}[h]
  \begin{center}
    \
    \begin{tikzpicture}[xscale=.8, yscale=0.5, font=\small]
    \icegrid{4}{4}

    \draw [darkgreen, path]
    (1,8) node {$2$} to [out=-90,in=180,looseness=2.0]
    (2,7) node {$2$} --
    (4,7) node {$2$} --
    (6,7) node {$2$} --
    (8,7) node {$2$};

    \draw [red, path]
    (3,8) node {$1$} --
    (3,6) node {$1$} to [out=-90,in=180,looseness=2.0]
    (4,5) node {$1$} --
    (6,5) node {$1$} --
    (8,5) node {$1$};

    \draw [brown, path]
    (5,8) node {$4$} --
    (5,6) node {$4$} --
    (5,4) node {$4$} to [out=-90,in=180,looseness=2.0]
    (6,3) node {$4$} --
    (8,3) node {$4$};

    \draw [blue, path]
    (7,8) node {$3$} --
    (7,6) node {$3$} --
    (7,4) node {$3$} --
    (7,2) node {$3$} to [out=-90,in=180,looseness=2.0]
    (8,1) node {$3$};

  \end{tikzpicture}
  \end{center}
  \caption{The sole admissible state for the model
  $\mathfrak{S}_{\sigma, \sigma w_0}$, with $\sigma = 3412$. We refer to this as the ``ground state'' for all models of the form $\mathfrak{S}_{\sigma, w}.$}
\end{figure}

Next we turn to recursive relations, which together with the above result,
completely determine the partition function. In expressing these recursions,
we find it convenient to use root theoretic notation. Let $x^{\alpha_i}$
denote $x_i / x_{i + 1}$.
Let $\tau_i$ be the operator on $\mathbb{C} (t) [x_1,
  \ldots, x_n]$ defined by~(\ref{eq:tauop}). In particular, $\tau_i$ is invertible with inverse
	\[ \tau_i^{- 1} = \frac{(t - 1) \bs{x}^{\alpha_i}}{t (1 - \bs{x}^{\alpha_i})} +
		 \frac{\bs{x}^{\alpha_i} - t}{t (1 - \bs{x}^{\alpha_i})} s_i, \]
We will also make use of $\overline\tau=\theta\tau\theta$, where $\theta$ is the map on functions
that sends $\bs{x}\to\bs{x}^{-1}$ and $\bs{y}\to\bs{y}^{-1}$. Thus
\begin{equation} \label{eq:tau-bar}
\overline\tau_i \cdot f(\bs{x}) 
= \frac{t-1}{1-\bs{x}^{-\alpha_i}} f(\bs{x}) + \frac{\bs{x}^{-\alpha_i}-t}{1-\bs{x}^{-\alpha_i}} f(s_i\bs{x})
\end{equation}
and
\[\overline\tau_i^{-1}\cdot f(\bs{x}) 
= \frac{(t - 1) \bs{x}^{-\alpha_i}}{t (1 - \bs{x}^{-\alpha_i})} + \frac{\bs{x}^{-\alpha_i} - t}{t (1 - \bs{x}^{-\alpha_i})} s_i.\]

\begin{proposition}[Recursion]
  \label{prop:recursion} Then 
  \[ Z^{\lambda}_{\sigma, ws_i} (\bs{x} ; \bs{y}) = \left\{\begin{array}{ll}
       \tau_i \cdot Z^{\lambda}_{\sigma, w} (\bs{x} ; \bs{y}), & \text{if }
       \ell (ws_i) < \ell (w),\\
       \tau_i^{- 1} \cdot Z^{\lambda}_{\sigma, w} (\bs{x} ; \bs{y}), &
       \text{if } \ell (ws_i) > \ell (w) .
     \end{array}\right. \]
  Here we have extended the action of $\tau_i$ to $\mathbb{C} (t) [x_1, \ldots,
  x_n ; y_1, \ldots, y_N]$ by having $s_i$ act only on $\bs{x}$ while acting
  trivially on $\bs{y}$.
\end{proposition}

\begin{proof}
  Because our $R$-vertex weights precisely match those in
  {\cite{BBBGIwahori}}, the same proof (Proposition 7.1 in
  {\cite{BBBGIwahori}}) demonstrates the recursion. The formula for the
  inverse of $\tau_i$ may be checked by direct computation, or inferred from
  the quadratic relations for $\tau_i$ given in {\cite{BBBGIwahori}}.
\end{proof}

By a \textit{Weyl vector} we mean a vector $\rho$ characterized by the condition
that $s_i (\rho) = \rho - \alpha_i$ for every simple root $\alpha_i$. A
common choice is to take $\rho$ to be half the sum of the positive roots, but for
$\mathrm{GL} (n)$ a convenient choice is $(n - 1, n - 2, \cdots, 0)$, where
half the sum of the positive roots would be $(\frac{n - 1}{2}, \frac{n -
3}{2}, \cdots, \frac{1 - n}{2})$. Changing the Weyl vector does not affect
anything of significance. For example, either vector produces the same result
in the Weyl character formula. So for us we make the convenient choice $\rho =
(n - 1, n - 2, \cdots, 0)$.

In {\cite{BBBGIwahori}}, it is checked that $\tau_i = \bs{z}^{\rho}
\mathfrak{T}_i \bs{z}^{- \rho}$ where the operators $\mathfrak{T}_i$ appear in
{\cite{BBBGIwahori}}, Equation (11), and are sometimes referred to as
``Demazure-Whittaker operators'' in contrast to ``Demazure-Lusztig
operators.''\footnote{The former arise in the antispherical module of the
affine Hecke algebra $\mathcal{H}$ -- the representation $\mathcal{H}
\otimes_{\mathcal{H}_0} \epsilon$ induced from the sign character $\epsilon$
of the finite Hecke algebra $\mathcal{H}_0$ -- while the latter
Demazure-Lusztig operators arise in the induced representation from the
``trivial'' character of the Hecke algebra.} The $\mathfrak{T}_i$ satisfy
braid relations (see for example, Proposition 5 of {\cite{BBL}}), and hence
the operators $\tau_i$ satisfy the braid relations $\tau_i \tau_{i + 1} \tau_i
= \tau_{i + 1} \tau_i \tau_{i + 1}$ for all $i$. Thus it makes sense to define
\[ \tau_w := \tau_{i_1} \tau_{i_2} \cdots \tau_{i_{\ell}} \]
for any reduced decomposition $w = s_{i_1} s_{i_2} \cdots s_{i_{\ell}}$.

\begin{theorem}
  \label{thm:partition-function}For any partition $\lambda$ and any
  permutations $\sigma, w$ in $S_n$ such that $\sigma$ is the maximal length
  coset representative for $[\sigma]$ in $S_n / \mathrm{Stab} (\lambda)$,
  \[ Z_{\sigma, w}^{\lambda} (\bs{x} ; \bs{y}) = t^{\ell (\sigma)} \tau_w^{-
     1} \tau_{\sigma w_0} \cdot \left[ \prod_{i = 1}^n \prod_{j <
     \lambda_i} \left( 1 - \frac{x_i}{y_j} \right) \right] \]
\end{theorem}

\begin{proof}
  Combine the previous two propositions, noting that we must first go down in
  Bruhat order from $\sigma w_0$ to $e$ to guarantee all steps are descents,
  and then back up to $w$ using only ascents in the Bruhat order.
\end{proof}

Note that the prior theorem is just one of the many ways in which we may
express the partition function $Z_{\sigma, w}^{\lambda} (\bs{x} ; \bs{y})$ in
terms of divided difference operators. In general, one must take care to use
the appropriate operator $\tau_i$ or $\tau_i^{- 1}$ according to whether one
is ascending or descending from $w$ to $ws_i$ in the Bruhat order. We have
chosen one sure-fire way to do this above, in proceeding from $\sigma w_0$
down to the identity and then back up in Bruhat order to $w$.

\subsection{Specialization of variables} \label{sec:specialization-variables}

We turn now to particular specializations of our partition functions in
the previous section.
In these cases, we restrict to the square lattice where $N = n$.
The partition $\lambda$ defining the top boundary conditions
will be taken to be $\square := (n, n - 1, \cdots, 1)$. 
This means that our boundary conditions are such that the
lattice is a square grid with $n$ rows and columns, and all top and right
boundary edges receive a single color according to $\sigma$ and $w$. These
conventions will be in force in all subsequent sections except
Subsection~\ref{subsec:iwahori}. Thus in particular the Weyl vector is 
$(n - 1, n - 2, \cdots, 0)$.  

We will later make a connection between our partition functions and certain
special functions in the K-theory of the flag variety. To this end, we need a
few preliminaries about how this specialization operation works on polynomials
in two sets of variables $\bs{x}$ and $\bs{y}$. Given a permutation $u$ and a
function $f (\bs{x} ; \bs{y}) = f (x_1, \ldots, x_n ; y_1, \ldots, y_n)$,
define the {\textit{restriction}} of $f$ at $u$ to be
\begin{equation} f |_u (\bs{y}) = f|_u (y_1, \ldots, y_n) := f (y_{u (1)}, \ldots, y_{u
   (n)} ; y_1, \ldots, y_n) . \label{eq:permrestriction} \end{equation}
Define
\[
\bs{x}_w:=
w^{-1}\cdot\bs{x} = w^{-1}\cdot (x_1,\ldots,x_n) = (x_{w(1)},\ldots, x_{w(n)}).
\]
   The symmetric group also acts on functions by permuting their $x$ variables
\[ w \cdot f (\bs{x} ; \bs{y}) = f (\bs{x}_w ; \bs{y}), \]
which has the following relationship with restriction
\begin{equation}
	\label{eq:relres}
 (w\cdot f) |_u = f|_{uw} . 
\end{equation}
Now for any permutation $u \in S_n$, define the following specialization of
our two sets of variables:
\[ Z^{\square}_{\sigma, w} |_u (\bs{y}) := Z^{\square}_{\sigma, w} (\bs{y}_u
   ; \bs{y}), \]
so that each $x_i$ is set to $y_{u (i)}$. In particular under this
specialization, for any root $\alpha$, $\bs{x}^{\alpha} |_u = \bs{y}^{u \alpha}$. We
sometimes drop the variables $\bs{y}$ and simply write $Z^{\square}_{\sigma,
w} |_u$ when clear from context. The following result follows from a
straightforward application of Propositions~\ref{prop:basecase}
and~\ref{prop:recursion} under this specialization operation, and its proof is
left to the reader.

\begin{proposition}
  \label{prop:special-recursion}For any permutations $\sigma, u$, then
  $Z^{\square}_{\sigma, \sigma w} |_u (\bs{y}) \ne 0$ if and only if $w =
  w_0$. In this case,
  \[ Z^{\square}_{\sigma, \sigma w_0} |_{w_0} (\bs{y}) = t^{\ell (\sigma)} 
     \prod_{\alpha \in \Phi^+} (1 - \bs{y}^{- \alpha}) \]
  and to any permutations $\sigma, w, u$ and any simple reflection $s_i$ in
  $S_n$,
  \[ Z^{\square}_{\sigma, ws_i} |_u = \left\{\begin{array}{ll}
       \frac{t - 1}{1 - y^{u \alpha_i}} Z^{\square}_{\sigma, w}  |_u +
       \frac{y^{u \alpha_i} - t}{1 - y^{u \alpha_i}} Z^{\square}_{\sigma, w}
       |_{us_i} & \textrm{if } \ell (ws_i) < \ell (w),\\
       \frac{(t - 1) y^{u \alpha_i}}{t (1 - y^{u \alpha_i})}
       Z^{\square}_{\sigma, w}  |_u + \frac{y^{u \alpha_i} - t}{t (1 - y^{u
       \alpha_i})} Z^{\square}_{\sigma, w} |_{us_i} & \textrm{if } \ell (ws_i)
       > \ell (w) .
     \end{array}\right. \]
\end{proposition}

\begin{proposition}
  \label{prop:vanishing}For any permutations $\sigma, w, u$ in $S_n$,
  $Z^{\square}_{\sigma, w} |_u (\bs{y}) = 0$ unless $\sigma^{- 1} w \leqslant
  u$.
\end{proposition}

\begin{proof}
  This follows by induction on $w$, moving down in the Bruhat order for
  elements of the form $\sigma^{- 1} w$. The base case of the induction is
  thus $w = \sigma w_0$ given in the previous proposition. The induction step,
  required for $s_i$ such that $\sigma^{- 1} ws_i < \sigma^{- 1} w$, follows
  from the recursion of the prior proposition, noting that if $\sigma^{- 1}
  ws_i \nleqslant u$ then both $\sigma^{- 1} w \nleqslant u$ and $\sigma^{- 1}
  w \nleqslant us_i$.
\end{proof}

\begin{proposition}
  For any permutations $\sigma, w$ in $S_n$,
  \[ Z^{\square}_{\sigma, w} |_{\sigma^{- 1} w} (\bs{y}) = t^{\frac{1}{2} 
     (\ell (\sigma) - \ell (w) + \ell (w^{- 1} \sigma))}  (- 1)^{\ell (w^{- 1}
		 \sigma w_0)}  \prod_{\substack{
       \alpha \in \Phi^+\\
       w^{- 1} \sigma (\alpha) \in \Phi^+ }} (1 - ty^{- \alpha})  
		 \prod_{\substack{
       \alpha \in \Phi^+\\
       w^{- 1} \sigma (\alpha) \in \Phi^- }} (1 - y^{- \alpha}) . \]
\end{proposition}

\begin{proof}
  Again, we use induction in the Bruhat order from the base case $\sigma w_0$
  to $w$. Note that for any permutation $v$, $\sigma^{- 1} vs_i < \sigma^{- 1}
  v$, then by Proposition~\ref{prop:vanishing}, $Z^{\square}_{\sigma,
  \sigma^{- 1} v} |_{\sigma^{- 1} vs_i} = 0$. This reduces our recursion when
  $\sigma^{- 1} vs_i < \sigma^{- 1} v$ from
  Proposition~\ref{prop:special-recursion} to :
  \begin{equation}
    Z^{\square}_{\sigma, vs_i} |_{\sigma^{- 1} vs_i} = \frac{y^{- \sigma^{- 1}
    v (\alpha_i)} - t}{1 - y^{- \sigma^{- 1} v (\alpha_i)}}
    Z^{\square}_{\sigma, v} |_{\sigma^{- 1} v} \cdot \left\{\begin{array}{ll}
      1 & \textrm{if } \ell (ws_i) < \ell (w),\\
      t^{- 1} & \textrm{if } \ell (ws_i) > \ell (w) .
    \end{array}\right. \label{t-recursion}
  \end{equation}
  The proposition now follows by combining the expression for the base case
  from Proposition~\ref{prop:special-recursion} with the above simplified
  recursion, recalling the standard fact that to any permutation $u$, the set
  $\{\alpha \in \Phi^+ \hspace{0.27em} | \hspace{0.27em} u (\alpha) \in \Phi^-
  \} = \{\alpha_{i_k}, s_{i_k} (\alpha_{i_{k - 1}}), \ldots, s_{i_k} s_{i_{k -
  1}} \cdots s_{i_2} (\alpha_{i_1})\}$ if $u = s_{i_1} \cdots s_{i_k}$ is a
  reduced expression. (See for example, {\cite[Proposition 20.10]{BumpLie}}).
  We still need to keep careful track of the power of $t$ in the above cases
  of (\ref{t-recursion}) at each step in the Bruhat chain. The total number of
  steps in the chain is $\ell (w^{- 1} \sigma w_0)$ so the power of $t^{- 1}$
  will then be
  \[ \frac{1}{2} (\ell (w^{- 1} \sigma w_0) - (\ell (\sigma w_0) - \ell (w)) =
     \frac{1}{2}  (\ell (\sigma) + \ell (w) - \ell (w^{- 1} \sigma)) . \]
  The total power of $t$ is this quantity multiplied by $t^{\ell (\sigma)}$
  (from the base case) giving the desired result.
\end{proof}

\begin{remark} 
Similar arguments to the ones from the previous three propositions can be used to show
	vanishing conditions for the restriction functions of the lattice model with
	more general boundary conditions corresponding to an arbitrary partition $\lambda$. 
	To state them, we can extend the notion of restriction in (\ref{eq:permrestriction}) 
	to the case of an integer composition 
	$\mu = (\mu_1,\ldots,\mu_n)$. Define the restriction of $f$ at $\mu$ to be
\[ 
f|_\mu (\bs{y}) := f (y_{\mu(1)}, \ldots, y_{\mu(n)} ; y_1, \ldots, y_n) .
\]
	Suppose that $\mu$ has all parts distinct.
	There is a unique partition $\nu = (\nu_1,\ldots,\nu_n)$ and permutation $u$
	such that $\mu_i = \nu_{u(i)}$ for all $i$, and we write $\mu = \nu u$. Then
	assuming $|\nu|\leqslant |\lambda|$, $Z^{\lambda}_{\sigma,w}|_{\nu u} = 0$
	unless $\nu = \lambda$ and $u\leqslant w_0\sigma^{-1}w$.

Moreover, in the case $u = w_0\sigma^{-1}w$,
\begin{equation} \label{eq:diag}
Z^{\lambda}_{\sigma, w}|_{\lambda w_0\sigma^{-1}w} = t^{(\ell(\sigma)-\ell(w) + \ell(w_0) - \ell(w_0\sigma^{-1}w))/2}\prod_{\substack{\alpha\in\Phi^-\\ w\sigma w_0(\alpha)\in\Phi^+}} \frac{\bs{y}^{\lambda(\alpha)} - t}{1 - \bs{y}^{\lambda(\alpha)}} \prod_{i = 1}^n \prod_{j < \lambda_i} \left( 1 - \frac{y_{\lambda_i}}{y_j} \right),
\end{equation}
where if $\alpha = e_i-e_j, \lambda(\alpha) = e_{\lambda_i}-e_{\lambda_j}$.
\end{remark}

\subsection{Vanishing Properties\label{subsec:vanishing}}

We have already witnessed several results in the previous section on the vanishing of 
the partition function $Z^\square_{\sigma,w}(\boldsymbol{x}; \boldsymbol{y})$ at specializations 
of the variables $\boldsymbol{y}$ in terms of $\boldsymbol{x}$. Here we make some additional 
remarks about vanishing of the partition function in more general cases, replacing the top boundary
condition $\square$ with an arbitrary partition $\lambda = (\lambda_1, \ldots, \lambda_n)$. 
Before doing so, we mention several related results from earlier works.

In Section~4 of~\cite{SahiInterpolation}, Sahi defines a family of non-symmetric polynomials 
$G_\alpha(\boldsymbol{x};q,t)$ depending on a composition $\alpha$ as the unique (up to normalization) 
family vanishing on certain integer points in the convex hull of the $W$-orbit of $\alpha$. Their symmetrization 
results in polynomials whose top homogeneous component are Macdonald 
polynomials.\footnote{In fact, there are further generalizations in \cite{SahiInterpolation} allowing 
for $t$ to be replaced by multiple parameters $\tau_1, \ldots, \tau_n$. These bring to mind earlier 
special cases of our lattice models treated in \cite{hkice}, which allowed
for $t$ to be replaced by parameters $t_i$ in row $i$. We don't pursue this connection further here.} If $q=0$, then
combining Theorem~4.5 in~\cite{SahiInterpolation} with results in \cite{BBBGIwahori} demonstrates that 
$G_\alpha(\boldsymbol{x}; 0, t)$ are (up to a constant multiple) Iwahori Whittaker functions. 
Indeed, both satisfy recursive relations on word length in terms of the Demazure-Whittaker 
operators $\tau_i$ in~\eqref{eq:tauop}. The connection between Iwahori Whittaker functions and 
specializations of our partition functions in this paper is treated in Section~\ref{subsec:iwahori}. 
Thus Sahi's vanishing results, which allow for general $q$, and our vanishing results, which allow 
for a second set of variables $\boldsymbol{y}$, coincide in the special case $q=0$ and 
$\boldsymbol{y} = (1,\ldots,1)$.

As we noted in the introduction, symmetric variants of Sahi's polynomials were treated via lattice models in 
Section~12 of~\cite{AggarwalBorodinWheelerColored}. There the vanishing of the matching partition function 
was explained by the vanishing of each admissible state. Such vanishing occurs in our models as well. Consider 
the admissible vertex of Case 2 in~Figure~\ref{fig:modified_weights} whose Boltzmann weight includes the 
factor $(-xy^{-1} + t^{\Sigma_c})$ where $\Sigma_c$ records the multiplicity along the vertical edge of the 
distinguished color $c$ traveling along the horizontal edge. If we choose boundary conditions such that all 
admissible states must have a Case 2 vertex at location $(i.j)$, then vanishing results at $x_i = t^{\Sigma_c} y_j$ 
immediately follow. We may refer to this as ``local vanishing'' of the partition function, in the case where the 
partition function is zero because specializations of Boltzmann weights vanish. 
In~\cite{AggarwalBorodinWheelerColored}, the vertices at which the vanishing occurs are called ``blocking vertices'' 
as they block the paths in non-zero admissible states. Propositions~\ref{prop:special-recursion} and~\ref{prop:vanishing} 
give examples of local vanishing results for lattice models, though the proofs make use of the Demazure recursion rather
than combinatorial arguments based on blocking vertices.

In this section, we call attention to certain ``global vanishing'' results for our partition functions. The simplest example 
of this can already be seen in Example~\ref{ex:partfun}. There we computed two partition functions:
\[ Z_{e,s}^{(2,1)}(\boldsymbol{x};\boldsymbol{y}) = 1 - x_1 y_1^{-1} \quad \text{and} \quad Z_{e,e}^{(2,1)}(\boldsymbol{x};\boldsymbol{y}) = (t-1) - t(1-x_2 y_1^{-1}). \]
Here $e$ is the identity element and $s$ is the non-trivial element in $S_2$. 
The former partition function $Z_{e,s}^{(2,1)}$ consists of a single state, so the notions of local and global 
vanishing coincide. But the latter function $Z_{e,e}^{(2,1)}$ exhibits only globally vanishing when $y_1 = t x_2$. This phenomenon extends
to all possible rank one cases according to the following result.

\begin{proposition} Let $\lambda = (\lambda_1, \lambda_2)$ and consider the associated partition functions 
\[ Z^{\lambda}_{\sigma, w} (\bs{x} ; \bs{y}) := Z^{(\lambda_1, \lambda_2)}_{\sigma, w} (x_1, x_2 ; y_1, \ldots, y_{\lambda_1}) \]
for any choice of $\sigma, w$ in $\{ e, s \} = S_2$.

\medskip

\noindent {\bf Case 1:} If $\sigma \ne w$, then $Z^{\lambda}_{\sigma, w} (\bs{x} ; \bs{y})$ is monostatic, hence only exhibits local vanishing.
This vanishing occurs precisely when $x_1= y_k$ for $k < \lambda_1$ or $x_2 = y_k$ for $k < \lambda_2$.

\medskip

\noindent {\bf Case 2:} If $\sigma = w = e$, then $Z^{\lambda}_{\sigma, w} (\bs{x} ; \bs{y})$ has $\lambda_1 - \lambda_2 + 1$ admissible states. It exhibits
local vanishing, i.e., all states vanish, if $x_i = y_k$ for $k < \lambda_2$ and $i = 1$ or $2$. It exhibits global vanishing if $x_1 = y_k$ and $x_2 = t^{-1} y_k$ for $\lambda_2 \leqslant k < \lambda_1$.
\end{proposition}

\begin{proof} In both Case 1 and Case 2, the local vanishing is a simple consequence of the fact that each state is ``frozen'' (forced into a particular configuration) 
at every vertex to the right of column $\lambda_2$. More precisely, all admissible states have colored paths traveling horizontally along each row to their exit 
on the right-hand boundary. Such vertices are of type ``Case 2'' in Figure~\ref{fig:modified_weights} which gives the local vanishing.

For the global vanishing in Case 2, we set $\sigma = w = e$. The case of $\sigma = w = s$ is similar and is left to the reader. In this case, we may evaluate the partition function via the recursion in Proposition~\ref{prop:recursion} to obtain:
\begin{align*}
Z^{\lambda}_{e, e} &= \frac{t-1}{1-x_1/x_2} Z^{\lambda}_{e, s} + \frac{x_1/x_2-t}{1-x_1/x_2} Z^{\lambda}_{e, s}
\\&= \frac{t-1}{1-x_1/x_2} 
    \prod_{k < \lambda_1} \left( 1 - \frac{x_1}{y_k} \right) \prod_{k < \lambda_2} \left( 1 - \frac{x_2}{y_k} \right) + \frac{x_1/x_2-t}{1-x_1/x_2}
   \prod_{k <  \lambda_1} \left( 1 - \frac{x_2}{y_k} \right) \prod_{k < \lambda_2} \left( 1 - \frac{x_1}{y_k} \right)
\\&= \frac{1}{1-x_1/x_2} \prod_{k< \lambda_2} \left( 1 - \frac{x_1}{y_k} \right) \left( 1 - \frac{x_2}{y_k} \right)
	\\& \ \times\left[ (t-1) \prod_{\lambda_2 \leqslant k < \lambda_1} \left( 1 - \frac{x_1}{y_k} \right) + \left( \frac{x_1}{x_2}-t \right)
   \prod_{\lambda_2 \leqslant k < \lambda_1} \left( 1 - \frac{x_2}{y_k} \right) \right],
\end{align*}
and this vanishes whenever either $x_1 =y_k$ or $x_2=y_k$ for some $k < \lambda_2$ (in which case all states vanish), or if
\begin{equation} \label{eq:global-vanishing}
(t-1) \prod_{\lambda_2 \leqslant k < \lambda_1} \left( 1 - \frac{x_1}{y_k} \right) = (t-x_1/x_2)
   \prod_{\lambda_2 \leqslant k < \lambda_1} \left( 1 - \frac{x_2}{y_k} \right).
\end{equation}
If $x_1 = y_a$ and $x_2 = t^{-1}y_a$ for some $a$ with $\lambda_2 \leqslant a < \lambda_1$, both sides of~\eqref{eq:global-vanishing} vanish. 
In this case, a state vanishes if and only if there is a ``Case 2'' vertex (as enumerated in Figure~\ref{fig:modified_weights}) in row 1, column $a$; 
this means that a state survives if and only if the left path enters row 2 in column $m$ for some $m\ge a$. The weight of such a state in the 
specialization $x_1 = y_a$ and $x_2 = t^{-1}y_a$ is
\[
(t-1) \frac{y_a}{y_m}\prod_{k< \lambda_2} \left(1-\frac{y_a}{y_k}\right) \prod_{k<m} \left(1-\frac{t^{-1}y_a}{y_k}\right) \prod_{m<k< \lambda_1} \left(1-\frac{y_a}{y_k}\right)
\]
if $a\le m< \lambda_1$. If $m= \lambda_1$, the weight of the state is
\[
t\prod_{k< \lambda_2} \left(1-\frac{y_a}{y_k}\right) \prod_{k< \lambda_1} \left(1-\frac{t^{-1}y_a}{y_k}\right).
\]
Since~\eqref{eq:global-vanishing} is satisfied, the partition function vanishes. One can also show directly that the sum over states~\eqref{generic-partition-function} vanishes, as
a consequence of the vanishing of the series:
\begin{equation} \label{eq:telescoping-series}
-t \prod_{k<\lambda_1} \left(1-\frac{t^{-1}y_a}{y_k}\right) + \sum_{a\le m< \lambda_1} (t-1) \frac{y_a}{y_m}\prod_{k<m} \left(1-\frac{t^{-1}y_a}{y_k}\right) \prod_{m<k< \lambda_1} \left(1-\frac{y_a}{y_k}\right),
\end{equation}
%
and we will give an alternate proof of global vanishing by showing the vanishing of~\eqref{eq:telescoping-series} explicitly. We may express~\eqref{eq:telescoping-series} in terms of the monomials of the form
\[
(-1)^j\frac{y_a^j}{\prod_{k\in I} y_k}
\]
where $I\subseteq [1,\lambda_1)\setminus \{a\}$ has size $j$ for some positive integer $j$ and each coefficient of these monomials is in $\mathbb{Z}[t,t^{-1}]$. We show that each such coefficient (corresponding to all possible choices of sets $I$) vanishes. The coefficient in
\[
-t \prod_{k< \lambda_1} \left(1-\frac{t^{-1}y_a}{y_k}\right)
\]
is independent of the choice of $I$ and equal to $-(1-t^{-1})t^{1-j} = t^{-j}-t^{1-j}$. The summand $m=a$ in~\eqref{eq:telescoping-series} is
\[
(t-1) \frac{y_a}{y_a}\prod_{k<a} \left(1-\frac{t^{-1}y_a}{y_k}\right) \prod_{a<k< \lambda_1} \left(1-\frac{y_a}{y_k}\right)
\]
with coefficient
\[
(t-1) t^{-I_{[1,a)}} = t^{1-I_{[1,a)}} - t^{-I_{[1,a)}}
\]
and for the summand for $m$ with $a<m< \lambda_1$ in~\eqref{eq:telescoping-series}, the coefficient in
\[
(t-1) \frac{y_a}{y_m}\prod_{k<m} \left(1-\frac{t^{-1}y_a}{y_k}\right) \prod_{m<k< \lambda_1} \left(1-\frac{y_a}{y_k}\right)
\]
is 0 if $m\notin I$, and if $m\in I$ is
\[
-(t-1)(1-t^{-1}) t^{-I_{[1,m)}} = 2t^{-I_{[1,m)}} - t^{1-I_{[1,m)}} - t^{-1-I_{[1,m)}}
\]
Summing these terms and setting $b = |I_{[1,a)}|$, we get:
\[
t^{1-b} - t^{-b} + t^{-j}-t^{1-j} + \sum_{i=b+1}^j (2t^{1-i} - t^{2-i} - t^{-i}),
\]
which telescopes to 0.
\end{proof}

\section{Partition functions for motivic Chern classes\label{sec:motivic-chern}}

In this section, we show that motivic Chern classes of Schubert cells and their
duals can be written in terms of our partition functions. Motivic Chern classes
are examples of characteristic classes of a smooth algebraic variety. They
were defined by Brasselet, Sch\"urmann, and Yokura
\cite{BrasseletSchurmannYokura}. These were then generalized to double (or
equivariant) motivic Chern classes by Feh\'er, Rim\'anyi, and Weber
\cite{FeherRimanyiWeber}. This latter paper also proved a remarkable result:
that double motivic Chern classes satisfy the axioms for Maulik-Okounkov
$K$-theoretic stable envelopes (see Remark~\ref{rem:motivic-chern-lifts}). We
will discuss the relationship of the lattice model to stable envelopes in the
next section.

\subsection{Demazure operators for motivic Chern classes}

Aluffi, Mihalcea, Sch\"urmann and Su \cite{AMSSCasselman} give recursive formulas for double motivic Chern classes of Schubert cells, which we take as our definitions.
If $G=GL_n$, the equivariant K-theory of the flag variety has the following description.
\[ K_{T \times G_m} (\operatorname{Fl}_n) \cong\mathbb{Z} [t, t^{- 1}, x_1^{\pm}, \cdots,
   x_n^{\pm}, y_1^{\pm}, \cdots, y_n^{\pm}] / \left\langle f (\bs{x}) -
   f (\bs{y}) | \text{$f$ symmetric} \right\rangle \]
We will refer to this well-known isomorphism as the GKM presentation.
See Vezzosi and Vistoli~\cite{VezzosiVistoli} Corollary~5.12 and
Rosu and Knutson~\cite{RosuEquivariant} Corollary~A.5.

\begin{equation}
  \label{eq:zqy} \bigoplus_{w \in S_n} \mathbb{Z} [t, t^{- 1}, y_1^{\pm},
  \cdots, y_n^{\pm}]
\end{equation}
where 
\[
K_{T \times G_m} (\operatorname{Fl}_n) \hookrightarrow \bigoplus_{w \in S_n} \mathbb{Z} [t, t^{- 1}, y_1^{\pm},\cdots, y_n^{\pm}]
\]
via
\[
\varphi\mapsto (\varphi(\bs{y}_w))_{w\in S_n}.
\]

This map is an injection and identifies $K_{T \times G_m} (G / B)$ with the
subring of (\ref{eq:zqy}) where either
\begin{equation} \label{eq:GKM-divisibility}
\text{$y_i - y_j$ divides $f_w - f_{(i, j) w}$}
\end{equation}
or equivalently
\[ \text{$y_{w(i)} - y_{w(j)}$ divides $f_w - f_{w(i, j)} $} \]
We will now define some classes in 
$K_{T \times G_m} (\operatorname{Fl}_n)$, identified with (\ref{eq:zqy}), using
the notation $[x]_w$ to mean that the $w$-th component of the element defined
is given by the expression~$x$.

\begin{definition} \label{def:motivic-classes} \cite[Propositions~7.1,~7.2,~7.3]{AMSSCasselman}
 The {\it double motivic Chern class of the Schubert cell $X (w)^{\circ}$} is
  determined by the following properties:
\begin{itemize} 
\item $[\MC(X (w)^{\circ}))]_u = 0$ unless $u
  \leqslant w$,
\item  $[\MC(X (1_W)^{\circ})]_{1_W} = \prod_{\alpha \in \Phi^+} (1
     -\bs{y}^{\alpha}),$
  and 
  \item if $ws_i > w$, then
		\[ [\MC(X (ws_i)^{\circ})]_u = - \frac{1 - t}{1
     -\bs{y}^{- u \alpha_i}} [\MC(X (w)^{\circ})]_u +
     \frac{1 - t\bs{y}^{u \alpha_i}}{1 - \bs{y}^{- u \alpha_i}}
     [\MC(X (w)^{\circ})]_{us_i} . \]
 \end{itemize} 
  The {\it double motivic Chern class of the opposite Schubert cell $Y
  (w)^{\circ}$} is determined by:
 \begin{itemize}
 \item  $[\MC(Y (w)^{\circ})]_u = 0$ unless $u \geqslant w$,
 \item $[\MC(Y (w_0)^{\circ})]_{w_0} = \prod_{\alpha \in \Phi^+} (1
     -\bs{y}^{- \alpha}),$ and 
  \item if $ws_i < w$, then
  \[ [\MC(Y (ws_i)^{\circ})]_u = \frac{1 - t}{1 -\bs{y}^{-
     u \alpha_i}} [\MC(Y (w)^{\circ})]_u + \frac{1 -
     t\bs{y}^{u \alpha_i}}{1 -\bs{y}^{- u \alpha_i}} [\MC(Y (w)^{\circ})]_{us_i}. \]
\end{itemize}
The {\it dual double motivic Chern class of the Schubert cell $X(w)^\circ$} is determined by:
\begin{itemize}
\item $[\MC^\vee(X(w)^\circ)]_u = 0$ unless $u\leqslant w$, 
\item $[\MC^{\vee}(X(1_W)^\circ)]_{1_W}=\prod_{\alpha\in \Phi^+} (1-\bs{y}^\alpha),$
and 
\item if $ws_i>w$, then
\[[\MC^{\vee}(X(ws_i)^\circ)]_u=\frac{1-t^{-1}}{\bs{y}^{u\alpha_i}-1}[\MC^{\vee}(X(w)^\circ)]_u+\frac{-t^{-1}+\bs{y}^{-u\alpha_i}}{\bs{y}^{-u\alpha_i}-1}[\MC^{\vee}(X(w)^\circ)]_{us_i}. \]
\end{itemize}
The {\it dual double motivic Chern class of the opposite Schubert cell $Y(w)^\circ$} is given by:
\begin{itemize}
\item $[\MC^\vee(Y(w)^\circ)]_u = 0$ unless $u\geqslant w$, 
\item $[\MC^{\vee}(Y(w_0)^\circ)]_{w_0}=\prod_{\alpha\in \Phi^+} (1-\bs{y}^{-\alpha}),$
and 
\item if $ws_i<w$, then
\[[\MC^{\vee}(Y(ws_i)^\circ)]_u=\frac{1-t^{-1}}{\bs{y}^{u\alpha_i}-1}[\MC^{\vee}(Y(w)^\circ)]_u+\frac{-t^{-1}+\bs{y}^{-u\alpha_i}}{\bs{y}^{-u\alpha_i}-1}[\MC^{\vee}(Y(w)^\circ)]_{us_i}.\]
\end{itemize}
\end{definition}

\begin{remark}
Our notation has a couple of small differences to that of \cite{AMSSCasselman}. Our $t$ equals their $-y$, and they use the notation $e^{\alpha}$ for what we call $\bs{y}^\alpha$.
\end{remark}

By the geometric arguments in \cite{AMSSCasselman}, the double motivic Chern
classes in Definition~\ref{def:motivic-classes} lie in $K_{T \times G_m}
(\operatorname{Fl}_n)$, so satisfy \eqref{eq:GKM-divisibility}. Applying the
GKM isomorphism, there exist elements of $\mathbb{Z}[t^\pm, \bs{x}^\pm,
\bs{y}^\pm]$ whose restriction functions match
Definition~\ref{def:motivic-classes}. These lifts are nonunique; however, the
lifts we give in the next definition have nice combinatorial properties,
including a direct connection to our partition function. See
Remark~\ref{rem:motivic-chern-lifts}.

For each $i \in {1,\dots,n-1}$ define operators on $\mathbb{C}[t,t^{-1},\bs{x}^{\pm}, \bs{y^{\pm}}]$ by 
\begin{equation} T_i = \frac{-1+t}{1- \bs{x}^{-\alpha_i}} + \frac{1-t\bs{x}^{\alpha_i}}{1- \bs{x}^{-\alpha_i}} s_i 
\qquad  \text{and} \qquad T_i^\vee = \frac{1-t^{-1}}{\bs{x}^{\alpha_i}-1} + \frac{-t^{-1} + \bs{x}^{-\alpha_i}}{\bs{x}^{-\alpha_i}-1}s_i, 
\label{eq:tandthat}
\end{equation}
where, as before, the operators treat the variables $\bs{y}^{\pm}$ as constants.

\begin{definition} \label{def:doublemotivic}
	The \emph{lifted double (dual) motivic Chern classes} of $X(w)^\circ$ and $Y(w)^\circ$ are each determined as elements of $\mathbb{C}[t,t^{-1},\bs{x}^{\pm}, \bs{y^{\pm}}]$ by the following properties:
\begin{align*}
&\MC(X(1_W)^\circ) = \prod_{i<j} \left(1-\frac{x_i}{y_j}\right), \quad &\text{and if $ws_i>w$, then } \MC(X(ws_i)^\circ)= T_i\MC(X(w)^\circ).\\
&\MC(Y(w_0)^\circ) = \prod_{i+j\leqslant n} \left(1-\frac{x_i}{y_j}\right), &\text{and if $ws_i<w$, then } \MC(X(ws_i)^\circ)= T_i\MC(Y(w)^\circ).\\
&\MC^\vee(X(1_W)^\circ) = \prod_{i<j} \left(1-\frac{x_i}{y_j}\right), &\text{and if $ws_i>w$, then } \MC^\vee(X(ws_i)^\circ)= T_i^\vee \MC(X(w)^\circ).\\
&\MC^\vee(Y(w_0)^\circ) = \prod_{i+j\leqslant n} \left(1-\frac{x_i}{y_j}\right), &\text{and if $ws_i<w$, then } \MC^\vee(Y(ws_i)^\circ)= T_i^\vee \MC(Y(w)^\circ).
\end{align*}
\end{definition}

It is apparent from these definitions that
\begin{equation}
\MC(Y(w)^\circ)(\bs{x}; \bs{y}) = \MC(X(w_0w)^\circ)(\bs{x},\bs{y}_{w_0}) \label{eqn:MCX-MCY}
\end{equation}
and
\begin{equation}\MC^\vee(Y(w)^\circ)(\bs{x}; \bs{y}) = \MC^\vee(X(w_0w)^\circ)(\bs{x},\bs{y}_{w_0}) \label{eqn:MCXdual-MCYdual}.
\end{equation}

The next proposition says that the functions in Definition \ref{def:doublemotivic} really are lifts of the double motivic Chern classes and their duals, notated with square brackets.

\begin{proposition}
For all permutations $u$ and $w$,
\[
\MC(X(w)^\circ)|_u = [\MC(X(w)^{\circ}))]_u, \qquad \MC(Y(w)^\circ)|_u = [\MC(Y(w)^{\circ}))]_u,
\]
\[
\MC^\vee(X(w)^\circ)|_u = [\MC^\vee(X(w)^{\circ}))]_u, \qquad \MC^\vee(Y(w)^\circ)|_u = [\MC^\vee(Y(w)^{\circ}))]_u.
\]
\end{proposition}

\begin{remark}
This proposition can be taken as an alternate proof that the functions in Definition~\ref{def:motivic-classes} are indeed elements of $K_{T \times G_m} (\operatorname{Fl}_n)$.
\end{remark}

\begin{proof}
We prove the result for the case of $\MC(X(w)^\circ)$, and the others are similar.

We use induction on $w$.
If $w=1_W$, then
\[
\MC(X(1_W)^\circ)|_u = \prod_{i<j} \left(1 - \frac{y_{u(i)}}{y_j}\right).
\]
This is zero precisely when $u(i) = j$ for some $i<j$, which happens precisely when $u\ne 1_W$. If $u=1_W$, then we have
\[
\MC(X(1_W)^\circ)|_{1_W}  = \prod_{i<j} \left(1 - \frac{y_i}{y_j}\right) = \prod_{\alpha\in \Phi^+} (1- \bs{y}^\alpha) = [\MC(X(1_W)^\circ)]_{1_W}
\]
so $\MC(X(1_W)^\circ)|_u  = [\MC(X(1_W)^\circ)]_u$ for all $u$.

Next, assume that $w<ws_i$ and that $\MC(X(w)^\circ)|_u  = [\MC(X(w)^\circ)]_u$ for all $u$. We have
\[
\MC(X(ws_i)^\circ) = T_i\MC(X(w)^\circ) = \frac{-1+t}{1- \bs{x}^{-\alpha_i}}\MC(X(w)^\circ) + \frac{1-t\bs{x}^{\alpha_i}}{1- \bs{x}^{-\alpha_i}} (s_i \MC(X(w)^\circ),
\]
and restricting to $u$ gives
\begin{align*}
\MC(X(ws_i)^\circ)|_u 
&= \frac{-1+t}{1- \bs{y}^{-u\alpha_i}} \MC(X(w)^\circ)|_u  + \frac{1-t\bs{y}^{u\alpha_i}}{1- \bs{y}^{-u\alpha_i}} \MC(X(w)^\circ)|_{us_i}
\\&= \frac{-1+t}{1- \bs{y}^{-u\alpha_i}} [\MC(X(w)^\circ]_u + \frac{1-t\bs{y}^{u\alpha_i}}{1- \bs{y}^{-u\alpha_i}} [\MC(X(w)^\circ)]_{us_i}
\\&= [\MC(X(ws_i)^\circ)]_u.
\end{align*}

In particular, if $u\not\leqslant ws_i$, then $u\not\leqslant w$ and $us_i\not\leqslant w$, so $\MC(X(ws_i)^\circ)|_u $ is zero in this case by induction.
\end{proof}

\begin{remark} \label{rem:motivic-chern-lifts}
The choice of Definition~\ref{def:doublemotivic} as a lift for Definition~\ref{def:motivic-classes} is not unique. In \cite[Section~7]{FeherRimanyiWeber}, Feh\'er, Rim\'anyi, and Weber give another possible choice of a lift, which are the double motivic Chern classes of matrix Schubert cells. Their lifts appear to be different from ours. Definition~\ref{def:doublemotivic} has the benefit of good factorization, and connection with the other functions of this paper.  See Section~\ref{sec:KZJ} for a comparison with a lift of \emph{motivic Segre classes} and a lattice model of Knutson and Zinn-Justin \cite{KnutsonZinnJustinMotivic}.
\end{remark}

Next, we express the lifted double motivic Chern classes in terms of our partition functions, starting with the dual functions.
%
%
%
\begin{theorem} \label{thm:dual-motivic-part-fun}
For all permutations $w$,
\[
\MC^\vee(X(w)^\circ) = (-t)^{-\ell(w)} Z^\square_{1_W,w_0w}(\bs{x}; \bs{y}_{w_0}),
\]
and
\[
\MC^\vee(Y(w)^\circ) = (-t)^{-\ell(w^{-1}w_0)}Z^\square_{1_W,w}(\bs{x}; \bs{y}).
\]
\end{theorem}
\begin{proof}
Both base cases $\MC^\vee(X(1)^\circ) = Z^\square_{1_W,w_0}(\bs{x}; \bs{y}_{w_0})$ and $\MC^\vee(Y(w_0)^\circ) = Z^\square_{1_W,w_0}(\bs{x}; \bs{y})$ follow from Proposition~\ref{prop:basecase} and Definition~\ref{def:doublemotivic}.

The inductive steps follow because $T_i^\vee = -t^{-1}\tau_i$, which is easily checked from the definition of the 
operators $\tau_i$ and $T_i^\vee$ in~(\ref{eq:tauop}) and~(\ref{eq:tandthat}) respectively. In the first equation in the statement of the theorem, applying this 
operator is length-increasing on the left side and length-decreasing on the right side, whereas it is length-decreasing 
on both sides of the second equation.
\end{proof}

Next, we move on to the lifted double ordinary motivic Chern classes, and prove the following relationship between these and the dual classes. Let $\bs{x}^{-1} = (x_1^{-1},\ldots,x_n^{-1})$ and let
\[
r(\bs{x};\bs{y}) := (-1)^{\ell(w_0)} \bs{x}^\rho \bs{y}^{-\rho} = (-1)^{\ell(w_0)}\frac{x_1^{n-1}x_2^{n-2}\cdots x_{n-1}}{y_1^{n-1} y_2^{n-2}\cdots y_{n-1}}.
\]

\begin{proposition} \label{prop:ordinary-dual-motivic-relation} For all permutations $w \in W$,
\[
\MC(X(w)^\circ)(\bs{x}; \bs{y}) = (-t)^{\ell(w)} r(\bs{x},\bs{y}_{w_0}) \MC^\vee(X(w)^\circ)(\bs{x}^{-1}; \bs{y}^{-1}) \] and \[\MC(Y(w)^\circ)(\bs{x}; \bs{y}) = (-t)^{\ell(w_0w)} r(\bs{x},\bs{y}) \MC^\vee(Y(w)^\circ)(\bs{x}^{-1}; \bs{y}^{-1}).
\]
\end{proposition}

This result is a lifted version, on the level of functions, of the geometric duality as in~\cite[Theorem~9.1(a)]{AMSSHirzebruch}.

Recall that $\theta$ is the involution which sends $f(\bs{x}; \bs{y}) \mapsto f(\bs{x}^{-1};\bs{y}^{-1})$. Thus $\theta$ commutes with the action of $s_i$ and moreover,
\[
\bs{x}^{-\rho} s_i\cdot \bs{x}^{\rho} = s_i\cdot \bs{x}^{\alpha_i} = \bs{x}^{-\alpha_i}\qquad \text{and} \qquad \theta \bs{x}^{\alpha_i}\theta = \bs{x}^{-\alpha_i}.
\]
One easily verifies that $r(\bs{x},\bs{y})$
satisfies the following properties:
\begin{equation}
r(\bs{x};\bs{x}) = (-1)^{\ell(w_0)}, \qquad r(\bs{x};\bs{y})r(\bs{y};\bs{z}) = (-1)^{\ell(w_0)} r(\bs{x};\bs{z}),
\label{eq:rinprods} \end{equation}
\begin{equation}
\theta r(\bs{x};\bs{y}) = r(\bs{y};\bs{x}) = (r(\bs{x};\bs{y}))^{-1}, \qquad s_i\cdot r(\bs{x},\bs{y}) = \bs{x}^{-\alpha_i} r(\bs{x},\bs{y}),
\label{eq:rundertheta} \end{equation}
\begin{equation}
r(\bs{x};\bs{y}) \prod_{i+j\leqslant n} \left(1-\frac{y_j}{x_i}\right) = \prod_{i+j\leqslant n} \left(1-\frac{x_i}{y_j}\right), \quad r(\bs{x},\bs{y}_{w_0}) \prod_{i<j} \left(1-\frac{y_j}{x_i}\right) = \prod_{i<j} \left(1-\frac{x_i}{y_j}\right).
\label{eq:rwithbasecase}
\end{equation}
%
%
\begin{proof}[Proof of Proposition~\ref{prop:ordinary-dual-motivic-relation}]
We use induction on $w$. For the first equation, when $w=1$, \[r(\bs{x},\bs{y}_{w_0}) \theta \MC^\vee(X(1)^\circ) = r(\bs{x},\bs{y}_{w_0}) \prod_{i<j} \left(1-\frac{y_j}{x_i}\right) = \prod_{i<j} \left(1-\frac{x_i}{y_j}\right).\]
If the first equation holds for $w$ and $w<ws_i$, then
\begin{align*}
\MC(X(ws_i)^\circ)
&= T_i \MC(X(w)^\circ) \\&= (-t)^{\ell(w)} T_i r(\bs{x},\bs{y}_{w_0}) \theta \MC^\vee(X(w)^\circ)
\\&= (-t)^{\ell(w)} (-t)\theta \bs{x}^{-\rho}T_i^\vee \bs{x}^{\rho} \theta r(\bs{x},\bs{y}_{w_0}) \theta \MC^\vee(X(w)^\circ)
\\&= (-t)^{\ell(w)+1} r(\bs{x};\bs{y}_{w_0}) \theta T_i^\vee \MC^\vee(X(w)^\circ)
\\&= (-t)^{\ell(ws_i)} r(\bs{x};\bs{y}_{w_0}) \theta \MC^\vee(X(ws_i)^\circ),
\end{align*}
where the third equality follows from the identity of operators $T_i = -t\theta \bs{x}^{-\rho}T_i^\vee \bs{x}^{\rho}\theta$ and the fourth equality is because $\bs{x}^\rho \theta r(\bs{x};\bs{y}_{w_0}) = \bs{x}^\rho r(\bs{y}_{w_0}; \bs{x})\theta $, and $\bs{x}^\rho r(\bs{y}_{w_0}; \bs{x})$ is symmetric in $\bs{x}$, so commutes with $T_i^\vee$.

The second equation follows from the first, along with (\ref{eqn:MCX-MCY}, \ref{eqn:MCXdual-MCYdual}).
\end{proof}

\begin{figure}[h]
\begin{center}
\begin{tabular}{|ccc|}
\hline
\hline
Case 1 & Case 2 & Case 3\\
\hline
\vertex{+}{\Sigma}{+}{\Sigma} &
\vertex{c}{\Sigma}{c}{\Sigma} &
\vertex{+}{\Sigma}{c}{\Sigma^-_c}
\\[4pt]
$(-tx)^{\Sigma_{[1,r]}}$ & 
$t^{\Sigma_{(c,r]}} x^{\Sigma_{[1,r]}}(-x^{-1} + t^{\Sigma_c}y^{-1})$ &
$(-t)^{\Sigma_{[1,c)}} t^{\Sigma_{(c,r]}} x^{\Sigma_{[1,r]}-1}$ \\[4pt]
\hline
\hline
Case 4 & Case 5 & Case 6\\
\hline
\vertex{c}{\Sigma}{+}{\Sigma^+_c} &
\vertex{c}{\Sigma}{d}{\Sigma^{+-}_{cd}} &
\vertex{d}{\Sigma}{c}{\Sigma^{+-}_{dc}}
\\    
$(t-1) x^{-1} (-t)^{\Sigma_{(c,r]}} x^{\Sigma_{[1,r]}+1}$ &
$(t-1) x^{-1} (-t)^{\Sigma_{(c,d)}} t^{\Sigma_{(d,r]}} x^{\Sigma_{[1,r]}}$ &
$(t-1) y^{-1} (-1)^{\Sigma_{(c,d)}} t^{\Sigma_{(d,r]}} x^{\Sigma_{[1,r]}}$ \\
\hline
\end{tabular}
\end{center}
	\caption{A dual set of Boltzmann weights for a rectangular vertex with row parameter $x$ and column parameter $y$. Throughout, $c<d$.}
\label{fig:dual-rect-weights}
\end{figure}

Consider next the set of weights in Figure \ref{fig:dual-rect-weights}. They are obtained from the weights in Figure \ref{fig:modified_weights} by inverting the parameters $x$ and $y$ and multiplying each Boltzmann weight by $x^b y^{-l}$, where $b$ is the number of colors on the bottom edge and $l$ is the number of colors on the left edge.

Let $\widetilde{Z}^\square_{\sigma,w}$ denote the partition function of the lattice model on the $n \times n$ square with boundary conditions as in Definition~\ref{def:boundary}
and with Boltzmann weights as in Figure~\ref{fig:dual-rect-weights}.

\begin{proposition} \label{prop:dual-partition-function}
\[\widetilde{Z}^\square_{\sigma,w}(\bs{x};\bs{y}) = r(\bs{x};\bs{y})Z^\square_{\sigma,w}(\bs{x}^{-1};\bs{y}^{-1}).\]
\end{proposition}

\begin{proof}
For any state, the total number of paths on the bottom edge of vertices in row $i$ is $n-i$, and the total number of paths on the left edge of vertices in column $i$ is $n-i$. Therefore, the result follows from the relationship between the Boltzmann weights in Figures \ref{fig:modified_weights} and \ref{fig:dual-rect-weights}.
\end{proof}

This allows us to represent the lifted double motivic Chern classes as partition functions.

\begin{corollary} \label{cor:motivic-part-fun} For all $w \in W$,
\[\MC(X(w)^\circ) = \widetilde{Z}^\square_{1_W,w_0w}(\bs{x}; \bs{y}_{w_0}) \qquad \text{and} \qquad \MC(Y(w)^\circ) = \widetilde{Z}^\square_{1_W,w}(\bs{x}; \bs{y}).\]
\end{corollary}

\begin{proof}
Combining Theorem \ref{thm:dual-motivic-part-fun} and Proposition \ref{prop:dual-partition-function}, we have
\begin{align*}
\MC(X(w)^\circ)
&= (-t)^{\ell(w)} r(\bs{x},\bs{y}_{w_0}) \theta \MC^\vee(X(w)^\circ)
\\&= (-t)^{\ell(w)} r(\bs{x},\bs{y}_{w_0}) \theta (-t)^{-\ell(w)} Z^\square_{1_W,w_0w}(\bs{x}; \bs{y}_{w_0})
\\&= \widetilde{Z}^\square_{1_W,w_0w}(\bs{x}; \bs{y}_{w_0}).
\end{align*}

Similarly, \begin{align*} \MC(Y(w)^\circ) &= (-t)^{\ell(w_0w)} r(\bs{x},\bs{y}) \theta \MC^\vee(Y(w)^\circ) \\&= (-t)^{\ell(w_0w)} r(\bs{x},\bs{y}) \theta (-t)^{-\ell(w^{-1}w_0)}Z^\square_{1_W,w}(\bs{x}; \bs{y}) \\&= \widetilde{Z}^\square_{1_W,w}(\bs{x}; \bs{y}). \qedhere \end{align*}
\end{proof}

\subsection{Comparison with Knutson and Zinn-Justin} \label{sec:KZJ}

Knutson and Zinn-Justin \cite{KnutsonZinnJustinMotivic} study similar geometric
objects via a different solvable lattice model. Their partition functions are
the \emph{motivic Segre classes} of Schubert cells, a renormalization of the
motivic Chern classes \cite[Remark~4.6]{AMSSCasselman}.

Both partition functions involve two sets of variables which are subject to
specialization to obtain motivic Chern or Segre classes. In their two-variable
forms these partition functions are very different.  Our double polynomials
(\ref{def:doublemotivic}) are quite different from theirs in equations (5) and
(7) of \cite{KnutsonZinnJustinMotivic}.  The difference comes because their
lattice model is associated to the standard $R$-matrix for
$U_q(\widehat{\mathfrak{gl}}(n+1))$, while the horizontal edges of our lattice
model are related to the standard module for $U_q({\mathfrak{gl}}(n|1))$, and
the vertical edges are related to $2^n$ dimensional Kac modules (or from a
different viewpoint to limits of Kirillov-Reshetikhin modules for
$U_q(\widehat{\mathfrak{gl}}(n|1))$). Yet the partition functions become
closely related upon restriction. 

For a permutation $w\in S_n$, let $Z^{\text{KZJ}}_w$ be the partition function with boundary conditions $B(1_W,w,\square)$ and weights given in Figure \ref{fig:KZJ-weights}.
The functions $Z^{\text{KZJ}}_w$ are the partition functions $S^w$ from equation (5) of \cite{KnutsonZinnJustinMotivic}.

\begin{remark} \label{rem:KZJ-conventions}
The weights in Figure \ref{fig:KZJ-weights} are the weights from equation (1) of \cite{KnutsonZinnJustinMotivic}, with the following changes. First, we have rotated their model $90^\circ$ clockwise to match the conventions of our model. Second, we have replaced their parameters $z'$ and $z''$ with $x$ and $y$, respectively.
\end{remark}

\begin{figure}[h] 
\begin{center}
\begin{tabular}{|ccccc|}
\hline
\hline
Case 1 & Case 2 & Case 3 & Case 4 & Case 5 \\
\hline
\vertex{a}{a}{a}{a} &
\vertex{d}{c}{d}{c} &
\vertex{c}{d}{c}{d} &
\vertex{d}{c}{c}{d} &
\vertex{c}{d}{d}{c}
\\[4pt]
$1$ & 
$\frac{q(1-yx^{-1})}{1-q^2 yx^{-1}}$ &
$\frac{q(1-yx^{-1})}{1-q^2 yx^{-1}}$ &
$\frac{1-q^2}{1-q^2 yx^{-1}}$ &
$\frac{(1-q^2)yx^{-1}}{1-q^2 yx^{-1}}$  \\
\hline

\hline
\end{tabular}
\end{center}
\caption{The Boltzmann weights from \cite{KnutsonZinnJustinMotivic}, with convention changes describe in Remark \ref{rem:KZJ-conventions}. Here, $a$ is any color, while $c$ and $d$ are distinct colors with $c<d$. For these weights, colors include the empty color $+$. Any weights which do not appear in this table are 0.}
\label{fig:KZJ-weights}
\end{figure}

 A similar argument to the one given in our Section \ref{sec:partition-function} or \cite[Section~2]{KnutsonZinnJustinMotivic} shows that
 
\[Z^{\text{KZJ}}_{w_0} = q^{\ell(w_0)} \prod_{i+j\leqslant n} \left(1-\frac{y_j}{x_i}\right) \cdot (1-q)^n \cdot \prod_{i+j\leqslant n+1} \left(1 - q^2\frac{y_j}{x_i}\right)^{-1}\]
and 
\[ Z^{\text{KZJ}}_{ws_i} = \tau^{\text{KZJ}}_i Z^{\text{KZJ}}_w, \qquad \text{if }ws_i < w.\]
where
\[ \tau^{\text{KZJ}}_i  = \frac{(1-q^2) x^{\alpha_i}}{1 - q^2 x^{\alpha_i}} + \frac{q(1-x^{\alpha_i})}{1 - q^2 x^{\alpha_i}} s_i.\] 

These functions are substantially different from (\ref{def:doublemotivic}). The largest difference is in the seed of the recursion, $w=w_0$ , where their last factor in particular is very different from our corresponding functions (\ref{def:doublemotivic}). However, the restriction functions turn out to be quite similar.

\begin{example}
Let $w = s_1$, and $v = s_2s_1$, and set $t=q^2$. We have
\[
Z^\square_{1_W,w}|_v = (1- q^2 ) \left(1 - \frac{y_3}{y_1}\right)\left(1 - q^2\frac{y_2}{y_1}\right),
\]
while
\[
Z^{KZJ}_w|_v =   \frac{q (q - 1)  (1 - \frac{y_3}{y_2}) (1-\frac{y_1}{y_3}) (1 - \frac{y_1}{y_2} )^2}{(q + 1)^2(1-q^2\frac{y_1}{y_2})^2 (1-q^2\frac{y_3}{y_2}) (1-q^2\frac{y_2}{y_3} ) (1-q^2\frac{y_1}{y_3})}.
\]
So these restriction functions are equal up to a normalization factor of
\[
\frac{(q + 1)^3 y_3  (1-q^2\frac{y_3}{y_2}) (1-q^2\frac{y_2}{y_3}) (1-q^2\frac{y_2}{y_1} ) (1-q^2\frac{y_1}{y_3}) (1-q^2\frac{y_1}{y_2})^2}{ qy_1(1-\frac{y_1}{y_2})^2(1-\frac{y_3}{y_2})}.
\]

However, the double functions are much more different. We have
\[
Z^\square_{1_W,w} = \left(1-\frac{x_1}{y_1}\right) \left(q^4\frac{x_2x_3}{y_1y_2} + q^4\frac{x_3^2}{y_1y_2} - q^4\frac{x_3}{y_2} - q^2\frac{x_2x_3}{y_1y_2} - q^2\frac{x_3}{y_1} + 1\right),
\]
while $Z^{KZJ}_w$ even fully factored has 18 distinct factors, one of which has 581 distinct terms.

In both cases, passing from the double functions to the restriction functions we see a miraculous simplification, but this is true to a much larger extent for the functions in \cite{KnutsonZinnJustinMotivic}. This appears to be the general situation.
\label{ex:kzjcomparison} \end{example}

\section{Partition functions for K-theoretic stable envelopes\label{sec:stable}}

In the previous section, we demonstrated that our partition functions on the lattice of size $n \times n$ with 
prescribed boundary conditions match polynomial representatives for  (specializations of) motivic Chern classes 
of various Schubert cells in $G = GL_n$. Moreover, we were led by the lattice model and our natural family of 
Boltzmann weights in row parameters $\boldsymbol{x}$ and column parameters $\boldsymbol{y}$ to define 
natural two-variable lifts of these polynomial representatives. In the present section, we repeat this process for 
localizations of the K-theoretic stable envelopes~\cite{OkounkovEnumerative, MaulikOkounkov}, which we refer to as ``stab polynomials'' for brevity.

These K-theoretic stable envelopes $\{ \textrm{stab}_X(w) \; | \; w \in W \}$ form a distinguished basis of the 
localized equivariant K-theory of $T^\ast (G/B)$, the cotangent bundle of the flag variety, where 
$X = (\mathfrak{C}, T^{1/2}, \mathcal{L})$ is a chosen triple of data  consisting of a chamber $\mathfrak{C}$ in 
the Lie algebra of the maximal torus, a polarization $T^{1/2}$ in the equivariant K-theory of the cotangent bundle, 
and a certain line bundle $\mathcal{L}$ in the equivariant Picard group of the cotangent bundle. Our notation here 
matches~\cite{MaulikOkounkov, OkounkovEnumerative} where these were first defined, and we refer the reader to that source for further 
details. Roughly speaking the choice of the triple $X$ identifies a relatively small support set for these K-theoretic 
classes and that the set of weights appearing in a certain set of torus fixed points is appropriately small. Maulik and 
Okounkov \cite{MaulikOkounkov} show that these classes are uniquely characterized by several axioms on the 
support of their localizations.

In~\cite{SuZhaoZhongKStable}, Su, Zhao, and Zhong identify certain natural families of classes 
$\{ \textrm{stab}_{\pm}(w) \}$ for $w \in W$ corresponding to certain triples of data. The class $\textrm{stab}_+$ 
corresponds to the positive Weyl chamber for which all roots are positive (and other natural choices of polarization 
and line bundle made explicit in Section~3 of~\cite{SuZhaoZhongKStable}). Similarly the family $\textrm{stab}_-$ 
depends on the opposite (or negative) Weyl chamber (and other natural choices). A main theorem 
of~\cite{SuZhaoZhongKStable} (Theorem 3.5) shows that members of each family $\textrm{stab}_\pm$ are 
related by Demazure-Lusztig operators, so it is perhaps not surprising that they might arise as special cases 
of partition functions of our lattice models.

\begin{remark}
	\label{rem:qiswhat}
  In most of this paper, the deformation parameter is denoted $t$. In
	Sections~{\ref{sec:stable}} and {\ref{sec:denom}} we will use the letter $q$ to denote
  $t$ or $t^{- 1}$ in order for convenient comparison with other literature.
  Thus in this section, $t = q^{- 1}$ and our results may be compared with
  those in {\cite{AMSSCasselman}}. On the other hand in
  Section~\ref{sec:denom} we take $t = q$ for better comparison with~\cite{BumpNakasujiKL}.
\end{remark}

If we take localizations of stable envelope classes, these take values in 
$K_{T \times \mathbb{C}^\ast}(\mathbb{C}) \simeq \mathbb{Z}[q^{1/2}, q^{-1/2}][\Lambda]$ where $T$ is the maximal 
torus with weight lattice $\Lambda$. Here $q^{1/2}$ denotes the standard representation of $\mathbb{C}^\ast$. Further 
identifying the weight lattice $\Lambda$ with monomials in a single set of variables $\boldsymbol{y}^{\pm 1} = (y_1^{\pm 1}, \ldots, y_n^{\pm 1})$, 
then the localizations of stable envelope classes are polynomials in $\boldsymbol{y}^{\pm 1}$, the so-called ``stab polynomials.''

In Section~5 of~\cite{FeherRimanyiWeber}, the stable envelopes $\textrm{stab}_+(w)$ and $\textrm{stab}_-(w)$ are related to the motivic Chern class $MC_{-q^{-1}}(X(w)^\circ)$ and $MC_{-q^{-1}}(Y(w)^\circ)$, respectively. Their proof relies on verifiying that these classes satisfy the Maulik-Okounkov axioms that uniquely characterize these Laurent polynomials. An alternate proof using tools and notation more similar to ours was later given in Section 8 of~\cite{AMSSCasselman} via comparison of localization properties.

To demonstrate that our partition functions match stab polynomials, we thus have several proof strategies available. 
We may use the axiomatic characterization of \cite{MaulikOkounkov} applied to our partition functions, we may connect 
partition functions to stable envelopes by combining the results of the prior section with Theorems 8.5 and 8.6 
in~\cite{AMSSCasselman}, or we may employ the Demazure-Lusztig recursions of \cite{SuZhaoZhongKStable}. We choose 
the latter method, as it very naturally leads to the two-variable lifting of stable envelope localizations we propose later in this 
section. Thus the proofs bear strong resemblance to those of the prior section, but each case of the recursion is an intricate 
and subtle calculation, so full details are provided.

We begin by recalling the explicit formulas for stab polynomials $\textrm{stab}_{\pm}$, those for the positive and negative 
chambers, from~\cite{SuZhaoZhongKStable} and Section 8 of~\cite{AMSSCasselman}. Given any triple of data $X$ and 
any pair of elements $u, w \in W$, we write $[\stab_X(w)]_u$ for the localization of $\stab_X(w)$ at~$u$.
To translate from the notation of~\cite{AMSSCasselman}, simply replace their $e^{\alpha_i}$ with our $\boldsymbol{y}^{\alpha_i}$.
While these are given as propositions in~\cite{SuZhaoZhongKStable, AMSSCasselman}, whose proof relies on ensuring that
the resulting polynomials satisfy the criteria of \cite{MaulikOkounkov} that uniquely determine them, we may take this as the
definition of these specializations.

\begin{definition}[Propositions 8.3, 8.4 in~\cite{AMSSCasselman}] \label{def:stab-functions} 
The localizations of the positive-chamber K-theoretic stable envelope $\stab_+$ are determined by the following properties:
\begin{enumerate}
\item For any pair $u, w \in W$, then $[\stab_+(w)]_u = 0$ unless $u \leqslant w$.
\item For every element $w \in W$, 
\[ [\stab_+(w)]_{w}= q^{\ell(w) / 2} \prod_{\substack{\alpha\in \Phi^+ \\ w(\alpha) \in \Phi^-}}(1- q^{-1} \boldsymbol{y}^{w(\alpha)}) \prod_{\substack{\alpha\in \Phi^+ \\ w(\alpha) \in \Phi^+}}(1-\boldsymbol{y}^{w(\alpha)}). \]
\item For any $w \in W$ and any simply reflection $s_i$ with $ws_i>w$,
\[[\stab_+(ws_i)]_u=\frac{q^{1/2}-q^{-1/2}}{1-\boldsymbol{y}^{u\alpha_i}}[\stab_+(w)]_u - 
\frac{q^{-1/2}\boldsymbol{y}^{u\alpha_i}-q^{1/2}}{1-\boldsymbol{y}^{-u\alpha_i}}[\stab_+(w)]_{us_i}.\]
\end{enumerate}
The localizations of the negative-chamber K-theoretic stable envelope $\stab_-$ are determined by the following properties:
\begin{enumerate}
\item For any pair $u, w \in W$, then $[\stab_-(w)]_u = 0$ unless $u \geqslant w$.
\item For any element $w \in W$, 
\[[\stab_-(w)]_{w}
= q^{\ell(w)/2} \prod_{\substack{\alpha\in \Phi^+ \\ w(\alpha) \in \Phi^-}}(1- \boldsymbol{y}^{-w(\alpha)}) \prod_{\substack{\alpha\in \Phi^+ \\ w(\alpha) \in \Phi^+}}(1-q \,\boldsymbol{y}^{-w(\alpha)}).\] 
\item For any $w \in W$ and any simple reflection $s_i$ such that $ws_i<w$,
 \[[\stab_-(ws_i)]_u=\frac{q^{-1/2}-q^{1/2}}{1-\boldsymbol{y}^{-u\alpha_i}}[\stab_-(w)]_u+
 \frac{q^{-1/2}-q^{1/2}\boldsymbol{y}^{-u\alpha_i}}{1-\boldsymbol{y}^{u\alpha_i}}[\stab_-(w)]_{us_i}.\]
\end{enumerate}
\end{definition}
Note there's some redundancy built into the previous statements, as we need only specify the ``diagonal terms'' $[\stab_+(e)]_e$ and $[\stab_{-}(w_0)]_{w_0}$.in each of the respective part (2)'s above to uniquely determine $\stab_\pm$. We'll make use of this in our next definition.

As with the double motivic Chern classes in the previous section, geometric arguments
of \cite{SuZhaoZhongKStable} show that the stable basis
classes in Definition~\ref{def:stab-functions} lie in $K_{T \times G_m}
(\operatorname{Fl}_n)$, so satisfy \eqref{eq:GKM-divisibility}. 
Applying the GKM isomorphism, there exist elements of $\mathbb{Z}[q^{\pm 1/2}, \bs{x}^{\pm 1},
\bs{y}^{\pm 1}]$ whose restriction functions match
Definition~\ref{def:stab-functions}, and the next definition gives such lifts.
First we define the operators $M_i^{\pm}$ on functions  
$f(\boldsymbol{x}; \boldsymbol{y}) \in \mathbb{Z}[q^{\pm 1/2}, \bs{x}^{\pm 1},
\bs{y}^{\pm 1}]$:
\begin{equation}
M_i^+ (f(\boldsymbol{x};\boldsymbol{y})):= \frac{q^{1/2}-q^{-1/2}}{1-\boldsymbol{x}^{\alpha_i}} 
f(\boldsymbol{x};\boldsymbol{y}) - \frac{q^{-1/2}\boldsymbol{x}^{\alpha_i}-q^{1/2}}{1-\boldsymbol{x}^{-\alpha_i}} 
f(\boldsymbol{x}^{s_i};\boldsymbol{y})
\label{eq:mplus} \end{equation}
and
\begin{equation}
M_i^- (f(\boldsymbol{x};\boldsymbol{y}))  := \frac{q^{-1/2}-q^{1/2}}{1-\boldsymbol{x}^{-\alpha_i}} 
f(\boldsymbol{x};\boldsymbol{y}) + \frac{q^{-1/2}-q^{1/2}\boldsymbol{x}^{-\alpha_i}}{1-\boldsymbol{x}^{\alpha_i}} 
f(\boldsymbol{x}^{s_i};\boldsymbol{y}).
\label{eq:mminus} \end{equation}
As usual $\boldsymbol{x}^{s_i}$ denotes the permutation action of the simple 
reflection $s_i$ on the variables, interchanging $x_i$ and $x_{i+1}$.

\begin{definition} \label{def:doublestab}
The \emph{lifted K-theoretic stable envelopes} $\stab_{\pm}$ are determined 
by the following set of conditions. At the identity element $e \in W$,
$$ \stab_+(e) = \prod_{i<j} \left(1-\frac{x_i}{y_j}\right), $$
and for $w \in W$ and simple reflection $s_i$ such that $ws_i>w$, then 
$$ \stab_+(ws_i) = M_i^+ \stab_-(w). $$
Similarly at the long element $w_0 \in W$, we set
\[\stab_-(w_0) = q^{\ell(w_0)/2} r(\boldsymbol{x}_{w_0};\boldsymbol{x})\prod_{i+j\leqslant n} \left(1-\frac{x_i}{y_j}\right),\] 
and for any $w \in W$ and simple reflection $s_i$ such that $ws_i<w$, then 
\[\stab_-(ws_i) = M_i^- \stab_-(w).\]
\end{definition}

Note the absence of brackets, e.g., $\stab_{\pm} (w)$ versus $[\stab_{\pm}
(w)]_u$, distinguishes the lifted two variable polynomials from their
localization in one variable set $\boldsymbol{y}$. The next result justifies
this notation choice and the validity of the term `lift' by matching the
specialization to the localization. Again for $f \in \mathbb{Z} [q^{\pm 1 /
2}, \boldsymbol{x}^{\pm 1}, \boldsymbol{y}^{\pm 1}]$ we define the specialization
at $u \in W$ by:
\[ f (\boldsymbol{x}; \boldsymbol{y}) |_u := f (\boldsymbol{y}_u ;
   \boldsymbol{y}) = f (y_{u (1)}, \ldots, y_{u (n)} ; y_1, \ldots, y_n) . \]
\begin{proposition}
  For all permutations $u$ and $w$,
  \[ \stab_+ (w) |_u = [\stab_+ (w)]_u \qquad \text{and} \qquad \stab_- (w)
     |_u = [\stab_- (w)]_u . \]
\end{proposition}

\begin{remark}
  As in the previous section, this proposition can be taken as an alternate
  proof that the functions in Definition~\ref{def:doublestab} are indeed
  elements of $K_{T \times G_m} (\mathrm{Fl}_n)$.
\end{remark}

\begin{proof}
  We prove the result for the case of $\stab_- (w)$, and the other case is
  similar. We use induction downward in Bruhat order beginning with the long
  element $w_0$. If $w = w_0$, then
  \[ \stab_- (w_0) |_u = q^{\ell (w_0) / 2} r (\boldsymbol{y}_{u w_0} ;
     \boldsymbol{y}_u)  \prod_{i + j \leqslant n} \left( 1 - \frac{y_{u
     (i)}}{y_j} \right) . \]
  This is zero if $u (i) = j$ for some $n - i \geqslant j$, or equivalently if
  $u (n + 1 - i) = j$ for some $i > j$, which happens whenever $u \ne w_0$. If
  $u = w_0$, then we have
  \[ \stab_- (w_0) |_{w_0} = q^{\ell (w_0) / 2} r (\boldsymbol{y};
     \boldsymbol{y}_{w_0})  \prod_{j < i} \left( 1 - \frac{y_i}{y_j} \right) =
     q^{\ell (w_0) / 2}  \prod_{i < j} \left( 1 - \frac{y_i}{y_j} \right) =
     [\stab_- (w_0)]_{w_0} . \]
  so $\stab_- (w_0) |_u = [\stab_- (w_0)]_u$ for all~$u$.
  
  Next, assume that $w > ws_i$ and that $\stab_- (w) |_u = [\stab_- (w)]_u$
  for all $u$. Then
  \[ \stab_- (ws_i) = M_i^- \stab_- (w) = \frac{q^{- 1 / 2} - q^{1 / 2}}{1
     -\boldsymbol{x}^{- \alpha_i}} \stab_- (w) + \frac{q^{- 1 / 2} - q^{1 / 2}
     \boldsymbol{x}^{- \alpha_i}}{1 -\boldsymbol{x}^{\alpha_i}} (s_i \stab_- (w)),
  \]
  and restricting to $u$ gives
  
  \begin{align*}
    \stab_- (ws_i) |_u & = \frac{q^{- 1 / 2} - q^{1 / 2}}{1 -\boldsymbol{y}^{- u
    \alpha_i}} \stab_- (w)  |_u + \frac{q^{- 1 / 2} - q^{1 / 2}
    \boldsymbol{y}^{- u \alpha_i}}{1 -\boldsymbol{y}^{u \alpha_i}} \stab_- (w)
    |_{us_i}\\
    & = \frac{q^{- 1 / 2} - q^{1 / 2}}{1 -\boldsymbol{y}^{- u \alpha_i}}
    [\stab_- (w)]_u + \frac{q^{- 1 / 2} - q^{1 / 2} \boldsymbol{y}^{- u
    \alpha_i}}{1 -\boldsymbol{y}^{u \alpha_i}} [\stab_- (w)]_{us_i}\\
    & = [\stab_- (ws_i)]_u .
  \end{align*}
  
  In particular, if $u \ngeqslant ws_i$, then $u \ngeqslant w$ and $us_i
  \ngeqslant w$, so $\stab_- (ws_i) |_u$ is zero in this case by induction.
\end{proof}

Next we connect the recursion operators $M_i^{\pm}$ defined in \eqref{eq:mplus} and \eqref{eq:mminus} to the prior recursive operators 
on lattice model partition functions $\tau_i$ from \eqref{eq:tauop}. Recall the operator $\overline{\tau}_i = \theta\tau_i\theta$ from \eqref{eq:tau-bar}.

\begin{proposition} \label{prop:stab-op-comp} For any $i \in [1,n]$,
\[M_i^- = q^{1/2} \boldsymbol{x}^{-\rho}\overline{\tau}_i \boldsymbol{x}^{\rho}|_{t=q^{-1}}, \qquad \text{and} \qquad M_i^+ = q^{-1/2} \boldsymbol{x}^{\rho} \overline{\tau}_i^{-1} \boldsymbol{x}^{-\rho} |_{t=q^{-1}}.\]
\end{proposition}

\begin{proof}
We prove this for $M_i^-$ and the result for $M_i^+$ follows similarly. Since
$\boldsymbol{x}^{-\rho} s_i \boldsymbol{x}^{\rho}= \boldsymbol{x}^{-\alpha_i}s_i$, 
\[ \boldsymbol{x}^{-\rho}\overline{\tau}_i \boldsymbol{x}^{\rho}|_{t=q^{-1}} = \frac{q^{-1}-1}{1-\bs{x}^{-\alpha_i}} + \frac{\bs{x}^{-\alpha_i}-q^{-1}}{1-\bs{x}^{-\alpha_i}} \bs{x}^{-\alpha_i}s_i
= \frac{q^{-1}-1}{1-\boldsymbol{x}^{-\alpha_i}} + \frac{\boldsymbol{x}^{-\alpha_i}-q^{-1}}{\boldsymbol{x}^{\alpha_i}-1} s_i = q^{-1/2}M_i^-. \qedhere \]
\end{proof}

Having connected the recursive operators, we may now formally connect our lifted K-theoretic stable envelopes to lattice model partition functions. Take $q^{-1}=t$ for the rest of this section. By Proposition~\ref{prop:recursion},
\[ Z^{\lambda}_{\sigma, ws_i} (\bs{x}^{-1} ; \bs{y}^{-1}) = \left\{\begin{array}{ll}
       \overline{\tau}_i \cdot Z^{\lambda}_{\sigma, w} (\bs{x}^{-1} ; \bs{y}^{-1}), & \text{if }
       \ell (ws_i) < \ell (w),\\
       \overline{\tau}_i^{- 1} \cdot Z^{\lambda}_{\sigma, w} (\bs{x}^{-1} ; \bs{y}^{-1}), &
       \text{if } \ell (ws_i) > \ell (w) .
     \end{array}\right. \]

\begin{theorem} \label{thm:stab-partition-function}
For all permutations $w$,
\[
\stab_-(w) = (-1)^{\ell(w_0)} q^{\ell(w_0)-\ell(w)/2} r(\boldsymbol{x}_{w_0}; \boldsymbol{y}) Z^\square_{1_W,w}(\boldsymbol{x}^{-1};\boldsymbol{y}^{-1})
\]
and
\[
\stab_+(w) = q^{\ell(w_0)-\ell(w)/2} r(\boldsymbol{x}; \boldsymbol{y}_{w_0}) Z^\square_{w_0,w}(\boldsymbol{x}^{-1};\boldsymbol{y}_{w_0}^{-1}).
\]
\end{theorem}
\begin{proof}
We prove the $\stab_-$ equality, leaving the similar result for $\stab_+$ to the reader. The proof is by induction on $w$. For the base case,
using Proposition~\ref{prop:basecase} and Definition~\ref{def:doublestab}, together with \eqref{eq:rinprods} and \eqref{eq:rwithbasecase}:
\begin{align*} \stab_-(w_0) &= q^{\ell(w_0)/2}
r(\boldsymbol{x}_{w_0};\boldsymbol{x})\prod_{i+j\leqslant n} \left(1-\frac{x_i}{y_j}\right) \\
&= (-1)^{\ell(w_0)} q^{\ell(w_0) - \ell(w_0)/2} r(\boldsymbol{x}_{w_0};\boldsymbol{y}) Z_{1_W,w_0}^\square(\boldsymbol{x}^{-1}; \boldsymbol{y}^{-1}). \end{align*} 

For the inductive step, If $ws_i<w$, then using Proposition \ref{prop:stab-op-comp} and the induction hypothesis, 
\begin{align*} \stab_-(ws_i) &= M_i^- \stab_-(w) \\
&=q^{1/2}\boldsymbol{x}^{-\rho}\overline{\tau}_i \boldsymbol{x}^{\rho} (-1)^{\ell(w_0)} q^{\ell(w_0)-\ell(w)/2} r(\boldsymbol{x}_{w_0}; \boldsymbol{y}) Z^\square_{1_W,w}(\boldsymbol{x}^{-1};\boldsymbol{y}^{-1})\\
&= (-1)^{\ell(w_0)} q^{\ell(w_0)-\ell(w)/2} q^{1/2} \boldsymbol{x}^{-\rho} \overline{\tau}_i \boldsymbol{x}^{\rho} r(\boldsymbol{x}_{w_0};\boldsymbol{y}) Z_{1_W,w}^\square(\boldsymbol{x}^{-1}; \boldsymbol{y}^{-1}) \\
&= (-1)^{\ell(w_0)} q^{\ell(w_0)-\ell(w)/2} q^{1/2} r(\boldsymbol{x}_{w_0};\boldsymbol{y}) \overline{\tau}_i Z_{1_W,w}^\square(\boldsymbol{x}^{-1}; \boldsymbol{y}^{-1}) \\
&= (-1)^{\ell(w_0)} q^{\ell(w_0)-\ell(ws_i)/2} r(\boldsymbol{x}_{w_0};\boldsymbol{y}) Z_{1_W,ws_i}^\square(\boldsymbol{x}^{-1}; \boldsymbol{y}^{-1}) \\
\end{align*}
where we have used that $\boldsymbol{x}^{\rho} r(\boldsymbol{x}_{w_0};\boldsymbol{y})$ is symmetric in $\boldsymbol{x}$, so commutes with $\overline{\tau}_i$.

For the $\stab_+$ equality, the base case is 
\[\stab_+(e) = \prod_{i<j} \left(1-\frac{x_i}{y_j}\right) = r(\boldsymbol{x},\boldsymbol{y}_{w_0})\theta \prod_{i<j} \left(1-\frac{x_i}{y_j}\right) 
= q^{\ell(w_0)} r(\boldsymbol{x}; \boldsymbol{y}_{w_0})\theta Z_{w_0,1_W}^\square(\boldsymbol{x}; \boldsymbol{y}_{w_0}).\]

For the inductive step, if $ws_i>w$, then using Proposition \ref{prop:stab-op-comp},
\begin{align*} \stab_+(ws_i) &= M_i^+ \stab_+(w) \\
&= q^{-1/2} \boldsymbol{x}^{\rho} \overline{\tau}_i^{-1} \boldsymbol{x}^{-\rho} q^{\ell(w_0)-\ell(w)/2} r(\boldsymbol{x}; \boldsymbol{y}_{w_0}) Z^\square_{w_0,w}(\boldsymbol{x}^{-1};\boldsymbol{y}_{w_0}^{-1})\\
&= q^{\ell(w_0)-\ell(w)/2}q^{-1/2} \boldsymbol{x}^{\rho} \overline{\tau}_i^{-1} \boldsymbol{x}^{-\rho} r(\boldsymbol{x}; \boldsymbol{y}_{w_0}) Z^\square_{w_0,w}(\boldsymbol{x}^{-1};\boldsymbol{y}_{w_0}^{-1}) \\
&= q^{\ell(w_0)-\ell(w)/2}q^{-1/2} r(\boldsymbol{x}; \boldsymbol{y}_{w_0}) \overline{\tau}_i^{-1} Z^\square_{w_0,w}(\boldsymbol{x}^{-1};\boldsymbol{y}_{w_0}^{-1}) \\
&= q^{\ell(w_0)-\ell(ws_i)/2} r(\boldsymbol{x}; \boldsymbol{y}_{w_0}) Z^\square_{w_0,ws_i}(\boldsymbol{x}^{-1};\boldsymbol{y}_{w_0}^{-1}) \\
\end{align*} 
where we used that $\boldsymbol{x}^{-\rho} r(\boldsymbol{x}; \boldsymbol{y}_{w_0})$ is symmetric in $\boldsymbol{x}$.
\end{proof}

\subsection{Interpolation stab functions\label{subsec:interpstab}}

Theorem~\ref{thm:stab-partition-function} suggests an interpolation of the functions in Definition~\ref{def:doublestab}.

\begin{definition} \label{def:interpolation-stab}
For all permutations $u$ and $w$, we define the \emph{interpolated stab polynomials} by
\[
\stab_u(w) := (-1)^{\ell(u w_0)} q^{\ell(w_0)-\ell(w)/2} r(\boldsymbol{x}_{(uw_0)^{-1}}; \boldsymbol{y}_u) Z_{u,w}^\square(\boldsymbol{x}^{-1},(\boldsymbol{y}_u)^{-1}).\]
\end{definition}

Note in particular that with this definition,
\[
\stab_{e}(w) = \stab_-(w), \qquad \stab_{w_0}(w) = \stab_+(w),
\]
so they indeed interpolate between the lifted K-theoretic stable envelopes $\stab_\pm$.

Now building on the results of the prior section, we may construct recursions for both these functions and their restrictions. Set 
$$ M_i^{u} := q^{1/2} \boldsymbol{x}^{-w_0u^{-1}w_0\rho}\overline{\tau}_i \boldsymbol{x}^{w_0u^{-1}w_0\rho} . $$ 
In particular
\[
M_i^{1_W} = q^{1/2}\boldsymbol{x}^{-\rho}\overline{\tau}_i \boldsymbol{x}^{\rho} = M_i^- \qquad \text{and} \qquad M_i^{w_0} = q^{1/2}\boldsymbol{x}^{\rho}\overline{\tau}_i \boldsymbol{x}^{-\rho} = (M_i^+)^{-1}.
\]
Given any fixed $u \in W$, the $M_i^u$ satisfy the braid relations since the $\tau_i$ and thus the $\overline{\tau}_i$ braid.

\begin{proposition} \label{prop:interpolated-stab}
For any $u \in W$, the interpolated stab polynomial $\stab_u$ is completely determined by the following properties:
\begin{enumerate}
\item At the element $u w_0 \in W$,
\begin{equation} \label{eq:interpolated-stab-base-case}
\stab_u(uw_0) = (-1)^{\ell(u )} q^{\ell(uw_0)/2} r(\boldsymbol{x}_{(uw_0)^{-1}};\boldsymbol{x})\prod_{i+j\leqslant n} \left(1-\frac{x_i}{y_{u(j)}}\right).
\end{equation}
\item For any $w \in W$ and any simple reflection $s_i$ with $ws_i<w$, then 
\begin{equation}\label{eq:stabfromm}
\stab_u(ws_i) = M_i^u \stab_u(w).
\end{equation}
\end{enumerate}
\end{proposition}

\begin{proof}
The base case $w=uw_0$ follows from the definitions since 
\begin{align*} \stab_u(uw_0) &= (-1)^{\ell(u w_0)} q^{\ell(w_0)-\ell(uw_0)/2} r(\boldsymbol{x}_{(uw_0)^{-1}}; \boldsymbol{y}_u) Z_{u,uw_0}^\square(\boldsymbol{x}^{-1},\boldsymbol{y}_u^{-1}) \\
&= (-1)^{\ell(u w_0)} q^{\ell(w_0)/2+\ell(u)/2} r(\boldsymbol{x}_{(uw_0)^{-1}}; \boldsymbol{y}_u) Z_{u,uw_0}^\square(\boldsymbol{x}^{-1},\boldsymbol{y}_u^{-1}) \\&= (-1)^{\ell(u w_0)} q^{\ell(w_0)/2-\ell(u)/2} r(\boldsymbol{x}_{(uw_0)^{-1}}; \boldsymbol{y}_u) \prod_{i+j\leqslant n} \left(1-\frac{y_{u(j)}}{x_i}\right) \\
&= (-1)^{\ell(u)} q^{\ell(w_0)/2-\ell(u)/2} r(\boldsymbol{x}_{(uw_0)^{-1}}; \boldsymbol{x}) \prod_{i+j\leqslant n} \left(1-\frac{x_i}{y_{u(j)}}\right).\end{align*}

For the inductive step, suppose $ws_i<w$. Then, 
\begin{align*}
M_i^u \stab_u(w)
&=q^{1/2} \boldsymbol{x}^{-w_0u^{-1}w_0\rho\rho}\overline{\tau}_i \boldsymbol{x}^{w_0u^{-1}w_0\rho} (-1)^{\ell(u w_0)} q^{\ell(w_0)-\ell(w)/2} r(\boldsymbol{x}_{(uw_0)^{-1}}; \boldsymbol{y}_u) Z_{u,w}^\square(\boldsymbol{x}^{-1},\boldsymbol{y}_u^{-1})\\
&=  (-1)^{\ell(u w_0)} q^{\ell(w_0)-\ell(w)/2}q^{1/2} \boldsymbol{x}^{-w_0u^{-1}w_0\rho}\overline{\tau}_i \boldsymbol{x}^{w_0u^{-1}w_0\rho} r(\boldsymbol{x}_{(uw_0)^{-1}}; \boldsymbol{y}_u) Z_{u,w}^\square(\boldsymbol{x}^{-1},\boldsymbol{y}_u^{-1}) \\
&=  (-1)^{\ell(u w_0)} q^{\ell(w_0)-\ell(w)/2}q^{1/2} r(\boldsymbol{x}_{(uw_0)^{-1}}; \boldsymbol{y}_u)\overline{\tau}_i Z_{u,w}^\square(\boldsymbol{x}^{-1},\boldsymbol{y}_u^{-1}) \\
&=  (-1)^{\ell(u w_0)} q^{\ell(w_0)-\ell(ws_i)/2} r(\boldsymbol{x}_{(uw_0)^{-1}}; \boldsymbol{y}_u) Z_{u,ws_i}^\square(\boldsymbol{x}^{-1},\boldsymbol{y}_u^{-1}) \\
&= \stab_u(ws_i),
\end{align*} 
since $\boldsymbol{x}^{w_0u^{-1}w_0\rho} r(\boldsymbol{x}_{(uw_0)^{-1}}; \boldsymbol{y}_u)$ is symmetric in $\boldsymbol{x}$, so commutes with $\tau_i$. 

Note that this determines $\stab_u(w)$ for all permutations $u$ and $w$ since one can get from $\stab_u(uw_0)$ to $\stab_u(w)$ by applying a combination of $M_i^u$ and $(M_i^u)^{-1}$ operators.
\end{proof}

Consider the $u$-Bruhat order, denoted $\leqslant_u$, where $w_1 \leqslant_u
w_2$ means $u^{-1}w_1 \leqslant u^{-1}w_2$ in the usual Bruhat order. In
particular, when $u=e$, the identity element, the $e$-Bruhat order is just the
Bruhat order, and when $u=w_0$, it is the opposite of the Bruhat order. We use
this $u$-Bruhat order to provide a vanishing condition for specializations in
the next result.

\begin{proposition}
  The restriction of the interpolated stab function $\stab_u (w) |_v$
  satisfies the following properties:
  \begin{enumerate}
    \item For any $u, v, w \in W$, then $\stab_u (w) |_v = 0$ unless $v
    \geqslant_u w$.
    
    \item For any element $u \in W$,
    \[ \stab_u (uw_0) |_{uw_0} = q^{\ell (uw_0) / 2}  \prod_{\alpha \in
       \Phi^+} (1 -\bs{y}^{\alpha}), \]
    \item For any $u, v, w \in W$ and any simple reflection $s_i$, then
    $\stab_u (ws_i) |_v$ equals 
		\begin{equation}
      \label{eq:interpolation-restriction-recurrence-case-1} \frac{q^{1 / 2} -
      q^{- 1 / 2}}{\bs{y}^{- v \alpha_i} - 1} \stab_u (w)  |_v -
      \frac{q^{1 / 2} \bs{y}^{- v \alpha_i} - q^{- 1 /
      2}}{\bs{y}^{- v \alpha_i} - 1} \bs{y}^{- \langle
      \alpha_i^{\vee}, w_0 u^{- 1} w_0 \rho \rangle v \alpha_i} \stab_u (w)
      |_{vs_i}
    \end{equation}
    if $ws_i < w$, and
    \begin{equation}
      \label{eq:interpolation-restriction-recurrence-case-2} \frac{(q^{1 / 2}
      - q^{- 1 / 2}) \bs{y}^{- v \alpha_i}}{\bs{y}^{- v \alpha_i}
      - 1} \stab_u (w)  |_v - \frac{q^{1 / 2} \bs{y}^{- v \alpha_i} -
      q^{- 1 / 2}}{\bs{y}^{- v \alpha_i} - 1} \bs{y}^{\langle
      \alpha_i^{\vee}, w_0 u^{- 1} w_0 \rho \rangle v \alpha_i} \stab_u (w)
      |_{vs_i}
    \end{equation}
    if $ws_i > w$. 
  \end{enumerate}
\end{proposition}

\begin{remark}
	Although most of our results are special to Cartan Type~A, the
	description of the restriction functions for $\stab_u$ in this Proposition
	is valid for all types. Thus, it can be taken as a \emph{definition} for
	(the GKM presentation of) a type-independent version of $\stab_u$.
\end{remark}

\begin{remark}
  If $\bs{y}^{v \alpha_i} = 1$ then $y_{v (i)} = y_{v (i + 1)}$. Then
  using (\ref{eq:relres}) we have $\stab_u (w) |_v = \stab_u (w) |_{vs_i}$.
  Since $\langle \alpha_i^{\vee}, w_0 u^{- 1} w_0 \rho \rangle$ is an integer,
  the numerator and denominators in
  (\ref{eq:interpolation-restriction-recurrence-case-1}) and
  (\ref{eq:interpolation-restriction-recurrence-case-2}) both vanish. Hence
  these recurrences do not introduce unwanted denominators.
\end{remark}

\begin{proof}
  We apply Proposition~\ref{prop:interpolated-stab} and the definition of the
  specialization at $v$ to obtain:
  \begin{equation}
    \label{eq:stabprelim} \stab_u (uw_0) |_v = (- 1)^{\ell (u)} q^{\ell (uw_0)
    / 2} r (\bs{y}_{v (uw_0)^{- 1}} ; \bs{y}_v)  \prod_{i + j
    \leqslant n} \left( 1 - \frac{y_{v (i)}}{y_{u (j)}} \right) .
  \end{equation}
  In the special case $v = uw_0$,
  
  \begin{align*}
    \stab_u (uw_0) |_{uw_0} & = (- 1)^{\ell (u)} q^{\ell (uw_0) / 2} r
    (\bs{y}; \bs{y}_{uw_0})  \prod_{i + j \leqslant n} \left( 1 -
    \frac{y_{uw_0 (i)}}{y_{u (j)}} \right) .\\
    & = (- 1)^{\ell (u)}  (- 1)^{\ell (w_0)} q^{\ell (uw_0) / 2} r
    (\bs{y}_u ; \bs{y}_{uw_0}) r (\bs{y}; \bs{y}_u) 
    \prod_{i + j \leqslant n} \left( 1 - \frac{y_{uw_0 (i)}}{y_{u (j)}}
    \right) .\\
    & = (- 1)^{\ell (u) + \ell (w_0)} q^{\ell (uw_0) / 2} r (\bs{y};
    \bs{y}_u)  \prod_{i + j \leqslant n} \left( 1 - \frac{y_{u
    (i)}}{y_{uw_0 (j)}} \right) .\\
    & = (- 1)^{\ell (u)} q^{\ell (uw_0) / 2}  \prod_{\alpha \in \Phi^+ 
    \text{with } u \alpha \in \Phi^-} \bs{y}^{- u \alpha}  \prod_{\alpha
    \in \Phi^+} (1 -\bs{y}^{u \alpha}) .\\
    & = q^{\ell (uw_0) / 2}  \prod_{\alpha \in \Phi^+} (1
    -\bs{y}^{\alpha}) .
  \end{align*}
   This proves Part~2 of the Proposition.
  
  Next assume that $ws_i < w$. We can then use \eqref{eq:stabfromm} to write
  \begin{align*}
    \stab_u (ws_i) |_v & = (M_i^u \stab_u (w)) |_v\\
    & = q^{1 / 2}  \left( \bs{x}^{- w_0 u^{- 1} w_0 \rho}  \bar{\tau}_i
    \bs{x}^{w_0 u^{- 1} w_0 \rho} \stab_u (w) \right) |_v\\
    & = q^{1 / 2}  \left( \bs{x}^{- w_0 u^{- 1} w_0 \rho}  \left(
    \frac{1 - q^{- 1}}{\bs{x}^{- \alpha_i} - 1} + \frac{q^{- 1}
    -\bs{x}^{- \alpha_i}}{\bs{x}^{- \alpha_i} - 1} s_i \right)
    \bs{x}^{w_0 u^{- 1} w_0 \rho} \stab_u (w) \right) |_v\\
    & = \left( \left( \frac{q^{1 / 2} - q^{- 1 / 2}}{\bs{x}^{-
    \alpha_i} - 1} + \frac{q^{- 1 / 2} - q^{1 / 2} \bs{x}^{-
    \alpha_i}}{\bs{x}^{- \alpha_i} - 1} \bs{x}^{s_i (w_0 u^{- 1}
    w_0 \rho) - w_0 u^{- 1} w_0 \rho} s_i \right) \stab_u (w) \right) |_v\\
    & = \left( \left( \frac{q^{1 / 2} - q^{- 1 / 2}}{\bs{x}^{-
    \alpha_i} - 1} + \frac{q^{- 1 / 2} - q^{1 / 2} \bs{x}^{-
    \alpha_i}}{\bs{x}^{- \alpha_i} - 1} \bs{x}^{- \langle
    \alpha_i^{\vee}, w_0 u^{- 1} w_0 \rho \rangle \alpha_i} s_i \right)
    \stab_u (w) \right) |_v\\
    & = \frac{q^{1 / 2} - q^{- 1 / 2}}{\bs{y}^{- v \alpha_i} - 1}
    \stab_u (w)  |_v + \frac{q^{- 1 / 2} - q^{1 / 2} \bs{y}^{- v
    \alpha_i}}{\bs{y}^{- v \alpha_i} - 1} \bs{y}^{- \langle
    \alpha_i^{\vee}, w_0 u^{- 1} w_0 \rho \rangle v \alpha_i} \stab_u (w)
    |_{vs_i} .
  \end{align*}
  
  If, alternatively, $ws_i > w$, then
  
  \begin{align*}
    \stab_u (ws_i) |_v & = (M_i^u)^{- 1} \stab_u (w) |_v\\
    & = q^{- 1 / 2}  \left( \bs{x}^{w_0 u^{- 1} w_0 \rho} 
    \bar{\tau}_i^{- 1} \bs{x}^{- w_0 u^{- 1} w_0 \rho} \stab_u (w)
    \right) |_v\\
    & = q^{- 1 / 2}  \left( \bs{x}^{w_0 u^{- 1} w_0 \rho}  \left(
    \frac{(q^{- 1} - 1)\bs{x}^{- \alpha_i}}{q^{- 1} (1 -\bs{x}^{-
    \alpha_i})} + \frac{\bs{x}^{- \alpha_i} - q^{- 1}}{q^{- 1} (1
    -\bs{x}^{- \alpha_i})} s_i \right) \bs{x}^{- w_0 u^{- 1} w_0
    \rho} \stab_u (w) \right) |_v\\
    & = \left( \left( \frac{(q^{1 / 2} - q^{- 1 / 2})\bs{x}^{-
    \alpha_i}}{\bs{x}^{- \alpha_i} - 1} - \frac{q^{1 / 2}
    \bs{x}^{- \alpha_i} - q^{- 1 / 2}}{\bs{x}^{- \alpha_i} - 1}
    \bs{x}^{- s_i (w_0 u^{- 1} w_0 \rho) + w_0 u^{- 1} w_0 \rho} s_i
    \right) \stab_u (w) \right) |_v\\
    & = \left( \left( \frac{(q^{1 / 2} - q^{- 1 / 2})\bs{x}^{-
    \alpha_i}}{\bs{x}^{- \alpha_i} - 1} - \frac{q^{1 / 2}
    \bs{x}^{- \alpha_i} - q^{- 1 / 2}}{\bs{x}^{- \alpha_i} - 1}
    \bs{x}^{\langle \alpha_i^{\vee}, w_0 u^{- 1} w_0 \rho \rangle
    \alpha_i} s_i \right) \stab_u (w) \right) |_v\\
    & = \frac{(q^{1 / 2} - q^{- 1 / 2}) \bs{y}^{- v
    \alpha_i}}{\bs{y}^{- v \alpha_i} - 1} \stab_u (w)  |_v - \frac{q^{1
    / 2} \bs{y}^{- v \alpha_i} - q^{- 1 / 2}}{\bs{y}^{- v
    \alpha_i} - 1} \bs{y}^{\langle \alpha_i^{\vee}, w_0 u^{- 1} w_0 \rho
    \rangle v \alpha_i} \stab_u (w) |_{vs_i} .
  \end{align*}
  This proves Part~3 of the Proposition.
  
  Finally, we prove Part~1, by induction on $w$. Specifically, we use reverse
  induction on the Coxeter length of $u^{- 1} w$. The base case is when $w =
  uw_0$. We now claim that
  \begin{equation}
    \label{eq:stabvanishes} \stab_u (uw_0) |_v = 0 \qquad \text{if } v \neq
    uw_0 .
  \end{equation}
  To see this, note that if $1 \neq y \in S_n$ there exists an $i$ such that
  $y (i) > i$. Applying this with $y = w_0^{- 1} u^{- 1} v$, if $v \neq uw_0$
  we have $w_0^{- 1} u^{- 1} v (i) > i$. Therefore $u^{- 1} v (i) < w_0 (i) =
  n + 1 - i$ which implies that $i + j \leqslant n$ where $j = u^{- 1} v (i)$.
  Hence one factor in the product in (\ref{eq:stabprelim}) is zero,
  proving~(\ref{eq:stabvanishes}).
  
  Now, suppose that (1) holds for $w$, and $ws_i \leqslant_u w$. That is, $u^{- 1}
  ws_i \leqslant u^{- 1} w$, and we assume that $\stab_u (w) |_v = 0$ unless $u^{-
  1} w \leqslant u^{- 1} v$. By Part~3,
  \[ \stab_u (ws_i) |_v = f_1 \stab_u (w)  |_v + f_2 \stab_u (w) |_{vs_i}, \]
  where $f_1$ and $f_2$ are given by
  either~\eqref{eq:interpolation-restriction-recurrence-case-1}
  or~\eqref{eq:interpolation-restriction-recurrence-case-2}. If $\stab_u
  (ws_i) |_v \neq 0$, then either $\stab_u (w) |_v \neq 0$ or $\stab_u (w)
  |_{vs_i} \neq 0$. The first case implies that $u^{- 1} w \leqslant u^{- 1} v$, so
  since $u^{- 1} ws_i \leqslant u^{- 1} w$, we have $u^{- 1} w \leqslant u^{- 1} v$ i.e.
  $w \leqslant_u v$. The second case implies that $u^{- 1} w \leqslant u^{- 1} vs_i$, and
  since right-multiplying by $s_i$ decreases the length of the left side, we
  again have $u^{- 1} w \leqslant u^{- 1} v$ i.e. $w \leqslant_u v$. This proves Part~1 of
  the Proposition.
\end{proof}

\begin{remark}
	Theorem~1.2 in~\cite{SuRestriction} gives a closed form expression for
	the restriction of the stable envelope in the context of equivariant cohomology. 
	In this context, the stable basis is (up to a sign) the characteristic classes of
	Verma modules. Finding closed forms for K-theoretic stable envelopes,
	apart from a few very special cases noted here, seems to be a challenging problem. 
\end{remark}

\section{The denominator conjectures} \label{sec:denom}

Bump and Nakasuji~{\cite{BumpNakasujiKL}} defined certain functions $r_{u, v}$
that are deformations of Kazhdan-Lusztig $R$-polynomials that arise in the
evaluation of integrals on $p$-adic groups. They made two conjectures (one of
them refined by Nakasuji and Naruse~{\cite{NakasujiNaruse}}). Both conjectures
were proved by Aluffi, Mihalcea, Sch{\"u}rmann and Su~{\cite{AMSSCasselman}}.

We remind the reader that in this section the parameter $q$ is the parameter
denoted~$t$ in the earlier sections of this paper, for comparison with the
literature. See Remark~\ref{rem:qiswhat}.

In this section we reconsider one of the two conjectures, the denominator
conjecture, which describes the poles of the $r_{u, v}$. \ This is
Theorem~10.5 in~{\cite{AMSSCasselman}}. We will prove it again here, and
indeed the proof in this section is motivated by their argument. However their
proof depends on their Theorem~7.4, which relies on properties of motivic
Chern classes and equivariant K-theory. One purpose of this section is to
remove this dependence in favor of elementary arguments. Their setup does
survive in their proof, since we will use arguments in a ring $\mathcal{R}$
(defined below) motivated by the GKM description of equivariant K-Theory.

The other purpose is to show in the Type~A case, the polynomials $r_{u, v}$
are essentially specializations of the partition functions $Z^{\Box}$ in our
earlier sections.

Let $T$ be a complex torus, and $\mathcal{O} (T)$ be the ring of regular
functions on $T$, generated by the functions $\bs{y} \mapsto
\bs{y}^{\lambda}$ with $\lambda$ in the ring $\Lambda = X^{\ast} (T)$ of
rational characters on $T$. We assume that $\Lambda$ contains the root system
$\Phi$ of a reductive complex Lie group with maximal torus~$T$, with Weyl
group acting on $\Lambda$ in the usual way. Let $q$ be an indeterminant, and
let $\mathcal{O}_q (T) =\mathbb{Z} [q, q^{- 1}] \otimes \mathcal{O} (T)$. Let
$W$ be the Weyl group of $\Phi$. If $\alpha \in \Phi$ let $r_{\alpha} \in W$
denote the reflection in the root $\alpha$. As usual, $\Phi$ is partitioned
into positive and negative roots, and $\{\alpha_1, \cdots, \alpha_{\ell} \}$
will denote the simple positive roots. Then $W$ is generated by the simple
reflections $s_i = r_{\alpha_i}$.We are using the same letter $r$ for both
reflections $r_{\alpha} \in W$ and for the function $r_{u, v}$ defined in the
next Theorem but we believe this should cause no confusion. The letter $s_i$
is reserved for {\textit{simple}} reflections.

\begin{theorem}[Bump and Nakasuji]
  \label{thm:bumpnakasuji}For $u, v \in W$ there is an element $r_{u, v}
  (\bs{y}) \in \mathcal{O}_q (T)$ such that $r_{u, v} (\bs{y})$
  vanishes unless $u \leqslant v$ in the Bruhat order, and $r_{v, v}
  (\bs{y}) = 1$, and that satisfies the following recursive relation. Let
  $s = s_i$ be a simple reflection such that $s_i v < v$. Then $\beta :=
  - v^{- 1} (\alpha_i)$ is a positve root. We have
  \begin{equation}
    \label{eq:rrecursion} r_{u, v} (\bs{y}) = \left\{ \begin{array}{ll}
      \frac{1 - q}{1 - \bs{y}^{\beta}} r_{u, sv} (\bs{y}) + r_{su, sv}
      (\bs{y}) & \text{if $su < u$,}\\
      (1 - q) \frac{\bs{y}^{\beta}}{1 - \bs{y}^{\beta}} r_{u, sv}
      (\bs{y}) + qr_{su, sv} & \text{if $su > u$.}
    \end{array} \right.
  \end{equation}
	If $\bs{y}\to\infty$ in such a direction that $\bs{y}^\alpha\to \infty$
	for all positive roots then $r_{u,v}(\bs{y})$ converges to the
	Kazhdan-Lusztig R-polynomial~\cite{KazhdanLusztigCoxeter}.
\end{theorem}

\begin{proof}
  See {\cite{BumpNakasujiKL}} Theorem~1.2.
\end{proof}

In addition to the recursion in Theorem~\ref{thm:bumpnakasuji}, we also need
the corresponding fact when $s_i$ is an ascent of $v$ instead of a descent.

\begin{proposition}
  \label{prop:altbnrecursion}Suppose that $s_i v > v$. Then $\beta :=
  v^{- 1} \alpha_i$ is a positive root. If $s_i u < u$ then
  \[ \frac{(\bs{y}^{- \beta} - q)  (q \bs{y}^{- \beta} - 1)}{(1 -
     \bs{y}^{- \beta})^2} r_{u, v} (\bs{y}) = \frac{1 - q}{1 -
     \bs{y}^{- \beta}} r_{u, s_i v} (\bs{y}) + r_{s_i u, s_i v}
     (\bs{y}), \]
  and if $s_i u > u$ then
  \[ \frac{(\bs{y}^{- \beta} - q)  (q \bs{y}^{- \beta} - 1)}{(1 -
     \bs{y}^{- \beta})^2} r_{u, v} (\bs{y}) = \frac{(1 - q)
     \bs{y}^{- \beta}}{1 - \bs{y}^{- \beta}} r_{u, s_i v} (\bs{y})
     + qr_{s_i u, s_i v} (\bs{y}) . \]
\end{proposition}

\begin{proof}
  We can apply (\ref{eq:rrecursion}) with $v' = s_i v$ replacing $v$ since
  $s_i v' < v$. For definiteness assume that $s_i$ is also a descent of $u$,
  so that the first case in (\ref{eq:rrecursion}) applies to the pair $(u,
  v')$, and the second case applies to the pair $(u', v')$ with $u' = s_i u$.
  This gives two relations which we can encode in the matrix identity
  \[ \left( \begin{array}{c}
       r_{u, s_i v}\\
       r_{s_i u, s_i v}
     \end{array} \right) = M \cdot \left( \begin{array}{c}
       r_{u, v}\\
       r_{s_i u, v}
     \end{array} \right), \qquad M = \left( \begin{array}{cc}
       \frac{1 - q}{1 - \bs{y}^{\beta}} & 1\\
       q & \frac{1 - q}{\bs{y}^{- \beta} - 1}
     \end{array} \right), \]
  where $\beta = - (v')^{- 1} (\alpha_i) = v^{- 1} (\alpha_i)$. We find that
  \[ - \det (M) = \frac{(\bs{y}^{\beta} - q)  (q \bs{y}^{\beta} -
     1)}{(1 - \bs{y}^{\beta})^2} = \frac{(\bs{y}^{- \beta} - q)  (q
     \bs{y}^{- \beta} - 1)}{(1 - \bs{y}^{- \beta})^2} . \]
  So multiplying by $- \det (M) M^{- 1}$ gives
  \[ \frac{(\bs{y}^{- \beta} - q)  (q \bs{y}^{- \beta} - 1)}{(1 -
     \bs{y}^{- \beta})^2} \left( \begin{array}{c}
       r_{u, v}\\
       r_{s_i u, v}
     \end{array} \right) = \left( \begin{array}{cc}
       \frac{1 - q}{1 - \bs{y}^{- \beta}} & 1\\
       q & \frac{1 - q}{\bs{y}^{\beta} - 1}
     \end{array} \right) \left( \begin{array}{c}
       r_{u, s_i v}\\
       r_{s_i u, s_i v}
     \end{array} \right), \]
  which is equivalent to the stated recursion.
\end{proof}

The \textit{denominator conjecture}, proved in~{\cite{AMSSCasselman}} is the
following statement. Let
\[ S (u, v) = \{\alpha \in \Phi^+ |u \leqslant vr_{\alpha} < v\} . \]
The \textit{Deodhar inequality} asserts that $|S (u, v) | \geqslant \ell (v)
- \ell (u)$, with equality if and only if the Kazhdan-Lusztig polynomial
$P_{w_0 v, w_0 u} = 1$. The polynomials $r_{u, v} (\bs{y})$ are convenient
for the calculation of certain unipotent integrals on $G' (F)$ with $F$ a
nonarchimedean local field, where $G$ is the Langlands dual of $G'$. The
polynomials $r_{u, v} (\bs{y})$ are deformations of the Kazhdan-Lusztig
$R$ polynomial in that $r_{u, v} (\bs{y}) \longrightarrow R (\bs{y})$
if $\bs{y} \rightarrow \infty$ in a direction such that
$\bs{y}^{\alpha} \rightarrow \infty$ for all positive roots $\alpha$. Then
the denominator conjecture (now a theorem) was the following statement.

\begin{theorem}
	\label{thm:denomconject}
  The factor
  \[ \prod_{\beta \in S (u, v)} (1 - \bs{y}^{\beta}) r_{u, v} (\bs{y})
  \]
  is analytic on~$T$.
\end{theorem}

One can try to deduce this from the recursion (\ref{eq:rrecursion}), arguing
by induction on $\ell (v)$. This almost works, but for certain roots $\beta
\not\in S (u, v)$ the factor $1 - \bs{y}^{\beta}$ appears in both terms of
the recursion, and the result depends on a cancellation that is difficult to
prove in general. That is, for $s$ a simple reflection, there may be a root
$\alpha$ that appears in both $S (u, sv)$ and $S (su, sv)$, but not $S (u,
v)$. Hence to prove that it is not a pole of $r_{u, v}$, one would need to
prove that these poles of $r_{u, sv}$ and $r_{su, sv}$ cancel. See~Bump and
Nakasuji~{\cite{BumpNakasujiKL}} for further discussion, and Example~1.7 of
that paper for an example of this cancellation phenomenon.

Let $\mathcal{O}_q (T)^W$ denote the direct sum of $|W|$ copies of
$\mathcal{O}_q (T)$. We will denote by $[f]$ an element of $\mathcal{O}_q
(T)^W$, and if $w \in W$, then $[f]_w$ denotes the $w$-th component of $[f]$.
Motivated by the GKM description of equivariant K-theory, we define a subring
$\mathcal{R} \subset \mathcal{O}_q (T)^W$ characterized by a divisibility
property.

\begin{definition}
  Let $\mathcal{R}$ be the subring of $\mathcal{O}_q (T)^W$ characterized by
  the following condition. An element $[f] \in \mathcal{O}_q (T)^W$ is in
  $\mathcal{R}$ if for every root $\alpha$ and every $w \in W$ the factor $1 -
  \bs{y}^{\alpha}$ divides $[f]_w - [f]_{r_{\alpha} w}$ in the unique
  factorization ring~$\mathcal{O}_q (T)$.
\end{definition}

\begin{remark}
  \label{rem:altdefr}If this is satisfied, then for every root $\beta$ and
  every $w \in W$ the factor $1 - \bs{y}^{w (\beta)}$ divides $[f]_w -
  [f]_{wr_{\beta}}$. Indeed let $\alpha = w (\beta)$ so that $wr_{\beta} =
  r_{\alpha} w$, and the statement follows.
\end{remark}

Let $\mathcal{O}(T_{\operatorname{reg}})$ be the ring obtained by adjoining
the factors $(1-\bs{y}^{\alpha})^{-1}$ to $\mathcal{O}(T)$, as $\alpha$ runs through the positive
roots. If $\alpha_i$ is a simple root and $s_i$ the corresponding reflection, define
the following operator $D_i$ on $\mathcal{O}_q (T_{\operatorname{reg}})^W$:
\[ (D_i [f])_w = \frac{q - 1}{1 - \bs{y}^{w (\alpha_i)}} [f]_w +
   \frac{\bs{y}^{w (\alpha_i)} - q}{1 - \bs{y}^{w (\alpha_i)}}
   [f]_{ws_i} . \]
Note that $D_i$ does not preserve $\mathcal{O}_q(T)^W$, since the
operation could introduce denominators. We will show in Lemma~\ref{lem:distable}
that $D_i$ does preserve the subring $\mathcal{R}$, but first we note
the following.

\begin{proposition}
The operators $D_i$ satisfy the quadratic relations $D_i^2=(q-1)D_i+q$
and the braid relations of the Iwahori Hecke algebra of~$W$.
\end{proposition}

\begin{proof}
This is an instance of Theorem~1.1 in~\cite{BBBF}. We will not make
use of this fact in the proof of the Denominator conjecture.
\end{proof}

\begin{lemma}
  \label{lem:distable}If $[f] \in \mathcal{R}$ then $D_i [f] \in \mathcal{R}$.
\end{lemma}

\begin{proof}
  We can write
  \[ (D_i [f])_w = - [f]_{ws_i} + (q - 1)  \frac{[f]_w - [f]_{ws_i}}{1 -
     \bs{y}^{w \alpha_i}} . \]
	This is integral by Remark~\ref{rem:altdefr}, so $D_i [f] \in \mathcal{O} (T)^W$. 
	To prove the divisibility property for $D_i[f]$, let $\beta$ be any root. Then
  \[ (D_i [f])_{r_{\beta} w} - (D_i [f])_w = \]
  \begin{equation}
    \label{eq:fastdiv} = - ([f]_{r_{\beta} ws_i} - [f]_{ws_i}) + (q - 1) 
    \left( \frac{[f]_{r_{\beta} w} - [f]_{r_{\beta} ws_i}}{1 -
    \bs{y}^{r_{\beta} w (\alpha_i)}} - \frac{[f]_w - [f]_{ws_i}}{1 -
    \bs{y}^{w (\alpha_i)}} \right)
  \end{equation}
  The first term is divisible by $1 - \bs{y}^{\beta}$ since $[f] \in
  \mathcal{R}$. For the second, there are two cases.
  
  \textbf{Case 1:} $\beta = \pm w (\alpha_i)$. Then $r_{\beta} = ws_i w^{-
  1}$, and also $r_{\beta} w (\alpha_i) = - w (\alpha_i)$. Using these facts,
  the second term equals
  \[ (q - 1)  \left( \frac{[f]_{ws_i} - [f]_{w}}{1 - \bs{y}^{- w
     (\alpha_i)}} - \frac{[f]_w - [f]_{ws_i}}{1 - \bs{y}^{w (\alpha_i)}}
		 \right) = (q - 1)  ([f]_{ws_i} - [f]_{w}),\]
	Where we have used the identity $\frac{1}{1-1/y}+\frac{1}{1-y}=1$.
  This is divisible by $1 - \bs{y}^{\beta}$ by Remark~\ref{rem:altdefr}.
  
  \textbf{Case 2:} $\beta \neq \pm w (\alpha_i)$. In this case the second
  term in (\ref{eq:fastdiv}) equals $q - 1$ times
  \[ \frac{(1 - \bs{y}^{w (\alpha_i)})  ([f]_{r_{\beta} w} -
     [f]_{r_{\beta} ws_i}) - (1 - \bs{y}^{r_{\beta} w (\alpha_i)})  ([f]_w
     - [f]_{ws_i})}{(1 - \bs{y}^{r_{\beta} w (\alpha_i)})  (1 -
     \bs{y}^{w (\alpha_i)})} \]
  We note that $1 -\bs{y}^{\beta}$ divides $1 -\bs{y}^{N \beta}$
  if $N \in \mathbb{Z}$, and taking $N = \langle \beta^{\vee}, w (\alpha_i)
  \rangle \in \mathbb{Z}$ we see that
  \[ \bs{y}^{r_{\beta} w (\alpha_i)} =\bs{y}^{w (\alpha_i)}
     \bs{y}^{- \langle \beta^{\vee}, w (\alpha_i) \rangle \beta} \equiv
     \bs{y}^{w (\alpha_i)} \quad \operatorname{mod} \; 1 -\bs{y}^{\beta} .
  \]
  Therefore the numerator in our last expression is $\equiv$ mod $1 -
  \bs{y}^{\beta}$ to
  \[ (1 - \bs{y}^{w (\alpha_i)})  (([f]_{s_{\beta} w} - [f]_{s_{\beta}
     ws_i}) - ([f]_w - [f]_{ws_i})), \]
  and since $[f] \in \mathcal{R}$, this is divisible by $1 -
  \bs{y}^{\beta}$. Since $1 - \bs{y}^{\beta}$ is prime to the
  denominator, it divides the second term in this case also.
\end{proof}

If $w \in W$ define
\[ \phi_w (\bs{y}) = \prod_{\substack{
     \alpha \in \Phi^+\\
		 w^{- 1} (\alpha) \in \Phi^+}}
 (q \bs{y}^{- \alpha} - 1) 
	 \prod_{\substack{
     \alpha \in \Phi^+\\
		 w^{- 1} (\alpha) \in \Phi^-}}
 (1 - \bs{y}^{- \alpha}) . \]
\begin{lemma}
  \label{lem:fratio}Let $\alpha_i$ be a simple root, $s_i$ the corresponding
  simple reflection and let $v \in W$. Then
  \[ \label{eq:fvfvs} \frac{\phi_{vs_i} (\bs{y})}{\phi_v (\bs{y})} =
     \left\{ \begin{array}{ll}
       \frac{1 - \bs{y}^{v (\alpha_i)}}{\bs{y}^{v (\alpha_i)} - q} &
       \text{if $vs_i > v$,}\\
       \frac{\bs{y}^{- v (\alpha)} - q}{1 - \bs{y}^{- v (\alpha)}} &
       \text{if } vs_i < v.
     \end{array} \right. \]
\end{lemma}

\begin{proof}
  We prove the first case and leave the second to the reader. With $v < vs_i$
  the root $v (\alpha_i)$ is positive and moreover
  \[ \begin{array}{c}
       \{\alpha \in \Phi^+ | (vs_i)^{- 1} \alpha \in \Phi^- \} = \{\alpha \in
       \Phi^+ |v^{- 1} \alpha \in \Phi^- \} \cup \{v (\alpha_i)\},\\
       \{\alpha \in \Phi^+ | (vs_i)^{- 1} \alpha \in \Phi^+ \} = \{\alpha \in
       \Phi^+ |v^{- 1} \alpha \in \Phi^- \} \setminus \{v (\alpha_i)\} .
     \end{array} \]
  This means that each of the two products defining $\phi_v$ and $\phi_{vs_i}$
  each differs by a single term, and the ratio $\phi_v (\bs{y}) /
  \phi_{vs_i} (\bs{y})$ is the ratio of these terms, that is $(1 -
  \bs{y}^{- v (\alpha_i)}) / (q \bs{y}^{- v (\alpha_i)} - 1)$ which
  equals the first case of~(\ref{eq:fvfvs}).
\end{proof}

Now if $u, v \in W$ define $[\Theta_u] \in \mathcal{O} (T)^W$ be defined by
\[ [\Theta_u]_v = \phi_v r_{u^{- 1}, v^{- 1}} . \]
\begin{proposition} \label{prop:D-of-Theta}
  \label{prop:thetarecurse}If $s_i u < u$ then $\Theta_{us_i} = D_i \Theta_u$.
\end{proposition}

\begin{proof}
  We must show that $[\Theta_{us_i}]_v = [D_i \Theta_u]_v$, that is
  \begin{equation}
    \label{eq:target} [\Theta_{us_i}]_v = \frac{q - 1}{1 - \bs{y}^{v
    (\alpha_i)}} [\Theta_u]_v + \frac{\bs{y}^{v (\alpha_i)} - q}{1 -
    \bs{y}^{v (\alpha_i)}} [\Theta_u]_{vs_i} .
  \end{equation}
  There are two cases. First suppose that $vs_i > v$. We want to apply the
  recursion (\ref{eq:rrecursion}), but with $u$ replaced by $u^{- 1}$ and $v$
  replaced by $s_i v^{- 1}$. Note that the first case of the recursion is then
  applicable and we obtain
  \[ r_{u^{- 1}, (vs_i)^{- 1}} = \frac{1 - q}{1 - \bs{y}^{\beta}} r_{u^{-
     1}, v^{- 1}} (\bs{y}) + r_{(us_i)^{- 1}, v^{- 1}} (\bs{y}) \]
  where now $\beta = v (\alpha_i)$ because we replaced $v$ by $s_i v^{- 1}$ in
  Theorem~\ref{thm:bumpnakasuji}. We multiply both sides of this by $\phi_v$
  and use the first case of Lemma~\ref{lem:fratio} to obtain (\ref{eq:target})
  in this case in the form
  \[ \frac{\bs{y}^{\beta} - q}{1 - \bs{y}^{\beta}} [\Theta_u]_{vs_i} =
     \frac{1 - q}{1 - \bs{y}^{\beta}} [\Theta_u]_v + [\Theta_{us_i}]_v .
  \]
  Next suppose that $vs_i < v$. Now we use the first case of
  Proposition~\ref{prop:altbnrecursion} with $u$ and $v$ replaced by $u^{- 1}$
  and $s_i v^{- 1}$ respectively. This tells us, with $\beta = - v (\alpha_i)$
  that
  \[ \frac{(\bs{y}^{- \beta} - q)  (q \bs{y}^{- \beta} - 1)}{(1 -
     \bs{y}^{- \beta})^2} r_{u^{- 1}, (vs_i)^{- 1}} = \frac{1 - q}{1 -
     \bs{y}^{- \beta}} r_{u^{- 1}, v^{- 1}} + r_{(us_i)^{- 1}, v^{- 1}} .
  \]
  Now we multiply both sides by $\phi_v$ and use the second case of
  Lemma~\ref{lem:fratio} to obtain
  \[ \frac{\bs{y}^{- \beta} - q}{1 - \bs{y}^{- \beta}}
     [\Theta_u]_{vs_i} = \frac{1 - q}{1 - \bs{y}^{- \beta}} [\Theta_u]_v +
     [\Theta_{us_i}]_v \]
  which again is equivalent to (\ref{eq:target}) in this case.
\end{proof}

\begin{proposition}
  For any $w \in W$, $\Theta_w \in \mathcal{R}$.
\end{proposition}

\begin{proof}
  If $w = w_0$ (the longest element of $W$) then $[\Theta_w]_v = 1$ if $v =
  w_0$ and 0 otherwise, so thes statement is true for $\Theta_{w_0}$. Now if
  $\Theta_w \in \mathcal{R}$ and $s_i$ is any left descent, then $\Theta_{s_i
  w} = D_i \Theta_w$ by Proposition~\ref{prop:thetarecurse}, so $\Theta_{s_i
  w} \in \mathcal{R}$ by Lemma~\ref{lem:distable}. The result thus follows by
  downward induction.
\end{proof}

\begin{proof}[Proof of Theorem~\ref{thm:denomconject}]
  Suppose that $u \leqslant v$. It follows from the recursive definition of
  $r_{u, v}$ that all denominator factors are of the form $1
  -\bs{y}^{\alpha}$ with $\alpha \in \Phi^+$. (Note that $1
  -\bs{y}^{\alpha}$ and $1 -\bs{y}^{- \alpha}$ differ by a unit in
  $\mathcal{O} (T)$, so we may restrict ourselve to $\alpha \in \Phi^+$.) What
  we must show is that all such denominator factors have $\alpha \in S (u,
  v)$, that is, $u \leqslant v.r_{\alpha} < v$.
  
  Recall that $[\Theta_{u^{- 1}}]_{v^{- 1}}$ is regular on $\mathcal{O} (T)$.
  Thus the poles of $r_{u, v} = \phi_{v^{- 1}}^{- 1} [\Theta_{u^{- 1}}]_{v^{-
  1}}$ can only come from the poles of $\phi_{v^{- 1}}^{- 1}$, and we need
  only need to consider those of the form $1 -\bs{y}^{- \alpha}$ (or
  equivalently $1 -\bs{y}^{\alpha}$), not those of the form
  $q\bs{y}^{- \alpha} - 1$. From the definition of $\phi_{v^{- 1}}$
  these have \ $\alpha \in \Phi^+$ and $v (\alpha) \in \Phi^-$. This condition
  is equivalent to $v r_{\alpha} < v$.
  
  We must also show that if $1 -\bs{y}^{\alpha}$ is is a denominator
  factor then $u \leqslant v.r_{\alpha}$. We note that $\phi_{v^{- 1}} r_{u,
  v} = [\Theta_{u^{- 1}}]_{v^{- 1}}$ is divisible by $1
  -\bs{y}^{\alpha}$ if $u \nleqslant v r_{\alpha}$. Indeed,
  $\Theta_{u^{- 1}} \in \mathcal{R}$ so $1 -\bs{y}^{\alpha}$ divides
  $[\Theta_{u^{- 1}}]_{v^{- 1}} - [\Theta_{u^{- 1}}]_{r_{\alpha} v^{- 1}}$and
  the second term vanishes since $u^{- 1} \nleqslant r_{\beta} v^{- 1}$. Thus
  $1 -\bs{y}^{\alpha}$ does not appear in the denominator of $r_{u, v} =
  \phi_{v^{- 1}}^{- 1} [\Theta_{u^{- 1}}]_{v^{- 1}}$.
\end{proof}

Up to this point the results of this section are valid for general Cartan types, but
we now specialize to Type~A to show that in this case
$r_{u,v}$ can be expressed in terms of our partition functions.

\begin{proposition} \label{prop:part-fun-r-polys}
For all $u,v\in S_n$, 
\[
Z^\square_{1_W,u}|_v = [\Theta_u]_v.
\]
\end{proposition}

\begin{proof}
Let  $[Z_u]$ be the element of $\mathcal{O}_q (T_{\operatorname{reg}})^W$ defined by $[Z_u]_v := Z^\square_{1_W,u}|_v$. We will prove the result by showing that $[Z_u] = [\Theta_u]$ for all $u\in S_n$.
The recursion in Proposition~\ref{prop:special-recursion} shows that when $us_i < u$, $[Z_{us_i}] = D_i[Z_u]$, so by Proposition~\ref{prop:D-of-Theta}, we only need to show that $[Z_{w_0}] = [\Theta_{w_0}]$.

By the base case of Proposition~\ref{prop:special-recursion}, with $\sigma=1_W$, the partition functions satisfy
\[
[Z_{w_0}]_v = Z^\square_{1_W,w_0}|_v = \begin{cases} \prod_{\alpha\in \Phi^+}(1-\bs{y}^{-\alpha}), &\text{if } v=w_0,\\ 0, & \text{otherwise},\end{cases}
\]
and on the other hand,
\[
[\Theta_{w_0}]_v.= \phi_v r_{w_0,v^{-1}} = \prod_{\substack{\alpha \in \Phi^+\\v^{- 1} (\alpha) \in \Phi^+}}(q \bs{y}^{- \alpha} - 1) 
	 \prod_{\substack{\alpha \in \Phi^+\\v^{- 1} (\alpha) \in \Phi^-}}(1 - \bs{y}^{- \alpha})
	 \begin{cases} 1, &\text{if } v=w_0,\\ 0, & \text{otherwise},\end{cases}
\]
In the case where $v=w_0$, the first product is empty and the second product is over all positive roots, so $Z^\square_{1_W,w_0}|_v = [\Theta_{w_0}]_v$ for all $v\in S_n$.
\end{proof}

\begin{corollary} \label{cor:KL-R}
For all $u,v\in S_n$,
\[
r_{u,v} = \phi_{v^{-1}}^{-1} Z^\square_{1_W,u^{-1}}|_{v^{-1}}.
\]
The Boltzmann weights for this lattice model are given in Figure~\ref{fig:r-poly-weights}.
\end{corollary}

\begin{proof}
The first statement follows from Proposition~\ref{prop:part-fun-r-polys}. From the Boltzmann weights in Figure~\ref{fig:modified_weights}, replace $y$ with $y_j$ and $x$ with $y_{v^{-1}(i)}$. These are the Boltzmann weights for the vertex in row $i$ and column $j$ of $Z^\square_{1_W,u^{-1}}|_{v^{-1}}$, and appear in Figure~\ref{fig:r-poly-weights}, with $\alpha := e_{v^{-1}(i)} - e_j$. 
\end{proof}

\begin{figure}[h]
\begin{center}
\begin{tabular}{|ccc|}
\hline
\hline
Case 1 & Case 2 & Case 3\\
\hline
\vertexnoxy{+}{\Sigma}{+}{\Sigma} &
\vertexnoxy{c}{\Sigma}{c}{\Sigma} &
\vertexnoxy{+}{\Sigma}{c}{\Sigma^-_c}
\\[4pt]
$(-q)^{\Sigma_{[1,r]}}$ & $q^{\Sigma_{(c,r]}} (-\bs{y}^\alpha + q^{\Sigma_c})$ & $(-q)^{\Sigma_{[1,c)}} q^{\Sigma_{(c,r]}}$
\\
\hline
Case 4 & Case 5 & Case 6\\
\hline
\vertexnoxy{c}{\Sigma}{+}{\Sigma^+_c} &
\vertexnoxy{c}{\Sigma}{d}{\Sigma^{+-}_{cd}} &
\vertexnoxy{d}{\Sigma}{c}{\Sigma^{+-}_{dc}}
\\[4pt]
$(q-1) \bs{y}^\alpha (-q)^{\Sigma_{(c,r]}}$ & $(q-1) \bs{y}^\alpha (-q)^{\Sigma_{(c,d)}} q^{\Sigma_{(d,r]}}$ & $(q-1) (-1)^{\Sigma_{(c,d)}} q^{\Sigma_{(d,r]}}$
\\
\hline
\end{tabular}
\end{center}
\caption{The Boltzmann weights for the vertex in row $i$ and column $j$ of the lattice model $Z^\square_{1_W,u^{-1}}|_{v^{-1}}$, where we set $\alpha = e_{v^{-1}(i)} - e_j$.}
\label{fig:r-poly-weights}
\end{figure}

\begin{remark}
By \cite[Theorem~2]{BumpNakasujiKL}, $r_{u,v}\to R_{u,v}$ under the limit sending $\bs{y}^\alpha\to\infty$ for all $\alpha\in\Phi^+$. One can see that $\phi_v\to (-1)^{\ell(w_0v)}$ for all $v\in S_n$, so by Corollary~\ref{cor:KL-R}, $R_{u,v}$ is equal to the limit of $Z^\square_{1_W,u^{-1}}|_{v^{-1}}$ as $\bs{y}^\alpha\to\infty$ for all $\alpha\in\Phi^+$. Therefore, one may interpret $R_{u,v}$ as a sum over all states which survive in the limit. However, the cancellation involved is nontrivial.
\end{remark}

\section{Special Cases\label{sec:special}}

In this section, we show how various transformations and special cases of our lattice model result in special functions 
from geometry and representation theory, thus providing new lattice model interpretations of these functions.

\subsection{\label{subsec:iwahori}Iwahori Whittaker functions}

If we take the $\mathbf{y}$ parameters to be zero, the models in this paper
specialize to the models in~{\cite{BBBGIwahori}}, which represent Iwahori
Whittaker functions. More precisely, the {\textit{Whittaker weights}} in
Figure~\ref{fig:whittaker_weights} specialize to the Boltzmann weights in
{\cite{BBBGIwahori}} when the column parameter (denoted $y$ in
Figure~\ref{fig:whittaker_weights}) is zero. It is shown there that the
partition functions of these models represent Iwahori Whittaker functions on
$G = \operatorname{GL} (n, F)$ where $F$ is a nonarchimedean local field.

\begin{remark}
  There are notational differences between this paper and
  {\cite{BBBGIwahori}}. In both papers, Iwahori Whittaker functions and their associated
  lattice model boundaries are indexed by a triple of data - a  partition and a pair of Weyl group elements. 
  Let $(\lambda; \sigma, w)$, with $\lambda$ a partition and $\sigma, w \in S_n$ be such 
  a triple in the notation of this paper, and let $(\mu; w_1, w_2)$ be a triple in the notation of
  \cite{BBBGIwahori}. To make the correspondence between the two notations giving rise to the
  same partition functions (and hence the same associated Whittaker functions), we set 
  $(\mu; w_1, w_2) = (\lambda - \rho; w^{-1} w_0, w_0 \sigma^{-1} w_0)$. 
\end{remark}

Before making a connection to Whittaker functions, let us relate the partition functions
of the two sets of Boltzmann weights presented earlier.
Let $\Omega_{\sigma, w}^{\lambda} (\bs{x} ; \bs{y})$ be the partition function
made with boundary conditions $B (\sigma, w, \lambda)$ as in
Definition~\ref{def:boundary} and Boltzmann weights as in
Figure~\ref{fig:whittaker_weights}.
We can utilize the close relationship between our two sets of Boltzmann weights in Figures~\ref{fig:whittaker_weights} and~\ref{fig:modified_weights} 
to relate their respective partition functions $\Omega_{\sigma, w}^\lambda$ and $Z_{\sigma,w}^\lambda$.

\begin{proposition} \label{prop:omegazident} For all permutations $\sigma, w$ and dominant weights $\lambda$,
\[\Omega_{\sigma, w}^\lambda(\bs{x}; \bs{y}) = \left(\prod_{j=1}^N y_j^{|\{i | \lambda_i > j\}| }\right) Z_{\sigma, w}^\lambda(-\bs{x}; \bs{y})\]
and thus by Theorem~\ref{thm:partition-function} evaluating $Z_{\sigma, w}^\lambda$,
$$\Omega_{\sigma, w}^\lambda(\bs{x}; \bs{y})
= t^{\ell(\sigma)} \tau_w^{-1} \tau_{w_0 \sigma^{-1}} \cdot \left[ \prod_{i=1}^n \prod_{j < \lambda_i} \left(x_i+y_j\right) \right]. $$
\end{proposition}
\begin{proof}
The first identity follows from the caption to Figure \ref{fig:modified_weights}. For the second identity, note that $\tau_i$ is invariant under the substitution $\bs{x}\mapsto -\bs{x}$. Using Theorem \ref{thm:partition-function}, we have
\begin{align*}
\Omega_{\sigma, w}^\lambda(\bs{x}; \bs{y})
&=  \left(\prod_{j=1}^N y_j^{|\{i | \lambda_i > j\}| }\right)  t^{\ell(\sigma)} \tau_w^{-1} \tau_{w_0 \sigma^{-1}} \cdot \left[ \prod_{i=1}^n \prod_{j < \lambda_i} \left(1 + \frac{x_i}{y_j}\right) \right]
\\&= t^{\ell(\sigma)} \tau_w^{-1} \tau_{w_0 \sigma^{-1}} \cdot \left[ \prod_{i=1}^n \prod_{j < \lambda_i} \left(x_i+y_j\right) \right]. \qedhere
\end{align*} 
\end{proof}

In particular, the prior result shows that in the specialization $(\bs{x}; \bs{y}) = (\bs{z}, \bs{0})$
$$ \Omega_{\sigma, w}^\lambda(\bs{z}; \bs{0}) = t^{\ell(\sigma)} \tau_w^{-1} 
\tau_{w_0 \sigma^{-1}} \bs{z}^\lambda $$ 
matching the divided difference operator description of Iwahori-Whittaker functions in Corollary 3.9~of \cite{BBBGIwahori} 
with our $t$ equal to their $v$. Note that we must shift all parts of the partition $\lambda$ by 1, as our first column 
in the lattice is labeled ``1" while the first column in~\cite{BBBGIwahori} is labeled ``0." Thus, comparing with \cite{BBBGIwahori}, 
we are immediately led to the following result:

\begin{proposition}
The standard Iwahori Whittaker function $ \phi^{\bs{z}}_w (g)$ evaluated at $g = \varpi^{- \lambda}
\sigma$ is (up to normalization) the partition function $\Omega_{w_0 \sigma^{- 1} w_0, w_0 w^{-
1}}^{\lambda + \rho} (\bs{z} ; \bs{0}).$
\end{proposition}

Several notations need to be briefly explained here. In the standard basis of 
Whittaker functions $\phi^{\bs{z}}_w$, the complex vector $\bs{z}$ refers to the
Langlands parameters of the associated principal series on $G$. We evaluate
at $g = \varpi^{-\lambda} \sigma$ where we may identify the Weyl group element $\sigma$ with an
element of $G$ as a permutation matrix. Moreover $\varpi^{-\lambda}$ denotes the torus element
with entry $\varpi^{-\lambda_i}$ in the $i$-th component. 
Furthermore in this description, it is assumed that $\lambda$ is \textit{$w_2$-almost
dominant} in the terminology of \cite{BBBGIwahori} Definition~3.4. We will not
repeat this definition here. Although $\lambda$ is not necessarily
dominant, this condition at least implies that $\lambda+\rho$ is dominant.

Note that in order to completely determine all values of Iwahori Whittaker functions, it
suffices to restrict $g$ to a set of double coset representatives for $N^-\backslash G/J$ 
where $N^-$ is the group of lower triangular unipotent matrices and $J$ is an Iwahori subgroup.
Moreover the Whittaker function vanishes on some double cosets. Taking both facts into account, 
it is shown in \cite{BBBGIwahori} that the set of $g$ described in the above proposition are exactly the double coset representatives
$g=\varpi^{-\lambda}\sigma$ for which $\phi_w(g)$ is nonzero.

Now we can combine the above with earlier results from Section~\ref{sec:motivic-chern}
on motivic Chern classes to connect these two families of special functions via their common
interpretations as a lattice model partition function.
To avoid complications let us assume that $\sigma = 1$, so we will be
comparing motivic Chern classes to Whittaker values on diagonal elements. We
will also replace $\lambda$ by $\lambda - \rho$. Thus $\phi_{w_0 w^{- 1}}
(\varpi^{- (\lambda - \rho)})$ is to be compared with $\Omega_{1_W,
w}^{\lambda} (\bs{z} ; \bs{0})$. These Whittaker values are equal to
nonsymmetric Macdonald polynomials as explained in~\cite{BBBGIwahori}.
Moreover, given the relation between $Z_{\sigma, w}^\lambda$ and
$\Omega_{\sigma,w}^\lambda$ above, we may immediately realize certain values of
Iwahori Whittaker functions as specializations of motivic Chern classes.

\begin{corollary} For any permutation $w$, the value of the Iwahori-Whittaker function
\[ \Omega_{1_W, w}^\rho(\bs{x}) = \left. (w_0\bs{y})^\rho  (-t)^{\ell(w_0)-\ell(w)} \MC^\vee(X(w_0w)^\circ)  \right|_{\bs{x}\to -\bs{x}; \bs{y}\to 0}. \]
\end{corollary}

\begin{proof} By Theorem \ref{thm:dual-motivic-part-fun},
\[ \left. Z_{1_W, w}^\rho(-\bs{x}; w_0\bs{y})\right|_{\bs{y}\to 0} = \left.  (-t)^{\ell(w_0)-\ell(w)} \MC^\vee(X(w_0w)^\circ)  \right|_{\bs{x}\to -\bs{x}; \bs{y}\to 0} \]
so the result follows immediately from Proposition~\ref{prop:omegazident}.
\end{proof}

In principle, one could make this connection simply at the level of the divided difference operators $\tau_w$ with sufficiently many clever substitutions and normalizations. A related connection along these lines was previously found in \cite{MihalceaSuWhittaker} (see Theorem~1.1 and the discussion following it, building upon the identity in Remark 5.4 of \cite{AMSSCasselman} in equivariant K-theory). But it is interesting that the agreement extends to the much finer level of individual admissible states of the lattice model, where local statistics have been previously shown to encode interesting geometric information (as in the data on Gr\"{o}bner degenerations encoded by the lattice models rendered as so-called ``pipe dreams'' in \cite{KnutsonMillerGrobner}).

\subsection{Factorial Schur functions\label{subsec:factorial}}

Next, we consider the Whittaker model under the process of \emph{color
merging}. This is a phenomenon in which a sum of partition functions
for a colored model can equal a partition function of a model with
fewer colors, or an uncolored model. It was called \emph{color blindness}
in~\cite{BorodinWheelerColored} and \emph{local lifting property} in
\cite{BumpNaprienko}. The same phenomenon may be seen in \cite{BBBGIwahori}
as Properties~A and~B (Figures 18 and 19).
In our case, the colored models we consider are related to uncolored
models studied by Bump, McNamara, and
Nakasuji \cite{BumpMcNamaraNakasujiFactorial}, and correspond to a Tokuyama-like
deformation of factorial Schur functions.

The Shintani-Casselman-Shalika formula for $GL(n)$ expresses the spherical Whittaker function as a
Schur function times a deformation of the Weyl denominator. On the other hand,
\cite{Tokuyama,HamelKing,BBBF} expressed this as the partition function of a solvable
lattice model. The spherical Whittaker function is expressed as a sum over the Weyl group of
Iwahori Whittaker functions by~\cite{BBBGIwahori}, and these also have representations
as partition functions of solvable lattice models -- colored models. Putting
all these pieces together, the Schur function times the deformed Weyl denominator
is expressed as a sum over the Weyl group of partition functions of colored models.
The following result is a factorial version of this fact.

\begin{theorem} \label{thm:refinement}
For all permutations $\sigma$ and dominant weights $\lambda$ of length $\leqslant r$,
\begin{equation} \label{eq:Whittaker-BMN-part-funs}
\sum_{w\in S_n} \Omega_{\sigma, w}^{\lambda+\rho}(\bs{x}; \bs{y})
= \left(\prod_{i<j} (x_i-tx_j)\right)s_\lambda(x|y),
\end{equation}
where $s_\lambda(x|y)$ is the factorial Schur function given by
\[
s_\lambda(x|y) = \frac{\det\left( (x_i-y_1)\cdots (x_i-y_{\lambda_j+j}) \right)_{1\leqslant i,j\leqslant n}}{\prod_{i<j} (x_i-x_j)}.
\]
\end{theorem}

Note in particular that since the right side is independent of $\sigma$, so is the left side.

The right side of \eqref{eq:Whittaker-BMN-part-funs} is the partition function
from \cite{BumpMcNamaraNakasujiFactorial}. Let
$\Omega^\lambda_{\text{BMN}}(\bs{x}; \bs{y})$ be the following partition
function:
\begin{itemize}
\item $r$ rows and $N$ columns, with row parameters $x_1,\ldots,x_r$ from top to bottom, and column parameters $y_1,\ldots,y_N$ from right to left.
\item Boltzmann weights as in \ref{fig:colorblind_weights}.
\item Along the top boundary, the boundary edge in columns $\lambda_i + r + 1 - i$ are $-$. All other top boundary edges are $+$. If $\ell(\lambda)<r$, we pad it with trailing $0$'s.
\item Along the right boundary, all boundary edges are $-$.
\item Along the bottom and left boundaries, all boundary edges are $+$.
\end{itemize}

\begin{figure}[h]
\begin{center}
\begin{tabular}{|cccccc|}
\hline
\hline
Case 1 & Case 2 & Case 3 & Case 4 & Case 5 & Case 6\\
\hline
\smallvertex{+}{+}{+}{+} &
\smallvertex{-}{-}{-}{-} &
\smallvertex{+}{-}{+}{-} &
\smallvertex{-}{+}{-}{+} &
\smallvertex{-}{+}{+}{-} &
\smallvertex{+}{-}{-}{+}
\\[4pt]
$1$ & 
$x+ty$ &
$-t$ &
$x+y$ &
$x(1-t)$ &
$1$ \\
\hline

\hline
\end{tabular}
\end{center}
	\caption{The $v_\Gamma$ weights from \cite{BumpMcNamaraNakasujiFactorial}, with $z$
	replaced by $x$, $\alpha$ replaced by $y$, and $t$ replaced by $-t$. Here, as
	is convention, a label of $-$ represents a (single) path, while $+$
	represents the absence of a path.}
\label{fig:colorblind_weights}
\end{figure}

By \cite[Theorem~1]{BumpMcNamaraNakasujiFactorial} (noting our substitutions),
\[
\Omega^\lambda_{\text{BMN}}(\bs{x}; \bs{y}) = \left(\prod_{i<j} (x_i-tx_j)\right)s_\lambda(x|y),
\]
so we need only show that
\[
\sum_{w\in S_n} \Omega_{\sigma, w}^{\lambda+\rho}(\bs{x}; \bs{y})
= \Omega^\lambda_{\text{BMN}}(\bs{x}; \bs{y}).
\]

 The left side is a sum over a collection of states; let $\mathfrak{S}$ denote this collection. A state is called \emph{strict} if no edge has multiple colors, otherwise it is called \emph{nonstrict}. Let $\mathfrak{S}_{\text{strict}}$ denote the states in $\mathfrak{S}$ which are strict, and let $\mathfrak{S}_{\text{ns}}$ denote the states which are nonstrict. Clearly,
 \[
 \sum_{w\in S_n} \Omega_{\sigma, w}^{\lambda+\rho}(\bs{x}; \bs{y}) = \sum_{\mathfrak{s}\in \mathfrak{S}_{\text{strict}}} \wt(\mathfrak{s}) + \sum_{\mathfrak{s}\in \mathfrak{S}_{\text{ns}}} \wt(\mathfrak{s}).
 \]

Theorem \ref{thm:refinement} then follows directly from the following Lemma.

\begin{lemma} \;
\begin{enumerate}
\item[(i)] $\sum_{\mathfrak{s}\in \mathfrak{S}_{\text{strict}}} \wt(\mathfrak{s}) = \Omega^\lambda_{\text{BMN}}(\bs{x}; \bs{y}).$
\item[(ii)] $\sum_{\mathfrak{s}\in \mathfrak{S}_{\text{ns}}} \wt(\mathfrak{s}) = 0.$
\end{enumerate}
\end{lemma}

\begin{proof}
The proof of this lemma is essentially the same as \cite[Lemma~8.5]{BBBGIwahori}. The key steps of that Lemma are the following properties:

\noindent\textbf{Property A:} For any $a,b,c,d\in\{0,1\}$ and for any sets $A,B\subseteq [1,r]$ such that $|A|=a, |B|=b$,

\begin{equation} \label{eq:color-merging}
	\wt\left(\vcenter{\hbox{\vertex{a}{b}{c}{d}}}\right)
	\quad=\sum_{|C|=c, |D|=d} \quad \wt\left(\vcenter{\hbox{\vertex{A}{B}{C}{D}}}\right),
\end{equation}
where the weights on the left side are taken from Figure \ref{fig:colorblind_weights}, while those on the right side are taken from Figure \ref{fig:modified_weights}.

\noindent\textbf{Property B:} For any set $\Sigma\subseteq [1,r]$, and any colors $c,d\in\Sigma$ such that $\Sigma$ contains no colors in between,
\begin{equation}
	\wt\left(\vcenter{\hbox{\vertex{+}{\Sigma}{c}{\Sigma^-_c}}}\right) \quad+\quad \wt\left(\vcenter{\hbox{\vertex{+}{\Sigma}{d}{\Sigma^-_d}}}\right) = 0.
\end{equation}
where the weights are taken from Figure \ref{fig:modified_weights}.

These properties are easily checked by direct inspection. The rest of the proof is unchanged from \cite{BBBGIwahori},  and we refer the reader that source.
\end{proof}

\subsection{Double Schubert and Grothendieck polynomials}

Consider the isobaric divided difference operators \[\overline{\partial}_i = \frac{1-x^{\alpha_i}s_i}{1-x^{\alpha_i}} = \frac{x_{i+1}-x_is_i}{x_{i+1}-x_i} = -\partial_i x_{i+1}, \qquad \text{where} \qquad \partial_i = \frac{1-s_i}{x_i-x_{i+1}}.\] These operators satisfy the braid relations, and so for any permutation $w \in S_n$, we can define \[\overline{\partial}_w = \overline{\partial}_{i_1} \overline{\partial}_{i_2}\cdots \overline{\partial}_{i_k}\] for any reduced expression $w = s_{i_1}\cdots s_{i_k}$, and $\overline{\partial}_w$ is independent of the choice of reduced expression.

The isobaric divided difference operators are used to define the \emph{double Grothendick (Laurent) polynomials} $\mathfrak{G}_w(\bs{x}; \bs{y})$, via the recursion \[\mathfrak{G}_{w_0}(\bs{x}; \bs{y}) = \prod_{i+j\leqslant n} \left(1 - \frac{x_i}{y_j}\right)\] and \[\mathfrak{G}_{ws_i}(\bs{x}; \bs{y}) = \overline{\partial}_i \mathfrak{G}_w(\bs{x}; \bs{y}), \qquad\qquad ws_i < w.\]

\begin{remark}
	There are several related definitions of double Grothendieck polynomials,
	which are usually related by simple transformations.
	This definition matches the one in \cite{KnutsonMillerGrobner}.
	The definition in \cite{KnutsonMillerSubword} coincides
	with $\mathfrak{G}_w(\bs{x}; \bs{y}^{-1})$.
	
\end{remark}

Consider now our operator $\tau_i$. After setting $t=0$, it becomes \begin{equation} \label{eq:tau-t=0} \tau_i|_{t=0} = \partial_i x_{i+1} = -\overline{\partial}_i.\end{equation} Let $\overline{Z}^\lambda_{\sigma,w}(\bs{x}; \bs{y})$ denote the partition function $Z^\lambda_{\sigma,w}(\bs{x}; \bs{y})_{t=0}$. We then have
\[\
\overline{Z}^\lambda_{\sigma,ws_i}(\bs{x}; \bs{y}) = -\overline{\partial}_i\overline{Z}^\lambda_{\sigma,w}(\bs{x}; \bs{y}), \qquad\qquad ws_i<w.
\]

We can evaluate $\overline{Z}^\lambda_{\sigma,w}(\bs{x}; \bs{y})$ using Theorem \ref{thm:partition-function}.

\begin{proposition}
Let $\lambda$ be a partition, and $\sigma, w\in S_n$. If $\lambda$ is nonstrict or $\sigma\ne 1_W$, then $\overline{Z}_{\sigma, w}^\lambda(\bs{x}; \bs{y})=0$. If $\lambda$ is strict, then
\[ \overline{Z}_{1_W, w}^\lambda(\bs{x}; \bs{y}) =  (-1)^{w^{-1}w_0}\overline{\partial}_{w^{-1}w_0} \cdot \left[ \prod_{i=1}^n \prod_{j < \lambda_i} \left(1 - \frac{x_i}{y_j}\right) \right].\]
\end{proposition}

\begin{proof}
By Theorem \ref{thm:partition-function}, the partition function $Z^\lambda_{\sigma,w}(\bs{x}; \bs{y})$ is a multiple of $t$ whenever $\sigma\ne 1_W$, so $\overline{Z}^\lambda_{\sigma,w}(\bs{x}; \bs{y}) = 0$ in this case. In particular, if $\lambda$ is nonstrict, then $w_0\lambda$ has a nontrivial stabilizer, so the maximal-length condition on $\sigma$ Theorem \ref{thm:partition-function} ensures that $\sigma\ne 0$ in that case too.

Now suppose that $\lambda$ is strict and that $\sigma=1_W$.  Setting $t=0$ in Theorem \ref{thm:partition-function}, we get
\[
\overline{Z}_{1_W, w}^\lambda(\bs{x}; \bs{y}) =  \tau_w^{-1}|_{t=0} \tau_{w_0}|_{t=0} \cdot \left[ \prod_{i=1}^n \prod_{j < \lambda_i} \left(1 - \frac{x_i}{y_j}\right) \right].
\]
Thus, we only need to show that $\tau_w^{-1}|_{t=0} \tau_{w_0}|_{t=0} = (-1)^{\ell(w^{-1}w_0)}\overline{\partial}_{w^{-1}w_0}$, and this follows from (\ref{eq:tau-t=0}) and the fact that since $w^{-1}w_0<w_0$, $\overline{\partial}_w^{-1} \overline{\partial}_{w_0} = \overline{\partial}_w^{-1} \overline{\partial}_w \overline{\partial}_{w^{-1}w_0} = \overline{\partial}_{w^{-1}w_0}$.
\end{proof}

Setting $\lambda=\rho$, we get the following corollary.

\begin{corollary} \label{cor:groth-polys-specialization}
\[ \overline{Z}_{1_W, w}^\square(\bs{x}; \bs{y}) =  (-1)^{\ell(w^{-1}w_0)}\mathfrak{G}_w(\bs{x}; \bs{y}).\]
\end{corollary}

Next, we recover the reduced pipe dream formula for Grothendieck polynomials of \cite{KnutsonMillerSubword} by considering the $t=0$ specialization of our model. A \emph{pipe dream} is a tiling of an $n\times n$ grid with the tiles
\raisebox{-1 ex}{\begin{tikzpicture}[scale=.5]
    \draw (0,0) -- (1,0);
    \draw (0,1) -- (1,1);
    \draw (0,0) -- (0,1);
    \draw (1,0) -- (1,1);
    \draw[rounded corners = 2mm, line width = .5mm] (1,.5) -- (.5,.5)-- (.5,1);
    \draw[rounded corners = 2mm, line width = .5mm] (.5,0) --(.5,.5)-- (0,.5);
\end{tikzpicture}}
and 
\raisebox{-1 ex}{\begin{tikzpicture}[scale=.5]
    \draw (0,0) -- (1,0);
    \draw (0,1) -- (1,1);
    \draw (0,0) -- (0,1);
    \draw (1,0) -- (1,1);
    \draw[rounded corners = 2mm, line width = .5mm] (0,.5)--(1,.5);
    \draw[rounded corners = 2mm, line width = .5mm] (.5,0) --(.5,1);
\end{tikzpicture}}, such that every 
\raisebox{-1 ex}{\begin{tikzpicture}[scale=.5]
    \draw (0,0) -- (1,0);
    \draw (0,1) -- (1,1);
    \draw (0,0) -- (0,1);
    \draw (1,0) -- (1,1);
    \draw[rounded corners = 2mm, line width = .5mm] (0,.5)--(1,.5);
    \draw[rounded corners = 2mm, line width = .5mm] (.5,0) --(.5,1);
\end{tikzpicture}}
appears above the main diagonal. These tiles form a collection of pipes originating on the top boundary of the grid and terminating along the right boundary. A pipe dream is said to be \emph{reduced} if no two pipes cross more than once.

A reduced pipe dream is associated to the permutation $w$ given by $w(i)=j$ if the pipe entering in the $i$th column (from the right) on the top boundary exits via the $j$th row (from the top) on the right boundary. A tile of the form
\raisebox{-1 ex}{\begin{tikzpicture}[scale=.5]
    \draw (0,0) -- (1,0);
    \draw (0,1) -- (1,1);
    \draw (0,0) -- (0,1);
    \draw (1,0) -- (1,1);
    \draw[rounded corners = 2mm, line width = .5mm] (0,.5)--(1,.5);
    \draw[rounded corners = 2mm, line width = .5mm] (.5,0) --(.5,1);
\end{tikzpicture}}
which appears in row $i$ and column $j$ has weight $1-\frac{x_i}{y_j}$, while a tile of the form
\raisebox{-1 ex}{\begin{tikzpicture}[scale=.5]
    \draw (0,0) -- (1,0);
    \draw (0,1) -- (1,1);
    \draw (0,0) -- (0,1);
    \draw (1,0) -- (1,1);
    \draw[rounded corners = 2mm, line width = .5mm] (0,.5)--(1,.5);
    \draw[rounded corners = 2mm, line width = .5mm] (.5,0) --(.5,1);
\end{tikzpicture}}
has weight $\frac{x_i}{y_j}$ if the pipes in question have another crossing north and/or west of the current tile, and $1$ if they do not.

\begin{corollary} \label{cor:pipe-dream} \cite[Corollary~5.5]{KnutsonMillerSubword}
$\mathfrak{G}_w(\bs{x}; \bs{y})$ is the sum of the weights of all reduced pipe dreams with permutation $w$.
\end{corollary}

\begin{proof}
Consider the partition function $\overline{Z}_{1_W, w}^\square(\bs{x}; \bs{y})$. Specialize $t=0$ to the Boltzmann weights in Figure \ref{fig:modified_weights}. The first thing to notice is that many of the weights become zero. In particular, the weight of Case 1 is zero whenever $|\Sigma|\geqslant 1$, the weight of Case 3 is zero whenever $|\Sigma|\geqslant 2$, while the weights of Cases 2, 4, 5, and 6 are zero whenever $\Sigma$ contains a color larger than the color on either horizontal edge, unless it equals the color of the other horizontal edge.

Therefore, the only vertices with shared edges which survive are Case 4 or Case 6, when every color in $\Sigma$ is less than $c$, or Case 5, when every color in $\Sigma^-_d$ is less than $c$. In particular, paths moving down and right can only combine, not separate, since Case 3 has weight zero unless $\Sigma=\{c\}$. This means that no edge in the model can have multiple colors unless there is an edge with multiple colors on the bottom boundary of the model, which is forbidden by our boundary conditions. Therefore, we may disallow multiple colors on any edge. The resulting Boltzmann weights equal those in \ref{fig:pipe-dream-weights}, and so $\overline{Z}_{1_W, w}^\square(\bs{x}; \bs{y})$ is the partition function of the lattice model with these weights and boundary conditions $B(1, w, \rho)$.

We claim that states of this lattice model biject with reduced pipe dreams. For the vertices in Figure~\ref{fig:pipe-dream-weights}, replace every Case 2 vertex with a 
\raisebox{-1 ex}{\begin{tikzpicture}[scale=.5]
    \draw (0,0) -- (1,0);
    \draw (0,1) -- (1,1);
    \draw (0,0) -- (0,1);
    \draw (1,0) -- (1,1);
    \draw[rounded corners = 2mm, line width = .5mm] (0,.5)--(1,.5);
    \draw[rounded corners = 2mm, line width = .5mm] (.5,0) --(.5,1);
\end{tikzpicture}}
tile, and every other vertex with a 
\raisebox{-1 ex}{\begin{tikzpicture}[scale=.5]
    \draw (0,0) -- (1,0);
    \draw (0,1) -- (1,1);
    \draw (0,0) -- (0,1);
    \draw (1,0) -- (1,1);
    \draw[rounded corners = 2mm, line width = .5mm] (1,.5) -- (.5,.5)-- (.5,1);
    \draw[rounded corners = 2mm, line width = .5mm] (.5,0) --(.5,.5)-- (0,.5);
\end{tikzpicture}}
tile. Since Case 2 only appears with a path on the horizontal edge, Case 4 cannot appear, every vertex below the main diagonal must be Case 1, and every vertex on the main diagonal must be Case 3, so every 
\raisebox{-1 ex}{\begin{tikzpicture}[scale=.5]
    \draw (0,0) -- (1,0);
    \draw (0,1) -- (1,1);
    \draw (0,0) -- (0,1);
    \draw (1,0) -- (1,1);
    \draw[rounded corners = 2mm, line width = .5mm] (1,.5) -- (.5,.5)-- (.5,1);
    \draw[rounded corners = 2mm, line width = .5mm] (.5,0) --(.5,.5)-- (0,.5);
\end{tikzpicture}}
tile appears above the main diagonal, and we indeed have a pipe dream. Furthermore, since in Case 2 $c<d$, any pair of pipes can only cross once, and therefore the pipe dream is reduced. This map is a bijection since a state of this model is determined by the location of Case 2 vertices.

Now for the weights. Notice that the vertex weights in Figure \ref{fig:pipe-dream-weights} match the weights of the corresponding tiles, except that Cases 4, 5, and 6 are negated. As a consequence of the above discussion, Cases 2, 5, and 6 appear a total of $\ell(w_0)$ times in any state of the model (and Case 4 does not appear). $\ell(w)$ of these vertices are Case 2, so the remaining $\ell(w^{-1}w_0)$ are Cases 5 and 6. Therefore, the weight of any state is $(-1)^{\ell(w^{-1}w_0)}$ times the weight of the resulting pipe dream, and by Corollary~\ref{cor:groth-polys-specialization},
$\mathfrak{G}_w(\bs{x}; \bs{y}) = (-1)^{\ell(w^{-1}w_0)}\overline{Z}_{1_W, w}^\square(\bs{x}; \bs{y})$ equals the sum of the  weights of all reduced pipe dreams with permutation $w$.
\end{proof}

\begin{figure}[h] 
\begin{center}
\scalebox{0.9}{
\begin{tabular}{|cccccc|}
\hline
\hline
Case 1 & Case 2 & Case 3 & Case 4 & Case 5 & Case 6\\
\hline
\vertex{+}{+}{+}{+} &
\vertex{d}{c}{d}{c} &
\vertex{+}{c}{c}{+} &
\vertex{c}{+}{+}{c} &
\vertex{c}{d}{d}{c} &
\vertex{d}{c}{c}{d}
\\[4pt]
$1$ & 
$\delta_{c<d} - xy^{-1}$ &
$1$ &
$-x y^{-1}$ &
$-x y^{-1}$ &
$-1$ \\
\hline

\hline
\end{tabular}}
\end{center}
\caption{The Boltzmann weights from Figure \ref{fig:modified_weights} under the specialization $t=0$, with the restriction that multiple colors cannot share an edge. Here, $c$ and $d$ are colors with $c\leqslant d$. In Case 2 only, $c$ may be $+$, which is considered smaller than all colors for the purpose of this weight.. All vertex configurations which do not appear in this table have weight zero. States of this model with boundary conditions $B(1,w,\rho)$ biject with reduced pipe dreams, and the weights match up to a sign.}
\label{fig:pipe-dream-weights}
\end{figure}

\begin{remark}
We can use the weights in Figure~\ref{fig:pipe-dream-weights} to obtain the
	double $\beta$-Grothendieck polynomials of Fomin and Kirillov
	\cite{FominKirillovFPSAC}, and therefore Schubert polynomials by
	setting $\beta=0$. More precisely, multiply the weights in Cases 4, 5, and 6
	by $-1$ (via a change of basis in the matrix of Boltzmann weights that does
	not alter the solvability of the model) as in the proof of
	Corollary~\ref{cor:pipe-dream}, and then make the substitution $x\mapsto
	1+\beta x, y^{-1}\mapsto 1+\beta y$. The weight in Case 2 when $c<d$ becomes
	$\beta(x+y+\beta xy)$, and the weight in Case 5 becomes $1+\beta(x+y+\beta
	xy)$. Comparing with \cite[Figure~3]{FrozenPipes} with $q=0$, we see that the
	resulting partition function is $(-\beta)^{\ell(w)}$ times the double
	$\beta$-Grothendieck polynomial.
\end{remark}

\begin{remark}
A similar argument to Corollary~\ref{cor:pipe-dream} shows that $\overline{Z}_{1_W, w}^\lambda(\bs{x}; \bs{y})$, is (up to a sign) a Grothendieck polynomial as well. If $\lambda$ is a strict partition with $n$ parts and $w$ is a permutation of $1,\ldots, n$, let $v(w,\lambda)$ be the permutation of $1,\ldots,\lambda_1$ where 
\[
v(j) = \begin{cases} w(i), & \text{if } j = \lambda_i, \\ n + i, & \text{if $j$ is the $i$th smallest non-part of $\lambda$}.\end{cases}
\]
Applying the bijection in the proof of Corollary~\ref{cor:pipe-dream} to states of $\overline{Z}_{1_W, w}^\lambda(\bs{x}; \bs{y})$, we obtain the formula
\[
\overline{Z}_{1_W, w}^\lambda(\bs{x}; \bs{y}) = (-1)^{\binom{\lambda_1+1}{2} - \ell(v(w,\lambda))} \mathfrak{G}_{v(w,\lambda)}(\bs{x}; \bs{y}).
\]
See \cite[Figure~9]{Zinn-Justin-LRSchur} for a similar argument for Grassmannian permutations.

It is interesting that in the $t=0$ case, varying $\lambda$ doesn't lead to any new partition functions, while in the case of general $t$, the functions can be very different.
\end{remark}

\bibliographystyle{habbrv}
\bibliography{schubert-lattice}

\end{document}